\documentclass[12pt,final]{amsart}
\usepackage[utf8]{inputenc}
\usepackage[T1]{fontenc}
\usepackage[letterpaper,centering]{geometry}
\usepackage{lmodern}
\usepackage{mathrsfs}
\usepackage{amsmath,amssymb,amsfonts}
\usepackage[unicode,pdfborder={0 0 0},final]{hyperref}
\usepackage{enumerate}
\usepackage{multicol}

\usepackage[all]{xy}
{\setbox0\hbox{$ $}}\fontdimen16\textfont2=\fontdimen17\textfont2
\entrymodifiers={+!!<0pt,\the\fontdimen22\textfont2>}
\SelectTips{cm}{12}

\usepackage{xr}
\theoremstyle{plain}
\newtheorem{thm}{Theorem}[section]

\newtheorem{lem}[thm]{Lemma}

\newtheorem{question}[thm]{Question}
\newtheorem{cor}[thm]{Corollary}
\newtheorem{prop}[thm]{Proposition}
\theoremstyle{definition}
\newtheorem{defn}[thm]{Definition}
\newtheorem{rmk}[thm]{Remark}
\newtheorem{rmks}[thm]{Remarks}
\newtheorem{example}[thm]{Example}

\newtheorem{Step}{Step}
\numberwithin{equation}{section}

\def\owrepositiontagaux#1#2#3{#2\hbox to 1pt{\hss$\mathsurround=0pt#1{#3}$\hss}}

\newcommand{\alg}{{\mathrm{alg}}}
\newcommand{\cl}{{\mathrm{cl}}}
\newcommand{\cx}{{\mathrm{cx}}}
\newcommand{\rs}{{\mathrm{r}}}
\newcommand{\gr}{{\mathrm{gr}}}
\newcommand{\Gm}{\mathbf{G}_\mathrm{m}}
\newcommand{\sB}{{\mathscr B}}
\newcommand{\sH}{{\mathscr H}}
\newcommand{\sF}{{\mathscr F}}
\newcommand{\sG}{{\mathscr G}}
\newcommand{\sI}{{\mathscr I}}
\newcommand{\sO}{{\mathscr O}}
\newcommand{\sX}{{\mathscr X}}
\newcommand{\A}{{\mathbf A}}
\renewcommand{\C}{{\mathbf C}}
\renewcommand{\P}{{\mathbf P}}
\newcommand{\Q}{{\mathbf Q}}
\newcommand{\R}{{\mathbf R}}
\newcommand{\SB}{{\mathbf{SB}}}
\newcommand{\Z}{{\mathbf Z}}
\newcommand{\Zl}{{\Z_{\ell}}}
\newcommand{\CH}{\mathrm{CH}}
\newcommand{\Hdg}{\mathrm{Hdg}}
\newcommand{\Gal}{\mathrm{Gal}}
\newcommand{\nr}{\mathrm{nr}}
\newcommand{\KS}{\textsc{KS}}
\newcommand{\Pic}{\mathrm{Pic}}
\newcommand{\Br}{\mathrm{Br}}
\newcommand{\Tr}{\mathrm{Tr}}
\newcommand{\Id}{\mathrm{Id}}
\newcommand{\Trf}{\mathrm{Tr}_f}
\newcommand{\Spec}{\mathrm{Spec}}
\renewcommand{\phi}{\varphi}
\renewcommand{\emptyset}{\varnothing}
\newcommand{\red}{{\mathrm{red}}}
\newcommand{\sm}{\mathrm{sm}}
\newcommand{\Ker}{{\mathrm{Ker}}}
\newcommand{\Ima}{{\mathrm{Im}}}
\newcommand{\Hom}{{\mathrm{Hom}}}
\newcommand{\Ext}{{\mathrm{Ext}}}
\newcommand{\et}{\text{ét}}
\newcommand{\torsion}{(\mathrm{torsion})}
\newcommand{\tors}{{\mathrm{tors}}}
\newcommand{\elw}{\mathrm{ind}}
\newcommand{\ab}{\mathrm{ab}}
\newcommand{\van}{\mathrm{van}}
\newcommand{\Sing}{\mathrm{Sing}}
\newcommand{\bS}{{\mathbf S}}
\newcommand{\divi}{\mathrm{div}}
\newcommand{\ci}{\mathscr{C}^{\infty}}
\newcommand{\RR}{{\mathrm{R}}}
\advance\textheight .22mm
\advance\topmargin -.11mm

\date{January 2nd, 2018; revised on October 19th, 2019}
\title[On the integral Hodge conjecture for real varieties, II]{On the integral Hodge conjecture for\\real varieties, II}

\author{Olivier Benoist}
\address{Institut de Recherche Math\'ematique Avanc\'ee,
UMR 7501, Universit\'e de Strasbourg et CNRS,
7 rue Ren\'e Descartes,
67000 Strasbourg, FRANCE}
\email{olivier.benoist@unistra.fr}

\author{Olivier Wittenberg}
\address{D\'epartement de math\'ematiques et applications, \'Ecole normale sup\'erieure, 45~rue d'Ulm, 75230 Paris Cedex 05, France}
\email{wittenberg@dma.ens.fr}

\begin{document}

\begin{abstract}
We establish the real integral Hodge conjecture for $1$\nobreakdash-cycles on various
classes of uniruled threefolds (conic bundles, Fano threefolds
with no real point, some del Pezzo fibrations)
and on conic bundles over higher-dimensional bases which themselves satisfy the real integral Hodge conjecture for $1$\nobreakdash-cycles.
In addition, we show that
 rationally connected threefolds over non-archimedean real closed fields
do not satisfy the real integral Hodge conjecture in general
and that over such fields,
Br\"ocker's EPT theorem remains true for simply connected surfaces of geometric genus zero
but fails for some~$K3$ surfaces.
\end{abstract}

\maketitle
\setcounter{section}{5}

\section*{Introduction}
{\renewcommand*{\thethm}{\Alph{thm}}\setcounter{thm}{3}

This article forms the second part of a work started in~\cite{BW1},
which centers around the formulation and the study of
a ``real integral Hodge conjecture'' for real algebraic varieties.
The first part was devoted to the definition of this property,
to examples,
and to the connections, in the case of $1$\nobreakdash-cycles,
with other topics in the theory of real algebraic cycles
(existence of curves of even genus, image of the Borel--Haefliger cycle class
map, unramified cohomology).
In this second part,
we address the question of the validity of the real integral
Hodge conjecture for $1$\nobreakdash-cycles on
specific classes of varieties.

Let us first recall the set-up put forward in \emph{op.\ cit}.
We denote by~$X$ a smooth and proper algebraic variety of dimension~$d$
over a real closed field~$R$.
Letting $C=R(\sqrt{-1}\mkern2mu)$ and
$G=\Gal(C/R)$,
one defines
 the \emph{equivariant cycle class map}
\begin{align*}
\cl:\CH_1(X) \to H^{2d-2}_G(X(C),\Z(d-1))
\end{align*}
(see \cite[\textsection\ref*{BW1-subsubsec:eqcl}]{BW1}).
It takes its values in the equivariant semi-algebraic cohomology
of the semi-algebraic space~$X(C)$, which, when $R$ is the field~$\R$ of real numbers,
is the same as the equivariant Betti cohomology of the complex manifold $X(\C)$.
The image of the equivariant cycle class map is contained in
the subgroup
\begin{align*}
\Hdg^{2d-2}_G(X(C),\Z(d-1))_0 \subseteq H^{2d-2}_G(X(C),\Z(d-1))
\end{align*}
of those classes that satisfy
a topological constraint, determined by the image of the class
in the group $H^{2d-2}_G(X(R),\Z(d-1))$, and a Hodge-theoretic constraint,
which is classical when $C=\C$ and which is trivial when $H^2(X,\sO_X)=0$
(see \cite[\textsection\ref*{BW1-par:realIHC}]{BW1}).
We consider the \emph{real integral Hodge conjecture for $1$\nobreakdash-cycles on~$X$}
only when $R=\R$ or $H^2(X,\sO_X)=0$; it
is the assertion that every element of
$\Hdg^{2d-2}_G(X(C),\Z(d-1))_0$
belongs to the image of the equivariant cycle class map.

Examples and counterexamples to the real integral Hodge conjecture for $1$\nobreakdash-cycles are given in \cite[\textsection\ref*{BW1-sec:realIHC},
\textsection\ref*{BW1-section:examples}]{BW1}. Its connections with
curves of even genus and with the Borel--Haefliger cycle class map
are discussed in \cite[\textsection\ref*{BW1-sec:onecycles}]{BW1}.
As it turns out, the point of view we adopt
leads to a systematic explanation for
all known examples of smooth proper real algebraic varieties~$X$ that
satisfy $H_1^\alg(X(\R),\Z/2\Z)\neq H_1(X(\R),\Z/2\Z)$
or that fail to contain a curve of
even genus
(see \cite[Remark~\ref*{BW1-rk:covers hknotalg}~(i)
and \textsection\ref*{BW1-section:examples}]{BW1}).

By analogy with Voisin's results \cite{voisinthreefolds} over the complex numbers,
we were led,
in \cite[Question~\ref*{BW1-mainquestion}]{BW1},
to raise the question
of the validity of the real integral Hodge conjecture, over~$\R$,
for $1$\nobreakdash-cycles on uniruled threefolds, on Calabi--Yau threefolds and on rationally connected varieties.

Motivated by this question, we first prove, in this article, the following
three theorems, which provide some evidence towards a positive
answer.  We then proceed to discuss the real integral Hodge conjecture
over non-archimedean real closed fields.

\begin{thm}[see Theorem~\ref{thm:fibres en coniques}]
\label{th:D}
Let $X\to B$ be a morphism between smooth, proper, irreducible varieties over~$\R$,
whose generic fibre is a conic.
Assume that~$B$
satisfies the real integral Hodge conjecture for $1$\nobreakdash-cycles. (Such is the case,
for instance, if~$B$ is a surface.)
Then~$X$ satisfies the real integral Hodge conjecture for $1$\nobreakdash-cycles.
\end{thm}

\begin{thm}[see Theorem~\ref{thm:solides fano}]
\label{th:E}
Let~$X$ be a smooth Fano threefold over~$\R$. Assume that $X(\R)=\emptyset$. Then~$X$ satisfies
the real integral Hodge conjecture.
\end{thm}

\begin{thm}[see Theorem~\ref{thm:solides fibres en del Pezzo}]
\label{th:F}
Let $X\to B$ be a morphism between smooth, proper, irreducible varieties
over a real closed field~$R$, whose generic fibre is a del Pezzo surface.
Let $\delta \in \{1,\dots,9\}$ denote the degree of this del Pezzo surface
and assume that~$B$ is a curve.
Then~$X$ satisfies the real integral Hodge conjecture under any of the
following conditions:
\begin{enumerate}[(i)]
\item $\delta\geq 5$;
\item $X(R)=\emptyset$ and $B(R)\neq\emptyset$;
\item $\delta\in\{1,3\}$ and the real locus of each smooth real fiber of~$f$ has exactly
 one connected component;
\item $R=\R$ and $\delta=3$.
\end{enumerate}
\end{thm}

The classes of varieties considered in these three theorems correspond to the
three possible outputs of the minimal model programme applied to a uniruled
threefold: a Fano threefold, a fibration into del Pezzo surfaces over
a curve, a conic bundle over a surface.
We note, however, that Theorem~\ref{th:D} allows conic bundles over a base of
arbitrary dimension (which creates significant additional difficulties
in its proof).

As explained in the introduction of~\cite{BW1} and in
\textsection\textsection\ref{sec:conicbundles}--\ref{sec:dPfibrations}
below, these theorems also have concrete implications for the study of the
group $H_1^\alg(X(\R),\Z/2\Z)$, for the
problem of~$\ci$ approximation of loops in~$X(\R)$ by algebraic
curves, for the existence of geometrically irreducible curves of
even geometric genus on~$X$ and for the study of the third unramified cohomology group
of~$X$ with~$\Q/\Z$ coefficients,
when~$X$ belongs to one of the families of varieties which appear in their statements (see
Corollary~\ref{coro:solides fibres en coniques},
Corollary~\ref{cor:approxconiques},
Corollary~\ref{cor:elw solides fano},
Proposition~\ref{prop:application of ihc for dp fibration}).

We devote each of
\textsection\ref{sec:conicbundles},
\textsection\ref{sec:Fano},
\textsection\ref{sec:dPfibrations}
to proving one of the above theorems.
The methods are varied.
Theorem~\ref{th:E} is established, for threefolds~$X$ such that~$-K_X$ is very ample,
by adapting to~$\R$
the strategy employed by Voisin~\cite{voisinthreefolds} over~$\C$:
we exhibit real points of the Noether--Lefschetz locus for anticanonical sections of~$X$.
The proofs of Theorems~\ref{th:D} and~\ref{th:F} exploit the geometry of the fibration structure.
Theorem~\ref{th:F} relies on a study of the homology of the singular
fibres, adapting to real closed fields an argument contained in~\cite{ewzcl} over separably
closed fields,
while the proof of Theorem~\ref{th:D}
applies the Stone--Weierstrass theorem in many ways,
through
various results of real algebraic geometry due to Akbulut and King, to Bröcker,
and to Ischebeck and Schülting.

Let us now briefly describe the contents of~\textsection\ref{nonarchimedeansection},
concerned with the phenomena that are specific to
non-archimedean real closed fields and with their implications for the study of
the real integral Hodge conjecture over~$\R$.

For smooth, projective and irreducible varieties~$X$ of dimension~$d\geq 1$
such that $H^2(X,\sO_X)=0$,
allowing an arbitrary real closed field as the ground field amounts to strengthening
the real integral Hodge conjecture over~$\R$ by requiring that the group
$H^{2d-2}_G(X(\C),\Z(d-1))_0$ be generated by the equivariant cycle classes of
curves on~$X$ whose degrees
are bounded independently of~$X$ whenever~$X$ varies in a bounded family of
real varieties and is defined over~$\R$ (see~\textsection\ref{subsec:curves of bounded degree}).

The proof of Theorem~\ref{th:E} heavily relies on Hodge theory, which forces us to restrict
its statement to~$\R$.
We do not know whether it remains true over non-archimedean real closed fields.
Theorem~\ref{th:D}, on the other hand, fails over such fields,
as we show in~\textsection\ref{sss:RCR}
by reinterpreting an example of Ducros~\cite{ducros}:
there exist a real closed field~$R$ and a conic bundle threefold~$X$
over~$\P^1_R \times \P^1_R$ failing the real integral Hodge conjecture.
We refer to~\textsection\ref{sss:discussion IHCR}
for a detailed discussion of the consequences of this counterexample.

The proof of Theorem~\ref{th:D} depends on the ground field being archimedean
through three uses of the Stone--Weierstrass theorem:
one to apply results of Akbulut and King~\cite{AKhomology}, \cite{AkbulutKing}
to approximate $\ci$ loops in~$X(\R)$ by real loci of algebraic curves and~$\ci$ surfaces
in~$B(\R)$
by connected components of real loci of algebraic surfaces; one to apply, to~$B$,
Bröcker's
EPT theorem~\cite{brocker} if~$B$ is a surface,
or its higher-dimensional version due to Ischebeck and Schülting~\cite{IS} in general;
and one to separate the connected components of the real loci of algebraic
surfaces contained in~$B$ by the signs of rational functions on these surfaces.

The first of these uses becomes irrelevant when $X(\R)=\emptyset$ and~$B$
is a surface.
Under these assumptions,
as we explain in~\textsection\ref{subsubsec:cokcycleclassmap},
the proof of Theorem~\ref{th:D} goes through over an arbitrary real closed field
as soon as~$B$ satisfies the conclusion of the EPT theorem
and possesses enough rational functions to separate the connected components of its real locus---two
properties that we refer to as the ``EPT'' and the ``signs'' properties
(see Proposition~\ref{prop:cokersstorsion}).

This process can be reversed: if~$B$ is a surface over a real closed field~$R$
and~$\Gamma$ denotes the anisotropic conic over~$R$, the real integral Hodge conjecture
for the trivial conic bundle $X = B \times \Gamma$ implies, in turn, the EPT and signs
properties for~$B$ (see Proposition~\ref{prop:IHCconicbundlesnonarch}).
We use this criterion to deduce the validity of the EPT and signs properties,
over arbitrary real closed fields,
for surfaces 
subject to the assumptions $H^2(B,\sO_B)=0$ and $\Pic(B_C)[2]=0$
(see Theorem~\ref{thm:signsEPTsurfaces}).

The assumption $H^2(B,\sO_B)=0$ is essential:
the signs property is known to fail
for some~$K3$ surfaces over non-archimedean real closed fields.
We complete the picture by showing that over non-archimedean real closed fields,
the EPT property can fail as well
(see Proposition~\ref{prop:cexEPT}; the counterexample is a~$K3$ surface),
thus answering a question raised in \cite{IS} and again in \cite{scheidererpurity}. We also verify that the hypothesis $\Pic(B_C)[2]=0$ cannot be dispensed with, by constructing a bielliptic surface failing the signs property (see Proposition \ref{prop:cexsigns}).

Additional positive and negative results are contained in~\textsection\ref{nonarchimedeansection}:
counterexamples to the real integral Hodge conjecture over a non-archimedean real closed field
for simply connected Calabi--Yau threefolds with and without real points (in~\textsection\ref{sss:CYR})
and for hypersurfaces of degree~$8$ in~$\P^4$ with no real point (in~\textsection\ref{subsec:hypersurfaces in P4}), and the validity of the real integral Hodge conjecture,
over arbitrary real closed fields, for smooth cubic hypersurfaces of dimension $\geq 3$ (see Theorem~\ref{thm:cubiques}).

\bigskip
\emph{Notation.}
Given a real closed field~$R$, we let $C=R(\sqrt{-1}\mkern2mu)$
and $G=\Gal(C/R)$.
We denote by~$\R$ the field of real numbers.
We recall from~\cite{BW1}
that \emph{variety} (over~$R$)
is understood as \emph{separated scheme of finite type} (over~$R$).
We refer to~\cite[\textsection\ref*{BW1-sec:cohomology}]{BW1} for generalities concerning the cohomology of varieties over $R$.
For a smooth and proper variety~$X$ of pure dimension~$d$ over~$R$,
we set $H_i(X(R),\Z/2\Z)=H^{d-i}(X(R),\Z/2\Z)$
and write $H_i^\alg(X(R),\Z/2\Z)=\cl_R(\CH_i(X))$
and $H^i_\alg(X(R),\Z/2\Z)=\cl_R(\CH^i(X))$,
where~$\cl_R$ denotes the Borel--Haefliger cycle class maps
(see~\cite[\textsection\ref*{BW1-subsubsec:borel-haefliger}]{BW1}).
}

\bigskip
\emph{Acknowledgements.}
We are grateful to Jacek Bochnak and Wojciech Kucharz
for calling our attention to an oversight in the argument
we had given for Theorem~\ref{thm:AK} in a previous version of this article, to Jean-Louis Colliot-Th\'el\`ene for pointing out that our proof of Theorem~\ref{thm:cubiques} for three-dimensional cubics works in higher dimension as well, and to the referee for their useful suggestions.

\section{Conic bundles}
\label{sec:conicbundles}

\subsection{The real integral Hodge conjecture for conic bundles}

Our goal in \S\ref{sec:conicbundles} is to prove the following theorem.

\begin{thm}
\label{thm:fibres en coniques}
Let $f:X\to B$ be a morphism of smooth proper connected varieties over $\R$ whose generic fiber is a conic. If $B$ satisfies the real integral Hodge conjecture for $1$-cycles, then so does $X$.
\end{thm}

This theorem is already non-trivial when $B$ is a surface.
 In this case, in view of the real Lefschetz $(1,1)$ theorem \cite[Proposition~\ref*{BW1-prop:real(1,1)}]{BW1} and of the reformulation of the real integral Hodge conjecture for $1$-cycles on $X$ in terms of unramified cohomology \cite[Remark~\ref*{BW1-rks:ch0 supported on surface}~(iii)]{BW1}, we get:

\begin{cor}
\label{coro:solides fibres en coniques}
Let $f:X\to B$ be a morphism of smooth proper connected varieties over $\R$ whose generic fiber is a conic. If $B$ is a surface, the real integral Hodge conjecture holds for $1$-cycles on $X$ and the unramified cohomology group $H^3_{\nr}(X,\Q/\Z(2))_0$ vanishes.
\end{cor}

\begin{rmk}
We show in Example \ref{ex:cexRC} below that an analogous statement over an arbitrary real closed field $R$ fails in general, even when $B$ is an $R$-rational surface. This is the reason why we only work over $\R$ in this section.
We explain geometric consequences of this failure in Example \ref{ex:implicationsDucros} below.

 The reason why the proof does not work over arbitrary real closed fields is that we extensively use consequences of the Stone--Weierstrass approximation theorem. In Proposition \ref{prop:cokersstorsion}, we  explain what remains of the proof over  general real closed fields when $B$ is a surface and~$X$ has no real point.
\end{rmk}

\begin{rmk}
The last statement of Corollary \ref{coro:solides fibres en coniques} is an analogue over $\R$ of a theorem of Parimala and Suresh over finite fields  \cite[Theorem 1.3]{parimalasuresh}.

The proof of Parimala and Suresh is very different from ours: they use the fact that the relevant unramified cohomology classes of $X$ come from cohomology classes of the function field of $B$, and study in detail their ramification over $B$.
It is possible to give a proof of Corollary \ref{coro:solides fibres en coniques} along these lines. 
The difficulties that arise, in this alternative proof, in the analysis of the ramification, are identical to those encountered in the proof we give below. This strategy is however not applicable when the base has dimension $\geq 3$.
\end{rmk}

Combining Theorem~\ref{thm:fibres en coniques} with the duality theorem \cite[Theorem~\ref*{BW1-th:nohodgetheoreticob}]{BW1}
and with an approximation result to be discussed next, namely Theorem~\ref{thm:AK} below,
yields the following corollary.

\begin{cor}
\label{cor:approxconiques}
Let $f:X\to B$ be a morphism of smooth projective connected varieties over $\R$ whose generic fiber is a conic.
Suppose that $H^2(B,\sO_B)=0$ and $\Pic(B_{\C})[2]=0$ and that $B$ satisfies the real integral Hodge conjecture for $1$-cycles.

Then $H_1(X(\R),\Z/2\Z)=H_1^\alg(X(\R),\Z/2\Z)$ and all one-dimensional compact $\ci$ submanifolds of~$X(\R)$
have an algebraic approximation in $X(\R)$ (in the sense of
\cite[Definition 12.4.10]{bcr},
recalled in Definition~\ref{def:approximation} below).
\end{cor}

Corollary~\ref{cor:approxconiques} is interesting already when $B$ is a geometrically rational surface. Its hypotheses are then satisfied by \cite[Proposition~\ref*{BW1-prop:real(1,1)}]{BW1}.

\begin{proof}[Proof of Corollary~\ref{cor:approxconiques}]
By Lemma~\ref{lem:Hi0 Pictors} below, we have $H^2(X,\sO_X)=\Pic(X_{\C})[2]=0$. 
We deduce from \cite[Theorem~\ref*{BW1-th:nohodgetheoreticob}]{BW1} and Theorem \ref{thm:fibres en coniques} that $H_1(X(\R),\Z/2\Z)=H_1^\alg(X(\R),\Z/2\Z)$. The second statement follows from Theorem~\ref{thm:AK} below.
\end{proof}

\newcommand{\citekollarpictors}{Koll\'ar~\cite{kollartorsionfree,kollarshaf}}
\begin{lem}[\citekollarpictors]
\label{lem:Hi0 Pictors}
Let $f:X\to B$ be a morphism of smooth, projective, connected varieties over an algebraically
closed field of characteristic~$0$,
whose geometric
generic fiber~$F$
is irreducible.  Let~$\ell$ be a prime number.
Assume that $\Pic(F)[\ell]=0$ and $H^i(F,\sO_F)=0$ for $i>0$.
Then the pull-back maps $\Pic(B)[\ell^\infty]\to \Pic(X)[\ell^\infty]$
and $H^i(B,\sO_B)\to H^i(X,\sO_X)$ for $i\geq 0$
are isomorphisms.
\end{lem}

\begin{proof}
The assertion on $H^i(B,\sO_B)\to H^i(X,\sO_X)$ follows from \cite[Theorem~7.1]{kollartorsionfree}
and from the Lefschetz principle.
The other assertion
is a close variation on \cite[Theorem~5.2]{kollarshaf};
for the sake of completeness, we provide a proof.
First,
we recall that
\begin{align}
\label{eq:picpi1dual}
\Pic(Z)[\ell^\infty]=\Hom(\pi_1^\et(Z)^{\ell,\ab},\Q/\Z(1))
\end{align}
for any smooth and
proper variety~$Z$ over an algebraically closed field of characteristic~$0$,
where $\pi_1^\et(Z)^{\ell,\ab}$ denotes the pro\nobreakdash-$\ell$ completion
of $\pi_1^\et(Z)^\ab$ (a finitely generated $\Zl$\nobreakdash-module).
In particular $\pi_1^\et(Z)^\ab/\ell=0$ if and only if $\Pic(Z)[\ell]=0$
(where, for any abelian group~$M$, we write $M/\ell$ for $M/\ell M$).

As~$F$ is irreducible and satisfies $H^i(F,\sO_F)=0$ for $i>0$, the morphism~$f$ has no multiple
fibre above codimension~$1$ points of~$B$ (see \cite[Proposition~7.3~(iii)]{ctvoisin}, \cite[Proposition~2.4]{elw}).
As, on the other hand, the group $\pi_1^\et(F)^\ab/\ell$ vanishes, it follows that $\pi_1^\et(X_b)^\ab/\ell$ vanishes
for any geometric point~$b$ of~$B$ above a codimension~$1$
point of~$B$ (see \cite[Proposition~6.3.5~(ii) a)]{Raynaud}).
As $\pi_1^\et(X_b)^\ab/\ell$ and $H^1(X_b,\Z/\ell\Z)$ are Pontrjagin dual and as
$\RR^1f_*\Z/\ell\Z$ is a constructible sheaf, we deduce that $\pi_1^\et(X_b)^\ab/\ell$,
hence also $\pi_1^\et(X_b)^{\ell,\ab}$, vanishes
for any geometric point~$b$ of an open subset $B^0 \subset B$ whose complement has codimension at least~$2$.
Letting $X^0=f^{-1}(B^0)$,
we can now apply \cite[Exp.~IX, Corollaire~6.11]{SGA1}
to conclude that the natural map $\pi_1^\et(X^0)^{\ell,\ab} \to \pi_1^\et(B^0)^{\ell,\ab}$
is an isomorphism.
As $\pi_1^\et(B^0)=\pi_1^\et(B)$ and $\pi_1^\et(X^0)\twoheadrightarrow \pi_1^\et(X)$ (\emph{op.\ cit.},
Exp.~V, Proposition~8.2 and Exp.~X, Corollaire~3.3),
it follows that the natural map
$\pi_1^\et(X)^{\ell,\ab} \to \pi_1^\et(B)^{\ell,\ab}$ is an isomorphism;
\emph{i.e.},
in view of~\eqref{eq:picpi1dual},
the pull-back map $\Pic(B)[\ell^\infty] \to \Pic(X)[\ell^\infty]$ is an isomorphism.
\end{proof}

\subsection{Algebraic approximation}
\label{subsec:algebraic approximation}

The proof of Theorem \ref{thm:fibres en coniques} will rely, in an essential way, on consequences of the Stone--Weierstrass approximation theorem.
One of these is an approximation result
due to Akbulut and King, which we now state.

If $M$ and $N$ are $\ci$ manifolds,
we endow $\ci(M,N)$ with the weak $\ci$ topology defined in \cite[p.\ 36]{Hirsch}.
We first recall the definition of algebraic approximation (see \cite[Definition 12.4.10]{bcr}).

\begin{defn}
\label{def:approximation}
Let $X$ be a smooth algebraic variety over $\R$.  A Zariski closed subset of $X(\R)$ is said to be \textit{nonsingular} if its Zariski closure in $X$ is smooth along $X(\R)$.

 A compact $\ci$ submanifold $Z\subset X(\R)$ is said to have \textit{an algebraic approximation} in $X(\R)$ if the inclusion $i:Z\hookrightarrow X(\R)$ can be approximated in $\ci(Z,X(\R))$ by maps whose images are nonsingular Zariski closed subsets of $X(\R)$.
\end{defn}

The existence of algebraic approximations for $\ci$ hypersurfaces is entirely understood (see \cite[Theorem 12.4.11]{bcr}, which goes back to Benedetti and Tognoli \cite[Proof of Theorem 4.1]{BenedettiTognoli}).
The following theorem  is a counterpart for curves.
Akbulut and King~\cite{AkbulutKing} established a slight variant,
which turns out to be
 equi\-valent to the statement given below
thanks to
a moving lemma of Hironaka~\cite{hironakasmoothing}.
Using entirely different methods,
Bochnak and Kucharz~\cite{bochnakkucharz} were able to give a direct proof
when~$X$ has dimension~$3$.

\begin{thm}[Akbulut--King]
\label{thm:AK}
Let $X$ be a smooth projective algebraic variety over $\R$ and $i:Z\hookrightarrow X(\R)$ be a one-dimensional compact  $\ci$ submanifold. The following are equivalent:
\begin{enumerate}[(i)]
\item $[Z]\in H_1^{\alg}(X(\R),\Z/2\Z)$,
\item $Z$ has an algebraic approximation in $X(\R)$,
\item For every neighbourhood $V$ of $Z$ in $X(\R)$, there exists a $\ci$ isotopy $$\phi:Z\times [0,1]\to V$$ such that $\phi|_{Z\times\{0\}}=i$ and $\phi|_{Z\times\{1\}}$ is arbitrarily close to $i$ in $\ci(Z,V)$ with nonsingular Zariski closed image.
\end{enumerate}
\end{thm}

\begin{proof}
The implication (iii)$\Rightarrow$(i) is immediate and (ii)$\Rightarrow$(iii) follows from
\cite[Proposition~4.4.4]{Wall}. It remains to prove (i)$\Rightarrow$(ii).
For this implication, the case $d=1$ is obvious, $d=2$ is \cite[Theorem 12.4.11]{bcr} and $d=3$ is due to
Bochnak and Kucharz \cite[Theorem 1.1]{bochnakkucharz}.
For any $d\geq 3$, it is a theorem of Akbulut and King \cite[Proposition~2 and Remark]{AkbulutKing}
that the condition
\begin{enumerate}
\item[(i')] $[Z]\in H_1(X(\R),\Z/2\Z)$ is a linear combination of Borel--Haefliger cycle classes
of \emph{smooth} algebraic curves lying in~$X$
\end{enumerate}
implies~(ii).  On the other hand,
conditions~(i) and~(i') are in fact equivalent, by a theorem
of Hironaka \cite{hironakasmoothing} according to which the classes of smooth curves generate
the Chow group of $1$\nobreakdash-cycles
of any smooth and projective
variety of dimension~$\geq 3$ over an infinite perfect field.

The ground field is assumed
to be algebraically closed
in \cite[p.~50, Theorem]{hironakasmoothing}.
As it turns out, Hironaka's argument for the proof of the above assertion
does not make use of this assumption.
Nevertheless, in order to dispel any doubt regarding its validity
over~$\R$, we give a brief outline of the argument.
Let $\iota:B \hookrightarrow X$ be a closed embedding, where~$B$ is an integral curve.
Let $\nu:B'\to B$
be its normalisation.
Choose a finite morphism $g:B' \to \P^1_\R$ in such a way that
$(\iota\circ \nu)\times g: B' \to X \times \P^1_\R$ is an embedding.
Applying the following proposition
twice, to $C=B'$ and to $C=\varnothing$ (with the same~$H$ and~$n$),
now shows that~$B$ is rationally equivalent, as a cycle on~$X$,
to the difference of two smooth curves.

\begin{prop}
\label{prop:hironaka51}
Let~$X$ be a smooth and projective variety of dimension $d\geq 3$ over an infinite field.
Let~$H$ be an ample divisor on $X \times \P^1$.
Let $C \subset X \times \P^1$ be a smooth closed subscheme everywhere of dimension~$1$,
with ideal sheaf $\sI \subset \sO_{X \times \P^1}$.
Assume that~$C$ does not contain any fibre of the first projection $p:X \times \P^1 \to X$.
For~$n$ large enough,
if $\sigma_1,\dots,\sigma_d \in H^0(X \times \P^1, \sI(nH))$ are general,
the subvariety of $X \times \P^1$ defined by $\sigma_1=\dots=\sigma_d=0$
is equal to $C \cup D$ for a smooth curve $D \subset X \times \P^1$
such that the map $D \to X$ induced by~$p$ is a closed embedding.
\end{prop}

Proposition~\ref{prop:hironaka51}
is what the proof of \cite[Theorem~5.1]{hironakasmoothing}
applied to $f=p$ and $Z'=C$
really establishes.
On the other hand,
in order to prove Proposition~\ref{prop:hironaka51}, it is clear that
the ground field may be assumed to be algebraically closed.
\end{proof}

\subsection{Sarkisov standard models}

In the proof of Theorem \ref{thm:fibres en coniques}, it will be convenient to replace the morphism $f:X\to B$ by one that is birational to it and whose geometry is as simple as possible. Such models were constructed by Sarkisov~\cite{Sarkisov}.

\subsubsection{Definition and existence}
\label{subsubsec:def sarkisov}

A morphism $f:X\to B$ of varieties over a field $k$ is said to be \textit{a Sarkisov standard model} if it satisfies the following conditions\footnote{Contrary to common usage, we do not require that $f$ be relatively minimal.}:
\begin{enumerate}[(i)]
\item The varieties $X$ and $B$ are smooth, projective and connected.
\item The morphism $f$ is flat.
\item There exists a reduced simple normal crossings divisor $\Delta\subset B$ 
such that the fibers of $f$ are smooth (rank $3$) conics outside of $\Delta$, rank $2$ conics above the regular locus of $\Delta$, and rank $1$ conics above the singular locus $\Sigma$ of $\Delta$.
\end{enumerate}

That every conic bundle is birational to a Sarkisov standard model has been proven by Sarkisov \cite[Proposition 1.16]{Sarkisov} over an algebraically closed field of characteristic $0$. His proof goes through over 
non-closed fields of characteristic $0$ and another proof has been provided by Oesinghaus \cite[Theorem 7]{Oesinghaus}.

\begin{thm}[\cite{Sarkisov, Oesinghaus}]
\label{thm:existenceSarkisov}
Let $f:X\to B$ be a morphism of integral varieties over a field $k$ of characteristic $0$ whose generic fiber is a smooth conic. There exist a birational morphism $B'\to B$ and a birational rational map $X'\dashrightarrow X$ inducing a morphism $f' : X'\to B'$ that is a Sarkisov standard model.
\end{thm}

\subsubsection{Computing their cohomology}
\label{subsub:cohoSarkisov}
Let $f:X\to B$ be a Sarkisov standard model over $\R$ whose degeneracy locus $\Delta$ has singular locus $\Sigma$.
 Let $i_{\Delta}:\Delta(\C)\to B(\C)$, $i_{\Sigma}:\Sigma(\C)\to B(\C)$ and $j:(\Delta\setminus\Sigma)(\C)\to\Delta(\C)$ be the inclusions. 
Sheafifying the push-forward maps $\big(f|_{f^{-1}(U)}\big)_*:H^2(f^{-1}(U),\Z(1))\to H^0(U,\Z)$
yields a morphism of $G$-equivariant sheaves $\Trf:\RR^2f_*\Z(1)\to\Z$ on $B(\C)$: the trace morphism, whose analogue in \'etale cohomology is defined in \cite[Th\'eor\`eme 2.9]{delignedualite}. 

The  morphism $\Trf$ may be computed over each stratum using the proper base change theorem. Over $B\setminus\Delta$, it is an isomorphism. Over $\Delta\setminus\Sigma$, $\RR^2f_*\Z(1)|_{(\Delta\setminus\Sigma)(\C)}$ is a $\Z$-local system of rank~$2$ whose stalk at a point is generated by the classes of points on the  two lines in the fibers, and the trace morphism is a surjection. The kernel of $\Trf |_{(\Delta\setminus\Sigma)(\C)}$ is then a $G$-equivariant $\Z$-local system $\mathscr{L}$ of rank $1$ on $(\Delta\setminus\Sigma)(\C)$ (corresponding geometrically to the monodromy action on the two lines of the rank~$2$ conics). Over~$\Sigma$, $\RR^2f_*\Z(1)|_{\Sigma(\C)}\simeq\Z$ and the trace morphism $\Trf |_{\Sigma(\C)}:\Z\to\Z$ is the multiplication by $2$ map. We deduce a canonical exact sequence of $G$-equivariant sheaves on $B(\C)$:
\begin{equation}
\label{eq:morphisme trace}
0\to (i_{\Delta})_*j_!\mathscr{L}\to \RR^2f_*\Z(1)\stackrel{\Trf}\longrightarrow \Z\to (i_{\Sigma})_*\Z/2\Z\to 0.
\end{equation}

\subsection{Proof of Theorem \ref{thm:fibres en coniques}}
\label{subsec:proofconicbundles}
This paragraph is devoted to the proof of Theorem~\ref{thm:fibres en coniques}.
Since the validity of the real integral Hodge conjecture for $1$-cycles is a birational invariant \cite[Proposition~\ref*{BW1-prop:birinvIHC}]{BW1}, we may assume that $f:X\to B$ is a standard Sarkisov model, by Theorem~\ref{thm:existenceSarkisov}. We define $\Delta$ and $\Sigma$ as in~\textsection\ref{subsubsec:def sarkisov}. Let $d$ be the dimension of $B$. The statement of Theorem~\ref{thm:fibres en coniques} is obvious if $d=0$ and follows from \cite[Proposition~\ref*{BW1-prop:real(1,1)}]{BW1} if $d=1$. From now on, we suppose that $d\geq 2$.

We fix $\alpha\in \Hdg^{2d}_G(X(\C),\Z(d))_0$ and shall now prove that $\alpha$ is in the image of the cycle class map $\cl:\CH_1(X)\to  H^{2d}_G(X(\C),\Z(d))$. We proceed in six steps.

\subsubsection{Analysis of the push-forward}
\label{subsub:reallocus}
We consider the push-forward homomorphism $f_*:H^{2d}_G(X(\C),\Z(d)) \to H^{2d-2}_G(B(\C),\Z(d-1))$
(see \cite[(\ref*{BW1-eq:pushforward equivariant complex})]{BW1})
and first make use of the algebraicity of $f_*\alpha$.

\begin{Step}
\label{step:real locus1}
We may assume that $\alpha|_{X(\R)}\in H^{2d}_G(X(\R),\Z(d))$ vanishes.
\end{Step}

\begin{proof}
The class $f_*\alpha$ belongs to $\Hdg^{2d-2}_G(B(\C),\Z(d-1))_0$ by \cite[Theorem~\ref*{BW1-th:stability of topological constraints}]{BW1}
 hence is algebraic by the hypothesis on $B$. Let $y\in \CH_1(B)$ be such that $f_*\alpha=\cl(y)$.

Consider the image of~$\alpha$ by the natural map
$$H^{2d}_G(X(\C),\Z(d))_0 \to H^d(X(\R),\Z/2\Z)=H_1(X(\R),\Z/2\Z)$$
(see \cite[(\ref*{BW1-eq:psi with point})]{BW1}) and represent it by a topological cycle. By $\ci$ approximation \cite[Chapter 2, Theorem 2.6]{Hirsch}, we may assume that it is a sum of $\ci$ loops $(g_i:\bS^1\to X(\R))_{1\leq i\leq s}$.
Let $h_i:=f(\R)\circ g_i:\bS^1\to B(\R)$ and $\bS:=\sqcup_i\bS^1$, and consider the disjoint union maps $g:=\sqcup_i g_i:\bS\to X(\R)$ and $h:=\sqcup_i h_i:\bS\to B(\R)$.

Since $f$ is a Sarkisov standard model, the locus where $f$ is not smooth has codimension $\geq 2$ in $X$, hence real codimension $\geq 2$ in $X(\R)$.
Since moreover $d\geq 2$, an application of multijet transversality \cite[Theorem 4.6.1]{Wall} shows that
after deforming $g$, which preserves its homology class, we may assume that its image is contained in the smooth locus of $f$, that~$h$ is an immersion,
that $h(\bS)$ has distinct tangents at the finite number of points over which $h$ is not injective, and that these points do not belong to $\Delta$
(see the proof of \emph{loc.\ cit.}, Theorem~4.7.7 and apply \cite[Theorem~21.5~(5)]{michorbook}).

Defining $\pi:B'\to B$ to be the blow-up of $B$ at this finite number of points, the map $h$ lifts to an embedding $h':\bS\to B'(\R)$. We introduce the cartesian diagram:
$$
\xymatrix@R=3ex{
X'\ar[d]^(.42){f'}\ar[r] & X \ar[d]^(.42)f \\
B'\ar[r]^{\pi}& B.
}
$$
The map $g$ lifts to an embedding $g':\bS\to X'(\R)$. We still denote by $\bS$ the images of $g'$ and $h'$.
By Lemma \ref{lemtubul}, there exists a neighbourhood $V$ of $\bS$ in $B'(\R)$ and
a $\ci$ map $s:V\to X'(\R)$ that is the identity on $\bS$, yielding a commutative diagram:
$$
\xymatrix{
&& X'(\R) \ar[r]\ar[d]^{f'(\R)}& X(\R)\ar^{f(\R)}[d] \\
\bS\ar@/_1pc/[rr]_{h'}\ar@{^{(}->}[r]\ar@<.4em>[rru]^{g'}&V\ar@<.15em>[ru]_s\ar@{^{(}->}[r]& B'(\R)\ar_{\pi(\R)}[r] & B(\R).
}
$$
The class $[h(\bS)]\in H^{d-1}(B(\R),\Z/2\Z)$  equals $\cl_{\R}(y)$ by \cite[Theorem~\ref*{BW1-th:conditions de krasnov} and Theorem~\ref*{BW1-th:stability of topological constraints}]{BW1},
hence is algebraic.
Since the kernel of $\pi(\R)_*:H^{d-1}(B'(\R),\Z/2\Z)\to H^{d-1}(B(\R),\Z/2\Z)$  is generated by classes of lines on the exceptional divisor of $\pi$,
one has $[h'(\bS)]\in H^{d-1}_{\alg}(B'(\R),\Z/2\Z)$. Then Theorem~\ref{thm:AK} (i)$\Rightarrow$(iii)
applies:
there exist a smooth projective curve $\Gamma$ and a morphism $\mu:\Gamma\to B'$ such that $\mu(\Gamma(\R))\subset V$ and $[\mu(\Gamma(\R))]=[h'(\bS)]\in H_1(V,\Z/2\Z)$.

Let $X'_\Gamma\to \Gamma$ be the base change of $f'$ by $\mu$, $S\to X'_\Gamma$ be a resolution of singularities and $\psi:S\to\Gamma$ be the composition.
We have a commutative diagram:
$$
\xymatrix@R=3ex{
S\ar[r]\ar[rd]_\psi&X'_\Gamma\ar[d]\ar[r]& X'\ar[d]^(.42){f'} \\
&\Gamma\ar[r]^{\mu}&B'.
}
$$
The lifting $s:V\to X'(\R)$ induces a $\ci$ section still denoted by $s:\Gamma(\R)\to S(\R)$.
In the commutative diagram
$$
\xymatrix@R=3ex{
H_1(S(\R),\Z/2\Z)\ar[r]&H_1(X'(\R),\Z/2\Z)\ar[r]& H_1(X(\R),\Z/2\Z) \\
H_1(\Gamma(\R),\Z/2\Z)\ar[u]^{s_*}\ar[r]& H_1(V,\Z/2\Z),\ar[u]^{s_*}&
}
$$
the image in $H_1(V,\Z/2\Z)$ of the fundamental class of $\Gamma(\R)$ is $[h'(\bS)]$. We deduce that  its image in $H_1(X(\R),\Z/2\Z)$ is $[g(\bS)]$. By Lemma \ref{lem:section conic bundle} below, its image in $H_1(S(\R),\Z/2\Z)$ is algebraic, hence so is its image in $H_1(X(\R),\Z/2\Z)$. We have shown the existence of $z\in\CH_1(X)$ such that $\cl_{\R}(z)=[g(\bS)]$. 
Using the decomposition \cite[(\ref*{BW1-eq:canonical decomposition})]{BW1} and applying \cite[Theorem~\ref*{BW1-th:conditions de krasnov}]{BW1} with $Y=z$, we now see that $(\alpha-\cl(z))|_{X(\R)}=0$.
Replacing $\alpha$ by $\alpha-\cl(z)$ completes the proof.
\end{proof}

We have used the following two lemmas.

\begin{lem}
\label{lemtubul}
Let $\phi: M\to N$ be a $\ci$ map between $\ci$ manifolds, and let $Z\subset M$ be a closed $\ci$ submanifold such that $\phi$ is submersive along $Z$ and such that $\phi|_Z: Z\to N$ is an embedding. Then there exists an open neighbourhood $V$ of~$\phi(Z)$ in $N$ and a $\ci$ section $s: V\to M$ of $\phi$ above $V$ such that $s\circ \phi|_Z=\Id_Z$.
\end{lem}

\begin{proof}
Choose a tubular neighbourhood of $Z$ in $M$ contained in the locus where $\phi$ is submersive,
and view it as an open embedding of the normal bundle $\iota: N_{Z/M}\to M$. The choice of a metric on $N_{Z/M}$ induces a splitting $\sigma: N_{\phi(Z)/N}\to N_{Z/M}$ of the surjection $\phi_*:N_{Z/M}\to N_{\phi(Z)/N}$ of  $\ci$ vector bundles on $Z$.
 The composition $\phi\circ \iota\circ\sigma$ realizes a diffeomorphism between an open neighbourhood of $\phi(Z)$ in $N_{\phi(Z)/N}$ and an open neighbourhood $V$ of $\phi(Z)$ in $N$ (apply \cite[Corollary A.2.6]{Wall}). The existence of a $\ci$ map $s: V\to M$ as required follows from the construction.
\end{proof}

\begin{lem}
\label{lem:section conic bundle}
Let $\psi:S\to\Gamma$ be a morphism between smooth projective connected varieties over $\R$ whose base is a curve, and whose generic fiber is a conic. Let $s:\Gamma(\R)\to S(\R)$ be a $\ci$ section. Then $[s(\Gamma(\R))]\in H^1_\alg(S(\R),\Z/2\Z)$.
\end{lem}

\begin{proof}
We suppose that $S(\R)\neq\emptyset$ since the theorem is trivial otherwise. In particular, $S$ and $\Gamma$ are geometrically connected.
By Lemma \ref{lem:Hi0 Pictors}, $H^2(S,\sO_S)=0$ and $\Pic(\Gamma_{\C})_{\tors}\to \Pic(S_{\C})_{\tors}$ is an isomorphism. By \cite[8.1/4]{neronmodels}, $\Pic(\Gamma)\simeq\Pic(\Gamma_{\C})^G$ and $\Pic(S)\simeq\Pic(S_{\C})^G$,
so that $\psi^*:\Pic(\Gamma)_{\tors}\to \Pic(S)_{\tors}$ is an isomorphism.
For any $\mathcal{L}\in\Pic(\Gamma)$,
the projection formula shows that
\begin{align*}
\deg\big([s(\Gamma(\R))]\smile\cl_{\R}(\psi^*\mathcal{L})\big)=\deg\big(\cl_{\R}(\Gamma)\smile\cl_{\R}(\mathcal{L})\big)\in\Z/2\Z\rlap{.}
\end{align*}
Since torsion line bundles have degree~$0$, the right-hand side of the above equation vanishes for any $\mathcal{L} \in \Pic(\Gamma)[2^\infty]$.  We now deduce from \cite[Corollary~\ref*{BW1-cor:surfaces}~(i)]{BW1} that $[s(\Gamma(\R))]\in H^1_\alg(S(\R),\Z/2\Z)$.
\end{proof}

\begin{Step}
\label{step:pushforward}
We may assume that  $f_*\alpha\in H^{2d-2}_G(B(\C),\Z(d-1))$ vanishes.
\end{Step}

\begin{proof}
The argument at the beginning of the proof of Step \ref{step:real locus1}
shows the existence of $y\in \CH_1(B)$ such that $f_*\alpha=\cl(y)$. By Step \ref{step:real locus1}, $\alpha|_{X(\R)}=0$, so that $(f_*\alpha)|_{X(\R)}=0$ by \cite[Theorem~\ref*{BW1-th:stability of topological constraints}]{BW1}
and $\cl_{\R}(y)=0$ by \cite[Theorem~\ref*{BW1-th:conditions de krasnov}]{BW1}. Ischebeck and Sch\"ulting have shown \cite[Main Theorem (4.3)]{IS} that there exists a $1$-cycle $\Gamma$ on $B$ whose support has only finitely many real points, and such that $[\Gamma]=y\in\CH_1(B)$ (when $B$ is a surface, this follows from Br\"ocker's EPT
theorem \cite{brocker}).

Since $f$ is a Sarkisov standard model, the normalisations of the components of $\Gamma$ lift to $X$. For those that are not contained in $\Delta$, this is a theorem of Witt \cite[Satz~22]{WittHasse}; those contained in $\Delta$ but not in $\Sigma$ lift uniquely to the locus where $f$ is not smooth; and those contained in $\Sigma$ lift because $f|_{\Sigma}:f^{-1}(\Sigma)_{\red}\to\Sigma$ is a $\P^1$\nobreakdash-bundle.
This proves the existence of a $1$-cycle $\tilde{\Gamma}$ on $X$ whose class $\tilde y:=[\tilde{\Gamma}]\in \CH_1(X)$ satisfies $f_*\tilde y=y\in \CH_1(B)$, and $\cl_{\R}(\tilde y)=0$ because the support of $\tilde{\Gamma}$ has only finitely many real points. By \cite[Theorem~\ref*{BW1-th:conditions de krasnov}]{BW1}, $\cl(\tilde y)|_{X(\R)}\in H^{2d}_G(X(\R),\Z(d))$ vanishes.
By \cite[Theorem~\ref*{BW1-th:stability of topological constraints}]{BW1},
replacing $\alpha$ by $\alpha-\cl(\tilde y)$ finishes the proof of Step~\ref{step:pushforward}. The conclusion of Step~\ref{step:real locus1} remains valid because $\cl(\tilde y)|_{X(\R)}=0$.
\end{proof}

\subsubsection{The restriction above the real locus}

Let $U\subset B$ be the complement of finitely many points, including at least one in each connected component of $B(\R)$, $B(\C)$ and $\Delta(\C)$, and at least one in the image of each connected component of $X(\R)$.
Denote by $f_U:X_U\to U$
the base change of $f:X\to B$.
Let $\Xi:=f^{-1}(U(\R))$ be the inverse image of $U(\R)$ by $f:X(\C)\to B(\C)$. 
Building on Step~\ref{step:real locus1}, we prove:

\begin{Step}
\label{step:real locus2}
The restriction $\alpha|_{\Xi}\in H^{2d}_G(\Xi,\Z(d))$ vanishes.
\end{Step}

\begin{proof}
There is an exact sequence of relative cohomology:
$$ H^{2d}_G(\Xi, X_U(\R),\Z(d))\to H^{2d}_G(\Xi,\Z(d))\to H^{2d}_G(X_U(\R),\Z(d)).$$
Since $\alpha|_{X_U(\R)}=0$ by Step~\ref{step:real locus1}, it suffices to prove that  $H^{2d}_G(\Xi, X_U(\R),\Z(d))=0$. 
The real-complex exact sequence \cite[(\ref*{BW1-eq:real-complex long 01})]{BW1} yields:
\begin{equation}
\label{eq:devissagexi}
H^{2d}(\Xi, X_U(\R),\Z)\to H^{2d}_G(\Xi, X_U(\R),\Z(d))\to H^{2d+1}_G(\Xi, X_U(\R),\Z(d+1)).
\end{equation}
The left-hand side of (\ref{eq:devissagexi}) vanishes by the relative cohomology exact sequence:
$$ H^{2d-1}(X_U(\R),\Z)\to H^{2d}(\Xi, X_U(\R),\Z)\to H^{2d}(\Xi,\Z).$$
Indeed, $H^{2d-1}(X_U(\R),\Z)=0$ by Poincar\'e duality since $2d-1\geq d+1$ and $X_U(\R)$ has no compact component
by the choice of $U$;
 and $H^{2d}(\Xi,\Z)=0$ because the only term that might contribute to it in the Leray spectral sequence for $f|_{\Xi}:\Xi\to U(\R)$ is $H^2(U(\R),\RR^2f_*\Z)$ if $d=2$, and computing this latter group using  (\ref{eq:morphisme trace})  shows that it vanishes by cohomological dimension \cite[Ch. II, Lemma 9.1]{delfshomology} and because $H^2(U(\R),\Z)=0$ as $U(\R)$ has no compact component.

To see that the right-hand side of (\ref{eq:devissagexi}) vanishes, we let $j:\Xi\setminus X_U(\R)\to \Xi$ be the inclusion and apply the equivariant cohomology spectral sequence \cite[(\ref*{BW1-eq:first spectral sequence})]{BW1} to $\sF=j_!\Z(d+1)$, noting that $H^{2d+1}(\Xi/G,\sH^0(G,\sF))$ vanishes by \cite[\emph{loc.\ cit.}]{delfshomology} since $2d+1>\dim(\Xi/G)$.
\end{proof}

\subsubsection{The Leray spectral sequence}
\label{subsubsec:Leray}

By Step \ref{step:real locus2}, the class $\alpha|_{X_U(\C)}\in H^{2d}_G(X_U(\C),\Z(d))$ lifts to a class $\tilde{\alpha}\in H^{2d}_G(X_U(\C),\Xi,\Z(d))$ in relative cohomology.

To compute the equivariant cohomology of a sheaf on~$X_U(\C)$, we shall consider the Leray spectral sequence for~$f_U$ obtained by composing the push-forward $f_{U*}$ from $G$-equivariant abelian sheaves on $X_U(\C)$ to $G$-equivariant abelian sheaves on $U(\C)$
with the invariant global sections functor. Applied to the sheaf $\Z(d)$, it reads:
\begin{equation}
\label{eq:LeraySarkisov}
E_2^{p,q}=H^p_G(U(\C),\RR^qf_{*}\Z(d))\Rightarrow H_G^{p+q}(X_U(\C),\Z(d)).
\end{equation}
Applying it to $u_!\Z(d)$, where $u:X_U(\C)\setminus\Xi\to X_U(\C)$ is the inclusion, and computing $\RR^qf_*(u_!\Z(d))$ using proper base change yields:
\begin{equation}
\label{eq:LeraySarkisovbis}
E_2^{p,q}=H^p_G(U(\C),U(\R),\RR^qf_{*}\Z(d))\Rightarrow H_G^{p+q}(X_U(\C),\Xi,\Z(d)).
\end{equation}
The natural morphism $u_!\Z(d)\to\Z(d)$ induces a morphism from (\ref{eq:LeraySarkisovbis}) to (\ref{eq:LeraySarkisov}).

By proper base change and since $f$ is a Sarkisov standard model, the only non-zero $G$-equivariant sheaves among the $\RR^qf_{*}\Z(d)$ are $f_{*}\Z(d)\simeq\Z(d)$ and $\RR^2f_{*}\Z(d)$. Consider the edge maps $\varepsilon:H^{2d}_G(X_U(\C),\Z(d))\to H^{2d-2}_G(U(\C),\RR^2f_{*}\Z(d))$ and $\tilde{\varepsilon}:H^{2d}_G(X_U(\C),\Xi,\Z(d))\to H^{2d-2}_G(U(\C),U(\R),\RR^2f_{*}\Z(d))$, and recall from \S\ref{subsub:cohoSarkisov} the definition of the trace map $\Trf: \RR^2f_{*}\Z(1)\to\Z$.

\begin{Step}
\label{step:notrace}
There exists an open subset $B^0\subset U$ such that $B\setminus B^0$ has dimension $\leq 1$ and such that 
$\Trf(\tilde{\varepsilon}(\tilde{\alpha}))|_{B^0}\in H^{2d-2}_G(B^0(\C),B^0(\R),\Z(d-1))$ vanishes.
\end{Step}

\begin{proof} 
Consider the long exact sequence of relative cohomology for $U(\R)\subset U(\C)$. 
The image of 
$\Trf(\tilde{\varepsilon}(\tilde{\alpha}))\in H^{2d-2}_G(U(\C),U(\R),\Z(d-1))$  in $H^{2d-2}_G(U(\C),\Z(d-1))$ is $\Trf(\varepsilon(\alpha|_{X_U(\C)}))$, by the compatibility between  (\ref{eq:LeraySarkisovbis}) and (\ref{eq:LeraySarkisov}). By the definition of the trace morphism, this is $(f_*\alpha)|_{U(\C)}$, which vanishes by Step \ref{step:pushforward}.
We deduce that $\Trf(\tilde{\varepsilon}(\tilde{\alpha}))$ lifts to a class $\eta_U\in H^{2d-3}_G(U(\R),\Z(d-1))$.

By purity \cite[(\ref*{BW1-eq:equivariant purity})]{BW1} applied with $V=B(\R)\setminus U(\R)$, $W=B(\R)$, $i=2d-2$, $c=d$ and $\mathscr{F}=\Z(d-1)$, one has, after choosing local orientations of~$B(\R)$ at the
(finitely many) points of $B(\R)\setminus U(\R)$, an exact sequence:
$$H^{2d-3}_G(B(\R),\Z(d-1))\to H^{2d-3}_G(U(\R),\Z(d-1))\to H^{d-2}_G(B(\R)\setminus U(\R),\Z(d-1)).$$
Since $H^{d-2}(G,\Z(d-1))=0$ and $B(\R)\setminus U(\R)$ is finite, the right-hand side group vanishes; thus $\eta_U$ lifts to a class $\eta\in H^{2d-3}_G(B(\R),\Z(d-1))$.

Let $B^0\subset B$ be given by Lemma \ref{lem:AKhomology} below. Replacing it by its intersection with~$U$, we may assume that $B^0\subset U$.
Then $\eta|_{B^0}$ lifts to $H^{2d-3}_G(B^0(\C),\Z(d-1))$, so that its image $\Trf(\tilde{\varepsilon}(\tilde{\alpha}))|_{B^0}\in H^{2d-2}_G(B^0(\C),B^0(\R),\Z(d-1))$ vanishes.
\end{proof}

We have used:

\begin{lem}
\label{lem:AKhomology}
Let $\eta\in H^{2d-3}_G(B(\R),\Z(d-1))$. There exists an open subset $B^0\subset B$ such that $B\setminus B^0$ has dimension $\leq 1$ and $$\eta|_{B^0(\R)}\in \Ima \left(H^{2d-3}_G(B^0(\C),\Z(d-1))\to  H^{2d-3}_G(B^0(\R),\Z(d-1))\right)\mkern-3mu\rlap{.}$$ 
\end{lem}

\begin{proof}
Let $E=H^{d-3}(G,\Z(d-1))$.
For any open subset~$B^0$ of~$B$, we recall that there is a canonical decomposition 
(see \cite[(\ref*{BW1-eq:canonical decomposition})]{BW1})
\begin{align}
\label{eq:canonical decomposition in proof of lemma AKhomology}
H^{2d-3}_G(B^0(\R),\Z(d-1)) = H^d(B^0(\R),E)\oplus \mkern-10mu\bigoplus_{\substack{p<d \\ p\mkern1mu\equiv\mkern1mu d \text{ mod } 2}}\mkern-10muH^p(B^0(\R),\Z/2\Z)\rlap{.}
\end{align}
For $p<d$ with $p\mkern1mu\equiv\mkern1mu d \text{ mod } 2$, let $\eta_p \in H^p(B(\R),\Z/2\Z)$ denote the degree~$p$ component of~$\eta$
according to this decomposition for $B^0=B$.
By \cite[Proposition~\ref*{BW1-prop:truncated projection integral coeff}]{BW1}, to prove the lemma, we may assume that $\eta_p=0$ for $p\leq d-4$.

By a theorem of Akbulut and King \cite[Lemma 9 (1)]{AKhomology}, there exist a smooth projective surface $S$ over $\R$ and a morphism $\varphi: S\to B$ such that $\eta_{d-2}$ is the image by $\varphi_*:H^0(S(\R),\Z/2\Z)\to H^{d-2}(B(\R),\Z/2\Z)$ of a class $\sigma\in H^0(S(\R),\Z/2\Z)$.

By the Stone--Weierstrass approximation theorem, there exists a rational
function $h\in\R(S)^*$ that is invertible on $S(\R)$, such that $h>0$ on the
connected components of $S(\R)$ where $\sigma$ vanishes and $h<0$ on the
other connected components (see \cite[Theorem~3.4.4 and Theorem~8.8.5]{bcr}).

Let $B^0 \subset B$ be an open subset such that $\dim(B\setminus B^0)\leq 1$,
small enough that~$h$ is invertible on $S^0=\phi^{-1}(B^0)$
and that no connected component of~$B^0(\R)$ is compact.
View $h$ as a morphism $h:S^0\to\mathbf{G}_m$.
As a consequence of the isomorphism~\cite[(\ref*{BW1-eq:cohomological dimension iso relative})]{BW1},
of the localisation exact sequence, and of the remark
that $\C^*/G$ and~$\R^*$ are disjoint unions of contractible spaces, we
have $H^2_G(\C^*,\R^*,\Z(1))=0$.  This vanishing and the isomorphism
$H^1_G(\R^*,\Z(1))= H^0(\R^*,\Z/2\Z)$ \cite[(\ref*{BW1-eq:canonical decomposition})]{BW1}
show that there exists a class $\mu\in H^1_G(\C^*,\Z(1))$ whose
restriction to a real point $x\in\R^*$ vanishes if and only if $x>0$.
The class
$\nu:=h^*\mu\in H^1_G(S^0(\C),\Z(1))$ then has the property that
 if $x\in S^0(\R)$, $\left.\nu\right|_x=0$ if and only if $\left.\sigma\right|_x=0$.

It follows from \cite[\textsection\ref*{BW1-subsubsec:reduction modulo 2} and Proposition~\ref*{BW1-prop:differentiable rr}]{BW1} 
applied to $\varphi|_{S^0}:S^0\to B^0$ that
in the decomposition~\eqref{eq:canonical decomposition in proof of lemma AKhomology},
the degree~$p$ component of $((\varphi|_{S^0})_*\nu)|_{B^0(\R)}$ is equal to~$0$ for $p<d-2$
and to $(\phi_*\sigma)|_{B^0(\R)}$ for $p=d-2$.
On the other hand, as no connected component of~$B^0(\R)$ is compact,
we have $H^d(B^0(\R),E)=0$.
As $\phi_*\sigma=\eta_{d-2}$, this concludes the proof of the lemma.
\end{proof}

Define $f^0:X^0\to B^0$ to be the base change of $f:X\to B$ by the open immersion $B^0\subset B$, and let $\Delta^0:=\Delta\cap B^0$, $\Sigma^0:=\Sigma\cap B^0$ and $\Xi^0:=\Xi\cap X^0(\C)$. 
Shrinking $B^0$, we may assume that $\Delta^0=\emptyset$ if $d=2$.

\begin{Step}
\label{step:killonopen}
The restriction $\alpha|_{X^0} \in H^{2d}_G(X^0(\C),\Z(d))$ vanishes.
\end{Step}

\begin{proof}
We have proven in Step \ref{step:notrace} that the trace $\Trf(\tilde{\varepsilon}(\tilde{\alpha}))|_{B^0}$ of $\tilde{\varepsilon}(\tilde{\alpha})|_{B^0}$ vanishes. By the exact sequence (\ref{eq:morphisme trace}) and the vanishing of $H^{2d-3}_G(\Sigma^0(\C),\Sigma^0(\R), \Z/2\Z)$ and $H^{2d-2}_G(\Delta^0(\C),\Delta^0(\R), j_!\mathscr{L}(d-1))$ proven in Lemma \ref{lem:relativevanishing} below, $\tilde{\varepsilon}(\tilde{\alpha})|_{B^0}=0$.

The analogue of (\ref{eq:LeraySarkisovbis}) for the morphism $f^0:X^0\to B^0$ is a spectral sequence:
\begin{equation}
\label{eq:LeraySarkisovter}
E_2^{p,q}=H^p_G(B^0(\C),B^0(\R),\RR^qf_{*}\Z(d))\Rightarrow H_G^{p+q}(X^0(\C),\Xi^0,\Z(d)).
\end{equation}
Recall that $\RR^qf_{*}\Z(d)=0$ unless $q\in\{0,2\}$ and $f_*\Z(d)\simeq\Z(d)$.
By the compatibility between (\ref{eq:LeraySarkisovbis}) and (\ref{eq:LeraySarkisovter}), the image of $\tilde{\alpha}|_{X^0}\in H_G^{2d}(X^0(\C),\Xi^0,\Z(d))$ by the edge map of (\ref{eq:LeraySarkisovter}) coincides with $\tilde{\varepsilon}(\tilde{\alpha})|_{B^0}$ hence vanishes.
Consequently, (\ref{eq:LeraySarkisovter}) shows that $\tilde{\alpha}|_{X^0}$ comes from a class in $H^{2d}_G(B^0(\C),B^0(\R),\Z(d))$. Since this group vanishes by Lemma~\ref{lem:relativevanishing} below, $\tilde{\alpha}|_{X^0}=0$. We deduce that $\alpha|_{X^0}=0$.
\end{proof}

\begin{lem}
\label{lem:relativevanishing}
The groups $H^{2d}_G(B^0(\C),B^0(\R),\Z(d))$, $H^{2d-3}_G(\Sigma^0(\C),\Sigma^0(\R), \Z/2\Z)$  and $H^{2d-2}_G(\Delta^0(\C),\Delta^0(\R), j_!\mathscr{L}(d-1))$ vanish.
\end{lem}

\begin{proof}
We only prove the vanishing of $H^{2d-2}_G(\Delta^0(\C),\Delta^0(\R), j_!\mathscr{L}(d-1))$. The other two groups are easier to analyse and are dealt with similarly. We assume that $d\geq 3$ since otherwise $\Delta^0=\emptyset$.

By the real-complex exact sequence \cite[(\ref*{BW1-eq:real-complex long 01})]{BW1},
it suffices to show the vanishing of $H^{2d-1}_G(\Delta^0(\C),\Delta^0(\R), j_!\mathscr{L}(d))$ and $H^{2d-2}(\Delta^0(\C),\Delta^0(\R), j_!\mathscr{L})$. The first group is equal to $H^{2d-1}(\Delta^0(\C)/G,\Delta^0(\R), \sH^0(G,j_!\mathscr{L}(d)))$ by the spectral sequence \cite[(\ref*{BW1-eq:first spectral sequence})]{BW1} hence vanishes by cohomological dimension \cite[Ch.~II, Lemma~9.1]{delfshomology}.

The long exact sequence of relative cohomology  and \cite[\emph{loc.\ cit.}]{delfshomology} show that the vanishing of the second group will be implied by that of $H^{2d-2}(\Delta^0(\C), j_!\mathscr{L})$.
 Letting $\chi:\widetilde{\Delta\setminus\Sigma}\to\Delta\setminus\Sigma$ be the double cover associated with $\mathscr{L}$, there is an exact sequence 
$0\to\Z\to\chi_*\Z\to\mathscr{L}\to 0$ of $G$-equivariant sheaves on $(\Delta\setminus\Sigma)(\C)$.
By \cite[\emph{loc.\ cit.}]{delfshomology} again, it suffices to prove that $H^{2d-2}(\Delta^0(\C), j_!\chi_*\Z)=0$.

Let $\widetilde{j}:\widetilde{\Delta\setminus\Sigma}\hookrightarrow\widetilde{\Delta}$ be a smooth compactification fitting in a commutative diagram:
$$
\xymatrix@R=3ex{
\widetilde{\Delta\setminus\Sigma}\ar[d]^(.45){\chi}\ar[r]^(.56){\widetilde{j}} &\widetilde{\Delta} \ar[d]^(.45){\pi} \\
\Delta\setminus\Sigma\ar[r]^(.56){j}& \Delta\rlap{,}\vphantom{\setminus}
}
$$
and let $\widetilde{\Delta}^0:=\pi^{-1}(\Delta^0)$. The Leray spectral sequence and proper base change for $\pi$ show that $H^{2d-2}(\Delta^0(\C), j_!\chi_*\Z)=H^{2d-2}(\widetilde{\Delta}^0(\C), \widetilde{j}_!\Z)$. A final application of \cite[\emph{loc.\ cit.}]{delfshomology}\ to the cokernel of the injection $\widetilde{j}_!\Z\to\Z$ shows that it suffices to prove the vanishing of $H^{2d-2}(\widetilde{\Delta}^0(\C),\Z)$, which follows from Poincar\'e duality since $\Delta^0(\C)$, hence $\widetilde{\Delta}^0(\C)$, has no compact component.
\end{proof}

\subsubsection{Algebraicity}

We are ready to complete the proof of Theorem \ref{thm:fibres en coniques}.

\begin{Step}
\label{step:removestuff}
The class $\alpha$ is in the image of $\cl:\CH_1(X)\to  H^{2d}_G(X(\C),\Z(d))$.
\end{Step}

\begin{proof}
By Step \ref{step:killonopen}, $\alpha|_{X^0}=0$.
From the structure of Sarkisov standard models,  and since $B\setminus B^0$ has dimension $\leq 1$, one sees that there is a descending chain of open subsets $X=X_0\supset\dots \supset X_k=X^0$ such that for $0\leq i<k$, $Z_i:=X_i\setminus X_{i+1}$ is of one of the three types described in the statement of Lemma \ref{lem:algebraichomology} below. This can be achieved by successively removing from $X$, in an appropriate way, first singular points of fibers of $f$, then components of fibers of $f$, then curves in the non-smooth locus of $f|_{\Sigma}:X_{\Sigma}\to\Sigma$, then inverse images of curves in $B$. 
The closed subsets removed in the first three steps are of type~(i). Those removed in the last step are of type (iii) if the conics above the curve are smooth, and of type (ii) if the conics above the curve are rank~$1$ conics, or rank~$2$ conics whose singular points have already been removed in the third step. Indeed, in this last situation, $Z$ has a structure of $\A^1$-bundle over a double \'etale cover of the curve.

Let $\delta_i:=\dim(Z_i)$ and consider the long exact sequence of cohomology with supports at every step, taking into account purity \cite[(\ref*{BW1-eq:equivariant purity subvariety})]{BW1}:
$$H^{2\delta_i-2}_G(Z_i(\C),\Z(\delta_i-1))\to H^{2d}_G(X_i(\C),\Z(d))\to H^{2d}_G(X_{i+1}(\C),\Z(d)).$$
By Lemma \ref{lem:algebraichomology} below,
the leftmost group consists of classes of algebraic cycles, therefore so does the
kernel of $H^{2d}_G(X_i(\C),\Z(d))\to H^{2d}_G(X_{i+1}(\C),\Z(d))$, and the statement follows.
\end{proof}

\begin{lem}
\label{lem:algebraichomology}
Let $Z$ be a connected variety of dimension $\delta$ over $\R$ that is either
\begin{enumerate}[(i)]
\item a smooth curve or the point,
\item a (Zariski locally trivial) $\A^1$-bundle $g:Z\to \Gamma$ over a smooth curve $\Gamma$,
\item or a smooth conic bundle $g:Z\to \Gamma$ over a smooth curve $\Gamma$.
\end{enumerate}
Then the cycle class map $\CH_1(Z)\to H^{2\delta-2}_{G}(Z(\C),\Z(\delta-1))$ is surjective.
\end{lem}

\begin{proof}
Assertion (i) is obvious.
In case (ii), the Leray spectral sequence 
\begin{equation}
\label{eq:Lerayalgebraicity}
E_2^{p,q}=H^p_G(\Gamma(\C),\RR^qg_*\Z(1))\Rightarrow H^{p+q}_G(Z(\C),\Z(1))
\end{equation}
for $g$ shows that $H^2_G(\Gamma(\C),\Z(1))\to H^2_G(Z(\C),\Z(1))$ is an isomorphism.
On the other hand, the cycle class map
$\Pic(\Gamma) \to H^2_G(\Gamma(\C),\Z(1))$
is surjective since its cokernel is torsion-free
(see \cite[Proposition~\ref*{BW1-prop:real(1,1)nonarch}]{BW1})
while $H^2_G(\Gamma(\C),\Q(1))=H^2(\Gamma(\C),\Q(1))^G$
is generated by the class of any closed point of~$\Gamma$.
Assertion~(ii) follows.
In case (iii), the Leray spectral sequence (\ref{eq:Lerayalgebraicity}) gives rise to an exact sequence:
$$0\to H^2_G(\Gamma(\C),\Z(1))\stackrel{g^*}{\longrightarrow}  H^2_G(Z(\C),\Z(1))\stackrel{g_*}{\longrightarrow} H^0_G(\Gamma(\C),\Z)\simeq \Z.$$
The argument given above shows that $H^2_G(\Gamma(\C),\Z(1))$ is generated by algebraic cycles. Moreover the  relative anticanonical line bundle gives rise to a class in $H^2_G(Z(\C),\Z(1))$ whose push-forward is $2\in\Z\simeq H^0_G(\Gamma(\C),\Z)$. To conclude suppose that there is a cohomology class in $H^2_G(Z(\C),\Z(1))$ whose push-forward is $1\in\Z\simeq H^0_G(\Gamma(\C),\Z)$. Restricting it to a fiber of $g$ over a real point $x\in \Gamma(\R)$ and applying \cite[Proposition~\ref*{BW1-prop:real(1,1)nonarch}]{BW1} to the real conic $Z_x$ shows that $Z_x$ has a real point, hence that $Z(\R)\to \Gamma(\R)$ is surjective. It is then a theorem of Witt \cite[Satz~22]{WittHasse} that $g$ has a section. The class in $H^2_G(Z(\C),\Z(1))$ of the image of this section induces $1\in\Z\simeq H^0_G(\Gamma(\C),\Z)$, which concludes the proof.
\end{proof}

\section{Fano threefolds}
\label{sec:Fano}

\subsection{Curves of even genus in Fano threefolds}

The goal of~\textsection\ref{sec:Fano} is to establish the following theorem.
We prove it in~\textsection\ref{par:systeme anticanonique}
when~$-K_X$ is very ample (relying on Hodge theory)
and in~\textsection\ref{subsec:classification of fano}
when~$-K_X$ is ample but not very ample (relying on the easier half of the classification of Fano threefolds).

\begin{thm}
\label{thm:solides fano}
Let~$X$ be a smooth Fano threefold over~$\R$.  If $X(\R)=\emptyset$, then the real integral Hodge conjecture holds for~$X$.
\end{thm}

By \cite[Corollary~\ref*{BW1-cor:IHC pour solides RC ou CY} and Proposition~\ref*{BW1-prop:parity of genus in irrelevant situations}]{BW1},
the statement of Theorem~\ref{thm:solides fano} can be given a more geometric formulation:

\begin{cor}
\label{cor:elw solides fano}
Any smooth Fano threefold over~$\R$ contains a geometrically
irreducible curve of even geometric genus.
\end{cor}

Equivalently, in the notation of \cite[\textsection\ref*{BW1-subsec:elw}]{BW1}, every smooth Fano threefold $X$ over~$\R$ satisfies that $\elw_1(X)=1$ (see \cite[Corollary~\ref*{BW1-cor:what is elw1}]{BW1}). Corollary \ref{cor:elw solides fano} is nontrivial only when $X(\R)=\emptyset$, by  \cite[Proposition~\ref*{BW1-prop:parity of genus in irrelevant situations}]{BW1}.

\begin{example}
\label{ex:elw hypersurface quartique}
By Corollary~\ref{cor:elw solides fano}, any smooth quartic
threefold $X\subset \P^4_\R$
with empty real locus contains a geometrically irreducible curve of even geometric genus.
We do not know whether the minimal degree (resp.~genus) of such a curve
can be bounded independently of the quartic threefold.
(The existence of a bound on the degree
would be equivalent to the validity of Theorem~\ref{thm:solides fano} over arbitrary real closed fields; see~\textsection\ref{subsec:curves of bounded degree} and the discussion in~\textsection\ref{sss:discussion IHCR}).
We do not even know whether any such~$X$ contains a conic in~$\P^4_\R$.
For such an~$X$,
we recall that there is a canonical isomorphism $H^4_G(X(\C),\Z(2))=\Z\oplus\Z/2\Z$,
the first summand being generated by $\cl(L+\bar L)$,
where $L\subset X_\C$ denotes a line and~$\bar L$ its conjugate
(see \cite[Lemma~\ref*{BW1-lem:computation cohomology real quartic}]{BW1}).
It is the algebraicity of the $2$\nobreakdash-torsion
class which is difficult to prove.
\end{example}

\begin{example}
\label{ex:revetement double SB3}
If we denote by~$\SB^3_\R$ the nontrivial three-dimensional Severi--Brauer variety over~$\R$
(corresponding to the algebra of $2\times 2$ matrices over the quaternions)
and define the degree of
a subvariety $V \subseteq\SB^3_\R$
as the degree of~$V_\C$ in~$\P^3_\C\simeq (\SB^3_\R)_\C$,
then,
by Corollary~\ref{cor:elw solides fano},
any double cover $X\to \SB^3_\R$ ramified along a smooth quartic surface
must contain a geometrically irreducible curve of even geometric genus.
We note that~$\SB^3_\R$ does contain many surfaces of each even degree,
since~$\smash[b]{\sO_{\P^3_\C}(2)}$
descends
from~$\P^3_\C$ to a
line bundle on~$\SB^3_\R$.

For such an~$X$, there is a canonical isomorphism $H^4_G(X(\C),\Z(2))=\Z$.
Indeed, the push-forward map $H^4(X(\C),\Z(2)) \to H^4(\SB^3_\R(\C),\Z(2))=\Z$
is an isomorphism
(see \cite[Lemma~1.23]{clemensdoublesolids});
by the five lemma applied to the push-forward morphism
between the exact sequences \cite[(\ref*{BW1-eq:conoyau de la norme})]{BW1} associated with~$X$
and with~$\SB^3_\R$, we deduce that the push-forward map
$H^4_G(X(\C),\Z(2)) \to H^4_G(\SB^3_\R(\C),\Z(2))$ is an isomorphism.
On the other hand,
the natural map $\Br(\R) \to \Br(\SB^3_\R)$ vanishes
(its kernel contains the class of $\SB^3_\R$)
and is onto (see \cite[Proposition~2.1.4~(iv)]{ctsd94}), so that $\Br(\SB^3_\R)=0$ and hence (see \cite[II, Th\'eor\`eme~3.1]{grothbrauer})
\begin{align*}
H^2_G(\SB^3_\R(\C),\Q/\Z(1))=H^2_\et(\SB^3_\R,\Q/\Z(1))=\Pic(\SB^3_\R)\otimes_\Z \Q/\Z=\Q/\Z\rlap{\text{.}}
\end{align*}
It follows that $H^4_G(\SB^3_\R(\C),\Z(2))=\Z$, by
\cite[Proposition~\ref*{BW1-prop:poincare duality equivariant}, Remark~\ref*{BW1-rks:equivariant poincare}~(ii)]{BW1}.

A line of $\P^3_\C \simeq (\SB^3_\R)_\C$ which is bitangent to the ramification locus
breaks up, in~$X_\C$, into two components.
If $L \subset X_\C$ denotes one of the components
and~$\bar L$ its complex conjugate,
the isomorphism
$H^4_G(X(\C),\Z(2))=\Z$
maps $\cl(L+\bar L)$ to~$2$.
It is the algebraicity of $\frac{1}{2}\cl(L+\bar L)$,
or, equivalently, the existence of a curve in~$X$ whose push-forward to~$\SB^3_\R$
has odd degree,
which is difficult to prove.
\end{example}

\begin{rmk}
The consequences of Theorem~\ref{thm:solides fano}
spelled out in
Examples~\ref{ex:elw hypersurface quartique}
and~\ref{ex:revetement double SB3}
are, as far as we know, nontrivial existence statements.
For some families of real Fano threefolds, however,
Theorem~\ref{thm:solides fano} is vacuous
as
one can produce a curve of even genus out of the geometry.
For instance,
a double cover $X\to\P^3_\R$ ramified along a smooth quartic surface,
with $X(\R)=\emptyset$,
contains a smooth del Pezzo surface of degree~$2$ (pull
back a general plane in~$\P^3_\R$)
and therefore a geometrically irreducible curve of even geometric genus
by \cite[Corollary~\ref*{BW1-cor:surfaces}~(ii)]{BW1}.
Other examples of this phenomenon will be given in the proof of Proposition~\ref{prop:preuve Fano 23}.
\end{rmk}

\subsection{An infinitesimal criterion}
\label{subsec:an infinitesimal criterion}

Let us now turn to the proof of Theorem~\ref{thm:solides fano}.
We shall adapt the strategy of Voisin~\cite{voisinthreefolds} to our situation.
Before considering the particular case of Fano threefolds, let us work out a real analogue of
Voisin's infinitesimal criterion \cite[\textsection1, Proposition~1]{voisinthreefolds}
for the algebraicity of cohomology classes supported on an ample surface.

To formulate it, we need some notation.
Let~$X$ be a smooth, projective variety over~$\R$
and $S \subset X$ be a smooth surface.
We denote by~$F_\infty$
the $\C$\nobreakdash-linear involution
of $H^2(S(\C),\C)$
induced by the complex conjugation involution of~$S(\C)$.
As the latter is antiholomorphic,
the endomorphism~$F_\infty$
interchanges
the summands $H^{p,q}(S_{\C})$ and $H^{q,p}(S_{\C})$ of the Hodge decomposition.
In particular, it stabilises $H^{1,1}(S_{\C})$.
For any subspace $V \subset H^2(S(\C),\C)$ stable under~$F_\infty$,
we denote by $V^{F_\infty=-1}$ the eigenspace $\{v \in V;F_\infty(v)=-v\}$.
Finally,
following~\cite[\textsection1]{voisinthreefolds}, we shall consider,
for $\lambda \in H^{1,1}(S_\C)$,
the composition
\begin{equation}
\label{eq:variation infinitesimale}
H^0(S_\C,N_{S_\C/X_\C})\xrightarrow{\KS}
H^1(S_\C,T_{S_\C})\xrightarrow{\lambda}
H^2(S_\C,\sO_{S_\C})
\end{equation}
of the Kodaira--Spencer map,
which is by definition the boundary of the short exact sequence $0\to T_S \to T_X|_S \to N_{S/X} \to 0$,
of the cup product with~$\lambda \in H^1(S_\C,\Omega^1_{S_\C})$,
and of the natural map 
 $H^2(S_\C,T_{S_\C} \otimes_{\sO_{S_\C}} \Omega^1_{S_\C}) \to H^2(S_\C,\sO_{S_\C})$.

\begin{prop}
\label{prop:critere infinitesimal}
Let~$X$ be a smooth, projective and geometrically irreducible variety over~$\R$,
of dimension~$d$.
Let $S \subseteq X$ be a smooth surface
such that $H^1(S,N_{S/X})=0$.
If there exists $\lambda\in H^{1,1}(S_{\C})^{F_\infty=-1}$ such that the
composition~\eqref{eq:variation infinitesimale}
is surjective,
then the image of the Gysin map
$H^2_G(S(\C),\Z(1)) \to H^{2d-2}_G(X(\C),\Z(d-1))$
(see \cite[(\ref*{BW1-eq:pushforward equivariant complex})]{BW1} for its definition)
is contained in the image of the equivariant cycle class map
\begin{align*}
\cl:\CH_1(X) \to H^{2d-2}_G(X(\C),\Z(d-1))\rlap{.}
\end{align*}
\end{prop}

\begin{proof}
We closely follow~\cite[\textsection1]{voisinthreefolds}, keeping track of $G$-actions.

The $\C$\nobreakdash-linear involution~$F_\infty$ stabilises the subgroup $H^2(S(\C),\R)\subseteq H^2(S(\C),\C)$
and therefore defines an $\R$\nobreakdash-linear action of~$G$ on
$H^{1,1}_{\R}(S_{\C}):=H^{1,1}(S_{\C})\cap H^2(S(\C),\R)$.
Let $V=H^{1,1}_\R(S_\C)^{F_\infty=-1}$.
We have $V_\C=H^{1,1}(S_\C)^{F_\infty=-1}$
and $V(1)=H^{1,1}_\R(S_\C)(1)^G$.
As the surjectivity of~\eqref{eq:variation infinitesimale}
is a Zariski open condition on~$\lambda$
and as the two summands of the decomposition
$V_\C=V \oplus V(1)$ induced by the decomposition $\C=\R\oplus\R(1)$ are Zariski dense in~$V_\C$
(viewed as an algebraic variety over~$\C$),
there exists
$\lambda\in H^{1,1}_{\R}(S_{\C})(1)^G$ such that~\eqref{eq:variation infinitesimale}
is surjective.  We now fix such a~$\lambda$.

The hypothesis that $H^1(S,N_{S/X})=0$ implies that the Hilbert scheme~$\mathcal{H}$ of~$X$ is smooth at~$[S]$.
Let~$U$ be an analytic neighbourhood of~$[S_{\C}]$ in $\mathcal{H}(\C)$
and $\pi:\mathcal{S}_U \to U$ denote the universal family.
After shrinking~$U$,
we may assume that~$\pi$ is smooth and projective,
that the local system $\sH^2_\Z:=\RR^2\pi_*\Z/\torsion$ is trivial,
and that~$U$ is $G$\nobreakdash-invariant,
in which case~$\pi$ is a $G$\nobreakdash-equivariant morphism
between complex analytic varieties
and~$\sH^2_\Z$ is a $G$\nobreakdash-equivariant sheaf.
We set $U(\R):=U\cap \mathcal{H}(\R)=U^G$.
After further shrinking~$U$, we may assume that~$U(\R)$ is contractible.

A trivialisation of $\sH^2_{\Z}$ induces a biholomorphism $H^2\simeq U\times H^2(S(\C),\C)$, where $H^2$ is
the total space of the holomorphic vector bundle on~$U$ whose sheaf of sections is the locally free $\sO_U^{\mathrm{an}}$\nobreakdash-module
 $\sH^2_\Z \otimes_\Z \sO_U^{\mathrm{an}}$.
We fix such a trivialisation and let
\begin{align*}
\tau: H^2\to H^2(S(\C),\C)
\end{align*}
denote the second projection.
Let $H^{1,1}_{\R}\subseteq H^2_\R \subseteq H^2$
be the smooth real subbundles
whose fibres above $t\in U$ are $H^{1,1}_{\R}(\mathcal{S}_t) \subseteq H^2(\mathcal S_t,\R) \subseteq H^2(\mathcal S_t,\C)$.
(We stress that $\mathcal{S}_t=\pi^{-1}(t)$ denotes a complex analytic surface;
thus, if $t=[S_\C]$, then $\mathcal{S}_t=S(\C)$.)
Building on Green's
infinitesimal criterion for the density of Noether--Lefschetz loci
(see~\cite[Proposition~17.20]{voisinbook}),
Voisin shows in~\cite[p.~48]{voisinthreefolds},
under the assumption that~\eqref{eq:variation infinitesimale}
is surjective,
that~$\tau$ restricts to a map
\begin{align*}
\tau_{1,\R}(1):H^{1,1}_\R(1)\to H^2(S(\C),\R(1))
\end{align*}
which is submersive at~$\lambda$.

This map is equivariant with respect to the natural actions of~$G$ on its domain and target,
and~$\lambda$ is $G$\nobreakdash-invariant.
As the fixed locus~$M^G$ of a $G$\nobreakdash-action on a smooth manifold~$M$ is again a smooth manifold, with tangent space $T_x(M^G)=(T_xM)^G$ for $x\in M^G$
(see, \emph{e.g.}, \cite[Chapter~I, \textsection2.1]{Audin}),
the map
\begin{align*}
\tau_{1,\R}(1)^G:H^{1,1}_{\R}(1)^G\to H^2(S(\C),\R(1))^G
\end{align*}
induced by~$\tau_{1,\R}(1)$
is again submersive at~$\lambda$.
Here, the manifold $H^{1,1}_{\R}(1)^G$ is a smooth real vector bundle on~$U(\R)$;
its fiber over
$t\in U(\R)$ is $H^{1,1}_{\R}(\mathcal{S}_t)(1)^G$.

As this map is submersive at~$\lambda$,
its image
$\mathrm{Im}(\tau_{1,\R}(1)^G)$
contains a nonempty open subset of $H^2(S(\C),\R(1))^G$.
Being a cone, it then contains a nonempty open cone.
It follows that $H^2_G(S(\C),\Z(1))/\torsion$, which is a lattice in $H^2(S(\C),\R(1))^G$,
is generated by
$\mathrm{Im}(\tau_{1,\R}(1)^G)\cap\left(H^2_G(S(\C),\Z(1))/\torsion\right)$
(see \cite[Lemma 3]{voisinthreefolds}).

The sheaf~$\sH^2_{G,\Z(1)}$ on~$U(\R)$
associated
with the presheaf
$\Omega \mapsto H^2_G(\pi^{-1}(\Omega),\Z(1))$
is locally constant as it is the abutment of the Hochschild--Serre spectral sequence,
all of whose terms are locally constant sheaves
(see \cite[Lemme~2.1]{artinfaisceauxconstructibles});
hence, the natural map
$\sH^2_{G,\Z(1)} \to \sH^2_\R(1)^G$,
where $\sH^2_\R=\sH^2_\Z\otimes_\Z\R$,
is locally constant.
(Alternatively, this would also be a consequence of the
$G$\nobreakdash-equivariant variant of Ehresmann's theorem \cite[Lemma 4]{Dimca}.)
As~$U(\R)$ is contractible,
it follows that~$\tau_{1,\R}(1)^G$ induces a map
$$\bigoplus_{t \in U(\R)} \Hdg^2_G(\mathcal S_t,\Z(1)) \to H^2_G(S(\C),\Z(1))/\torsion\rlap{\text{,}}$$
which we have shown to be surjective,
and that it fits into a commutative diagram
\begin{align*}
\xymatrix{
\displaystyle\smash[b]{\bigoplus_{t\in U(\R)}}\CH_1(\mathcal S_t) \ar[r]^(.41)\cl \ar@<.5em>[d] &
\displaystyle\smash[b]{\bigoplus_{t\in U(\R)}}\Hdg^2_G(\mathcal S_t,\Z(1)) \ar@{->>}[r] \ar@<.5em>[d] &
H^2_G(S(\C),\Z(1))/\torsion \ar[d] \\
*!<-.5em,0ex>\entrybox{\CH_1(X)} \ar[r]^(.41)\cl &
*!<-.5em,0ex>\entrybox{H^{2d-2}_G(X(\C),\Z(d-1))} \ar[r] & H^{2d-2}_G(X(\C),\Z(d-1))/\torsion
}
\end{align*}
whose vertical arrows are the Gysin maps.
The top horizontal map on the left is surjective
by the real Lefschetz~$(1,1)$ theorem
(see \cite[Proposition~\ref*{BW1-prop:real(1,1)}]{BW1}).
This diagram shows that $H^2_G(S(\C),\Z(1))$ is generated by torsion classes and by classes
whose image in $H^{2d-2}_G(X(\C),\Z(d-1))$
is algebraic.
As torsion classes in $H^2_G(S(\C),\Z(1))$ are themselves algebraic
(\emph{loc.\ cit.}),
this completes the proof.
\end{proof}

\subsection{The anticanonical linear system}
\label{par:systeme anticanonique}

We do not know how to verify our criterion over~$\R$ in the same generality as Voisin does over~$\C$.
However, for Fano threefolds without real points, it will be possible,
in most cases, to apply it with $S\in|{-}K_X|$,
as~$S$ is then a $K3$ surface, which makes the analysis tractable.

In~\textsection\ref{par:systeme anticanonique}, we establish
Theorem~\ref{thm:solides fano}
under the assumption that~$-K_X$ is very ample.
We proceed in two steps.
First,
we apply
Proposition~\ref{prop:critere infinitesimal} to a
smooth anticanonical divisor
and deduce, in Proposition~\ref{prop:preuve Fano 1},
that Theorem~\ref{thm:solides fano}
holds as soon as~$X$ contains a good enough smooth anticanonical divisor.
We then check,
in Proposition~\ref{prop:good K3},
that
when $X(\R)=\emptyset$,
any smooth anticanonical divisor is good enough in this sense.

For smooth $S \subset X$,
we denote by $H^2_\van(S(\C),\C)\subset H^2(S(\C),\C)$
the orthogonal complement
of the image of $H^2(X(\C),\C)$
with respect to the cup product pairing.
The subspace $H_\van^{1,1}(S_\C) = H^{1,1}(S_\C) \cap H^2_\van(S(\C),\C)$
is stable under~$F_\infty$.

\begin{prop}
\label{prop:preuve Fano 1}
Let~$X$ be a smooth Fano threefold over~$\R$, with~$-K_X$ very ample.
Let $S \in |{-}K_X|$ be a smooth anticanonical divisor.
If $H_{\van}^{1,1}(S_{\C})^{F_\infty=-1}\neq 0$,
then the image of the Gysin map $H^2_G(S(\C),\Z(1)) \to H^4_G(X(\C),\Z(2))$
is contained in the image of the equivariant cycle class map
$\cl:\CH^2(X) \to H^4_G(X(\C),\Z(2))$.
If, in addition,
the Gysin map $H_1(S(\R),\Z/2\Z)\to H_1(X(\R),\Z/2\Z)$
is surjective,
then~$X$ satisfies the real integral Hodge conjecture.
\end{prop}

\begin{proof}
According to \cite[Proposition~\ref*{BW1-prop:weak lefschetz surjectivity}, Remarks~\ref*{BW1-rk:topological condition 1-cycles and ak}~(i)]{BW1},
the image of the Gysin map $H^2_G(S(\C),\Z(1)) \to H^4_G(X(\C),\Z(2))$
coincides with $H^4_G(X(\C),\Z(2))_0$
as soon as
the Gysin map
$H_1(S(\R),\Z/2\Z)\to H_1(X(\R),\Z/2\Z)$ is surjective.
In view of this remark, the second assertion results from the first one.

To prove the first assertion,
we consider
the family $\pi:\mathcal{S}\to B$
of smooth anticanonical sections of~$X_{\C}$.
Here~$B$ is a Zariski open subset of $|{-}K_{X_\C}|$,
the map~$\pi$ is smooth and projective, and its fibres are~$K3$ surfaces.
To conform with
the notation used in the proof of Proposition~\ref{prop:critere infinitesimal},
we view~$\mathcal{S}$ and~$B$ as complex analytic varieties endowed with compatible antiholomorphic actions of~$G$
and we set $B(\R)=B^G$
and $\mathcal S_b(\R)={\mathcal S_b}^G$
for $b \in B(\R)$.
We choose once and for all
a contractible open neighbourhood
$B(\R)^0\subset B(\R)$
of the point of~$B(\R)$ defined by~$S$; by Ehresmann's theorem,
the image of the Gysin map
$H_1(\mathcal S_b(\R),\Z/2\Z)\to H_1(X(\R),\Z/2\Z)$
is independent of~$b\in B(\R)^0$.
To establish Proposition~\ref{prop:preuve Fano 1}, it now suffices
to check that
Proposition~\ref{prop:critere infinitesimal} is applicable
to $\mathcal S_b \in |{-}K_X|$ for at least one $b \in B(\R)^0$.
We have
$H^1(\mathcal S_b,N_{\mathcal S_b/X_\C})=H^1(\mathcal S_b,\sO_{X_\C}(-K_{X_\C})|_{\mathcal S_b})=0$
for all $b \in B$ anyway,
by the Kodaira vanishing theorem.
It therefore suffices to prove the surjectivity of the composition
\begin{equation}
\label{eq:variation infinitesimale Sb}
H^0(\mathcal S_b,N_{\mathcal S_b/X_\C})\xrightarrow{\KS}
H^1(\mathcal S_b,T_{\mathcal S_b})\xrightarrow{\lambda}
H^2(\mathcal S_b,\sO_{\mathcal S_b})\rlap{\text{,}}
\end{equation}
for at least one $b \in B(\R)^0$
and one $\lambda\in H^{1,1}(\mathcal S_b)^{F_\infty=-1}$.
Let us argue by contradiction and suppose that
the map~\eqref{eq:variation infinitesimale Sb}
fails to be surjective
for every such~$b$ and every such~$\lambda$.
The vector space $H^2(\mathcal S_b,\sO_{\mathcal S_b})$
has dimension~$1$
since~$\mathcal S_b$ is a~$K3$ surface,
hence~\eqref{eq:variation infinitesimale Sb} in fact vanishes for all $b \in B(\R)^0$
and all~$\lambda\in H^{1,1}(\mathcal S_b)^{F_\infty=-1}$.

As the map $\pi:\mathcal{S}\to B$ is $G$\nobreakdash-equivariant,
so is the local system $\sH^2_{\R}=\RR^2\pi_*\R$.
The sub-local system $\sH^2_{\R,\van}\subset  \sH^2_\R$ of vanishing cohomology,
defined as the orthogonal complement of the image of $H^2(X(\C),\R)$
with respect to the cup product pairing, is also $G$\nobreakdash-equivariant.
We denote by
$H^2_{\van}$
the holomorphic vector bundle on~$B$ whose sheaf of sections is the locally
free $\sO_B^{\mathrm{an}}$\nobreakdash-module $\sH^2_{\R,\van} \otimes_\R \sO_B^{\mathrm{an}}$.
It comes endowed with the Gauss--Manin connection and carries a variation of Hodge structure.

For $b \in B$, let
$H^{1,1}_{\R,\van}(\mathcal S_b)=H^{1,1}(\mathcal S_b) \cap H^2_\van(\mathcal S_b,\R)$.
We consider the smooth real subbundles
$H^{1,1}_{\R,\van}(1)^G\subseteq H^2_{\R,\van}(1)^G\subseteq \left.H^2_{\van}\right|_{B(\R)}$
whose fibres above $b \in B(\R)$ are
$H_{\R,\van}^{1,1}(\mathcal{S}_{b})(1)^G\subseteq H^2_{\van}(\mathcal{S}_{b},\R(1))^G\subseteq H^2_{\van}(\mathcal{S}_b,\C)$.
The Gauss--Manin connection on~$H^2_\van$
induces a connection on $H^2_{\R,\van}(1)^G$
since the action of~$G$ on $\left.H^2_{\van}\right|_{B(\R)}$ comes from an action on the local system $\left.\sH^2_{\R,\van}\right|_{B(\R)}$.

Let $b\in B(\R)^0$.
The derivative of a local section~$s$ of $H^{1,1}_{\R,\van}(1)^G$ around~$b$
along a vector field on~$B(\R)^0$, with respect to this connection,
is a local section of $H^2_{\R,\van}(1)^G$
whose projection to $H^{0,2}(\mathcal{S}_b)$ is controlled by the
map~\eqref{eq:variation infinitesimale Sb} for $\lambda=s(b)$, by a theorem of Griffiths
(see \cite[Th\'eor\`eme 17.7]{voisinbook}), hence vanishes.  It follows
that the fibre above~$b$ of this derivative lies in the subgroup
\begin{align*}
H^2_{\van}(\mathcal{S}_{b},\R(1))^G\cap\big(H^{2,0}(\mathcal{S}_b)\oplus H^{1,1}(\mathcal{S}_b)\big)=H_{\R,\van}^{1,1}(\mathcal{S}_{b})(1)^G
\end{align*}
of $H^2_{\van}(\mathcal{S}_b,\C)$,
where the equality comes from the fact that $F_\infty$
interchanges $H^{2,0}(\mathcal{S}_b)$ and $H^{0,2}(\mathcal{S}_b)$
and stabilises $H^{1,1}(\mathcal{S}_b)$.
We have thus proved that
$\left.H^{1,1}_{\R,\van}(1)^G\right|_{B(\R)^0}$
is a flat subbundle of $\left.H^2_{\van}\right|_{B(\R)^0}$.

Let us now consider the effect on
$\left.H^{1,1}_{\R,\van}(1)^G\right|_{B(\R)^0}$,
which is a vector bundle on~$B(\R)^0$,
of parallel transport
from~$B(\R)^0$ to an arbitrary point of~$B$
with respect to the Gauss--Manin connection on the holomorphic vector bundle~$H^2$ on~$B$.

\begin{lem}
\label{lem:parallel transport}
Let $b \in B(\R)^0$
and $\alpha\in H_{\R,\van}^{1,1}(\mathcal{S}_b)(1)^G$.
For any $x \in B$ and any path~$\gamma$ from~$b$ to~$x$, the parallel transport
of~$\alpha$ along~$\gamma$ belongs to $H^{1,1}(\mathcal S_x) \subset H^2(\mathcal S_x,\C)$.
\end{lem}

\begin{proof}
Let $\rho:(\widetilde{B},\widetilde{b})\to (B,b)$
denote the universal cover of the pointed space~$(B,b)$ and let $\widetilde{B}(\R)=\rho^{-1}(B(\R))$.
Let $Z\subseteq\widetilde{B}$ denote the set of $\tilde{x} \in \widetilde{B}$
such that the parallel transport of~$\alpha$ from~$\tilde{b}$ to~$\tilde{x}$ remains of type~$(1,1)$.
This is
a well-defined locus since~$\widetilde{B}$ is simply connected.
We need only check that $Z=\widetilde{B}$.

For any $\tilde{x} \in \widetilde{B}$, the parallel transport
$\alpha(\tilde{x})\in H^2(\mathcal S_{\rho(\tilde{x})},\C)$ of~$\alpha$
from~$\tilde{b}$ to~$\tilde{x}$ belongs to the subgroup $H^2(\mathcal S_{\rho(\tilde{x})},\R(1))$,
since $\alpha\in H^2(\mathcal S_b,\R(1))$
and~$H^2$ is the complexification of the flat smooth real vector bundle on~$B$
associated with the local system~$\sH^2_\R$.
It follows that~$\alpha(\tilde{x})$ is of type~$(1,1)$ if and only if $\alpha(\tilde{x})\in F^1H^2(\mathcal S_{\rho(\tilde{x})},\C)$, where~$F$ denotes the Hodge filtration.
As $F^1 H^2 \subset H^2$ is a holomorphic subbundle,
we deduce that~$Z$ is a closed complex analytic subvariety of~$\widetilde{B}$.

On the other hand, as
$\left.H^{1,1}_{\R,\van}(1)^G\right|_{B(\R)^0}$
is a flat subbundle of $\left.H^2_{\van}\right|_{B(\R)^0}$,
the connected component
$\widetilde{B}(\R)^0$
of~$\tilde{b}$ in~$\rho^{-1}(B(\R)^0)$
is contained in~$Z$.
Let $W \subseteq Z$ be a (reduced) closed complex analytic subvariety of minimal dimension
containing~$\widetilde{B}(\R)^0$.
For any $u \in \widetilde{B}(\R)^0$, the inclusion of real vector spaces
\begin{align*}
T_u\widetilde{B}(\R) \subseteq T_u W \subseteq T_u \widetilde{B}=T_u\widetilde{B}(\R)\otimes_\R\C
\end{align*}
shows that $T_uW=T_u\widetilde{B}$,
since $T_uW$ is a complex vector space.
As a consequence, we have either $\dim(W)=\dim(\widetilde{B})$ or
$\widetilde{B}(\R)^0 \subseteq \Sing(W)$.
As~$\dim(W)$ is minimal and $\dim(\Sing(W))<\dim(W)$, it follows
that $\dim(W)=\dim(\widetilde{B})$, so that $W=\widetilde{B}$, and hence $Z=\widetilde{B}$.
\end{proof}

Let~$E$ be the smallest holomorphic flat subbundle of~$H^2_{\van}$
such that
\begin{align*}
\left.H^{1,1}_{\R,\van}(1)^G\right|_{B(\R)^0} \subseteq \left.E\right|_{B(\R)^0}\rlap{\text{.}}
\end{align*}
It follows from
Lemma~\ref{lem:parallel transport}
that~$E$
is purely of type~$(1,1)$ at every point of~$B$.
In particular $E\neq H^2_\van$.
On the other hand, the given
smooth anticanonical section~$S$ of~$X$
satisfies $H_{\van}^{1,1}(S_{\C})^{F_\infty=-1}\neq 0$
by assumption and it defines a point of~$B(\R)^0$,
so that $E\neq 0$
(note
that $H_{\van}^{1,1}(S_{\C})^{F_\infty=-1}= H^{1,1}_{\R,\van}(S_\C)(1)^G \otimes_\R \C$).
This contradicts the absolute irreducibility of the action of monodromy on vanishing cohomology
(see~\cite[Theorem~7.3.2]{Lamotke})
and concludes the proof of Proposition~\ref{prop:preuve Fano 1}.
\end{proof}

We do not know whether an anticanonical divisor satisfying all the assumptions
of Proposition~\ref{prop:preuve Fano 1} always exists.
The next statement provides a positive answer to this question
when $X(\R)=\emptyset$:
in this case, any smooth anticanonical divisor will do.

\begin{prop}
\label{prop:good K3}
Let~$X$ be a smooth Fano threefold over~$\R$.
If $S\in |{-}K_X|$ is smooth and satisfies $S(\R)=\emptyset$,
then $H_{\van}^{1,1}(S_{\C})^{F_\infty=-1}\neq 0$.
\end{prop}

\begin{proof}
Let us apply the Lefschetz fixed-point theorem to
the complex conjugation involution of~$S(\C)$,
which has no fixed point
and
acts trivially
on $H^0(S(\C),\C)=\C$ and on $H^4(S(\C),\C)=\C$.
We have $H^1(S(\C),\C)=0$
and $H^3(S(\C),\C)=0$
as~$S$ is a~$K3$ surface.
Hence $\Tr(F_\infty\mid H^2(S(\C),\C))=-2$.
As~$F_\infty$ interchanges $H^{2,0}(S_{\C})$ and~$H^{0,2}(S_{\C})$,
it follows that
$\Tr(F_\infty\mid H^{1,1}(S_\C))=-2$.
As $\dim H^{1,1}(S_\C)=20$,
we deduce that $\dim H^{1,1}(S_\C)^{F_\infty=-1}=11$.
On the other hand, the Picard rank of~$X_{\C}$ is at most~$10$
(see \cite[Theorem 2]{morimukai81classification}, \cite[Theorem~1.2]{morimukai83onfano};
this also follows from \cite[Theorem~1.1]{casagrande},
see \cite[Theorem~1.3]{dellanoce}).
Therefore $H_{\van}^{1,1}(S_\C)$
has codimension at most~$10$ in~$H^{1,1}(S_{\C})$ and intersects $H^{1,1}(S_{\C})^{F_\infty=-1}$ nontrivially.
\end{proof}

\begin{rmk}
The statement of Proposition~\ref{prop:good K3} would fail
without the assumption that $S(\R)=\emptyset$. For instance,
if~$X$ has Picard rank~$1$, then $H_{\van}^{1,1}(S_{\C})^{F_\infty=-1}=0$
if and only if the eigenvalue~$-1$ for the action of~$F_{\infty}$ on $H^2(S(\C),\Q)$ has multiplicity~$2$. As can be seen from the classification of real~$K3$ surfaces \cite[Ch.~VIII, §3]{silhol}, this happens exactly when the real locus of~$S$ is a union of~$10$ spheres.
\end{rmk}

Putting together Proposition~\ref{prop:preuve Fano 1}
and
Proposition~\ref{prop:good K3},
we have now established
Theorem~\ref{thm:solides fano}
under the assumption that~$-K_X$ is very ample.

\subsection{Classification of Fano threefolds}
\label{subsec:classification of fano}

We finish the proof of Theorem~\ref{thm:solides fano} by dealing with the case when $-K_X$ is not very ample. To do so, we need to rely, to some extent, on classification results for complex Fano threefolds. This classification roughly consists of two distinct steps: first, one studies the anticanonical linear system and classifies the (few) Fano threefolds for which it is badly behaved; then, one classifies the (many) Fano threefolds for which it is well behaved.
We shall only need the first of these two steps,
which is mainly due to Iskovskikh~\cite{iskovskikh} and whose statement is contained in
\cite[Theorem~2.1.16 and Proposition~2.4.1]{fanorusse}.

\begin{prop}[Iskovskikh]
\label{prop:classification fano}
If~$X$ is a smooth complex Fano threefold, exactly one of the following holds:
\begin{enumerate}
\item The anticanonical bundle $-K_X$ is very ample.
\item The linear system $|{-}K_X|$ is base point free and realises~$X$ as a finite double cover, the possible cases being listed in \cite[Theorem 2.1.16]{fanorusse}.
\item The base locus of $|{-}K_X|$ is either a point or a smooth rational curve.
\end{enumerate}
\end{prop}

In view of \cite[Corollary~\ref*{BW1-cor:IHC pour solides RC ou CY}]{BW1},
the following proposition completes the proof of Theorem~\ref{thm:solides fano}.
We recall that the notation~$\elw_1(X)$ was introduced in~\cite[\textsection\ref*{BW1-subsec:elw}]{BW1}.

\begin{prop}
\label{prop:preuve Fano 23}
If~$X$ is a smooth Fano threefold over~$\R$
and if~$-K_X$ is not very ample, then $\elw_1(X)=1$.
\end{prop}

\begin{proof}
If the base locus of $|{-}K_X|$ is a point or a smooth geometrically rational curve,
then~$\elw_1(X)=1$
by \cite[Corollary~\ref*{BW1-cor:what is elw1}]{BW1}.
Thus, by Proposition~\ref{prop:classification fano}, we may assume
that~$X$ falls into one of the six cases numbered~(i) to~(vi) in~\cite[Theorem~2.1.16]{fanorusse}.

In case~(i), the variety~$X_\C$ is also the one described in
\cite[Theorem~2.4.5~(i)~a)]{fanorusse}, where it is explained that there is
a unique $H \in \Pic(X_\C)$ such that $-K_{X_{\C}}=2H$, and that~$|H|$ has
a unique base point. This base point is necessarily $G$\nobreakdash-invariant, hence it is a real point of~$X$,
so that $\elw_1(X)=1$.

In cases~(ii) and~(iv), the intersection of two general members of~$|{-}K_X|$ is a smooth and geometrically irreducible
curve of genus~$2$ or~$4$, respectively.
Therefore $\elw_1(X)=1$ in these two cases,
by \cite[Corollary~\ref*{BW1-cor:what is elw1}]{BW1}.

In case (iii), the ramification divisor of the morphism induced by $|{-}K_X|$ is a smooth degree~$4$ del Pezzo surface~$S$.
Such a surface satisfies $\elw_1(S)=1$ by \cite[Corollary~\ref*{BW1-cor:surfaces}~(ii)]{BW1}
or by a theorem of Comessatti \cite[discussion below Teorema~VI, p.~60]{comessatti},
see also \cite{colliotsanspoint}. It follows that $\elw_1(X)=1$.

In case (v), the variety~$X_\C$ has Picard rank~$2$ and can be realised both
as a degree~$2$ del Pezzo fibration over~$\P^1_\C$ or as the blow-up of a smooth
elliptic curve in a Fano threefold 
whose anticanonical bundle
is very ample
(a double cover of~$\P^3_\C$ ramified
along a smooth quartic).  It follows that the two boundaries of the nef
cone of~$X_\C$ are generated by the two line bundles inducing this
fibration and this contraction, so that these morphisms are canonical and
descend to~$\R$.
Therefore~$X$ is isomorphic, over~$\R$, to the blow-up
of a smooth and geometrically irreducible curve of genus~$1$ in
a Fano threefold~$Z$ such that~$-K_Z$ is very ample.
We have already proved,
in~\textsection\ref{par:systeme anticanonique},
the validity of Theorem~\ref{thm:solides fano} and Corollary~\ref{cor:elw solides fano}
for~$Z$; hence $\elw_1(Z)=1$.
A geometrically irreducible curve of even geometric genus in~$Z$ cannot be contained in
the centre of the blow-up since the latter is a smooth and geometrically irreducible curve of genus~$1$.
By considering its strict
transform in~$X$, we deduce that $\elw_1(X)=1$.

In case (vi), we have $X_{\C}\simeq \P^1_{\C}\times S_{\C}$ for a complex
degree~$2$ del Pezzo surface~$S_{\C}$. The line bundle $-K_{X_{\C}}$
induces a double cover $X_{\C}\to\P^1_{\C}\times\P^2_{\C}$.
Arguing as above,
the two projections
$X_{\C}\to\P^1_\C$ and $X_\C\to \P^2_{\C}$
are canonical,
and so is the Stein factorisation $X_{\C}\to S_{\C}$ of the second one,
so that
these three morphisms must
descend to~$\R$.  Therefore $X\simeq B\times S$ for a geometrically rational curve~$B$
and a degree~$2$ del Pezzo surface~$S$ over~$\R$.
If~$S$ contains a real point~$s$,
then~$B \times S$ contains the geometrically rational curve $B \times \{s\}$
and hence $\elw_1(X)=1$.
Otherwise,
by the theorem of Comessatti quoted above,
the surface~$S$ contains a geometrically rational curve.
Its inverse image in~$X$ is a geometrically rational surface,
which, by the same theorem again, must contain a geometrically rational curve,
so that $\elw_1(X)=1$.
\end{proof}


\begin{rmk}
It is likely that many of the cases in which~$|{-}K_X|$ induces a double
cover could have been handled by a variant of the strategy
of~\S\ref{par:systeme anticanonique}. However, note that the theorem on the
irreducibility of the monodromy action~\cite[Theorem~7.3.2]{Lamotke} used
at the very end of the proof of Proposition \ref{prop:preuve Fano 1} fails
for such~$X$: the vanishing cohomology splits as the sum of its invariant
and anti-invariant parts with respect to the involution given by the
anticanonical double cover.
\end{rmk}

\section{Del Pezzo fibrations}
\label{sec:dPfibrations}

\subsection{Main result}

In this section, we study del Pezzo fibrations $f:X\to B$ over a curve $B$ over a real closed field $R$. 
Our main goal  is the following statement.

\begin{thm}
\label{thm:solides fibres en del Pezzo}
Let $f:X\to B$ be a morphism between smooth proper connected varieties over a real closed field $R$ such that $B$ is a curve and the generic fiber of $f$ is a del Pezzo surface of degree $\delta$.
The cycle class map
$\cl:\CH_1(X)\to H^4_G(X(C),\Z(2))_0$ is surjective
 in each of the following cases:
\begin{enumerate}[(i)]
\item $\delta\geq 5$,
\item $X(R)=\emptyset$ and $B(R)\neq\emptyset$,
\item $\delta\in\{1,3\}$ and the real locus of each smooth real fiber of $f$ has exactly
 one connected component,
\item $R=\R$ and $\delta=3$.
\end{enumerate}
\end{thm}

In the above statement, the surjectivity of $\cl:\CH_1(X)\to H^4_G(X(C),\Z(2))_0$ is nothing but
 the real integral Hodge conjecture for $1$-cycles on $X$ in the sense of \cite[Definition~\ref*{BW1-def:ihc real closed field}]{BW1}, since $H^2(X,\sO_X)=0$ (see Lemma \ref{lem:Hi0 Pictors}).

Our proof of Theorem \ref{thm:solides fibres en del Pezzo} is uniform, except in case (iv) where we rely on a specific trick explained in \S\ref{subsec:cubicfibrations}.
   To prove all other cases of Theorem \ref{thm:solides fibres en del Pezzo}, we use in a geometric way the fibration structure. We separately construct curves in the fibers of $f$ in \S\ref{subsec:vertical} (adapting to the real situation a strategy due to Esnault and the second-named author \cite{ewzcl} over a separably closed field), and (multi-)sections of $f$ in \S\ref{subsec:horizontal}. We show that their classes generate the whole of $H^4_G(X(C),\Z(2))_0$ using a variant of the Leray spectral sequence, studied in \S\ref{subsec:joliLeray}.

\begin{example}\label{exeasydP}
  Theorem \ref{thm:solides fibres en del Pezzo} is easy if $\delta\in\{5,7,9\}$. Indeed, the generic fiber $X_{R(B)}$ is then rational over $R(B)$
(see \cite[Theorem~3.7 and Theorem~3.15]{maninrational} and \cite{Swinnerton-Dyer}, and note that $\Br(R(B))[3]=0$ since $\Br(C(B))[3]=0$ by Tsen's theorem), so that
$X$ is birational to $\P^2_R\times B$, and one may use \cite[Proposition \ref*{BW1-prop:birinvIHC}]{BW1}.
\end{example}

The first case to which Theorem \ref{thm:solides fibres en del Pezzo} does not apply, that of degree $4$ del Pezzo fibrations, is related to an old conjecture of Lang, as we now explain.

\begin{prop}
Let $f:X\to \Gamma$ be a degree $4$ del Pezzo fibration over the anisotropic conic $\Gamma$ over $R$.
Then the following assertions are equivalent:
\begin{enumerate}[(i)]
\item $X$ satisfies the real integral Hodge conjecture for $1$-cycles,
\item $X$ contains a geometrically integral curve of even geometric genus,
\item $X$ contains a geometrically rational curve,
\item $f$ has a section,
\item $X_{R(\Gamma)}$ has a rational point.
\end{enumerate}
\end{prop}

\begin{proof}
The equivalence between (i) and (ii) is \cite[Corollary~\ref*{BW1-cor:IHC pour solides RC ou CY}]{BW1}.
Suppose that (ii) holds, and let $\Delta\to X$ be a smooth projective geometrically integral curve of even genus. 
By \cite[Proposition~\ref*{BW1-prop:genusundermorphism}]{BW1}, the degree of $\Delta\to\Gamma$ is odd.
This implies that $X_{R(\Gamma)}$ has a $0$-cycle of odd degree. By a theorem of Coray \cite{coray}, it has a rational point, proving (v). The implications (v)$\Rightarrow$(iv)$\Rightarrow$(iii)$\Rightarrow$(ii) are trivial.
\end{proof}
 
In the formulation (v), the validity of the above assertions follows from the conjecture
of Lang \cite[p.~379]{langrealplaces} according to which the function field of a curve over $R$ without rational points is a $C_1$ field (see \cite[Theorem~1 and Theorem~3]{langquasialg}).
This particular case of Lang's conjecture has been proven in \cite[Theorem 0.13]{pi} if $R=\R$; it is still open if $R$ is a non-archimedean real closed field.

\subsection{Applications}
When Theorem \ref{thm:solides fibres en del Pezzo} applies,
one can combine it with the next proposition,
which builds on the duality theorem \cite[Theorem~\ref*{BW1-th:image psi}]{BW1}
and on Theorem~\ref{thm:AK},
to compute $H_1^{\alg}(X(R),\Z/2\Z)$ and to deduce approximation results.

\begin{prop}
\label{prop:application of ihc for dp fibration}
Let $f:X\to B$ be a morphism of smooth projective connected varieties over a real closed field~$R$ with irreducible geometric generic fiber $F$.
Assume that $B$ is a curve, that $\Pic(F)[2]=0$, that $H^i(F,\sO_F)=0$ for $i>0$, 
and that the cycle class map $\cl:\CH_1(X)\to H^4_G(X(C),\Z(2))_0$ is surjective.
\begin{enumerate}[(i)]
\item
The group $H_1^{\alg}(X(R),\Z/2\Z)$ is:
$$\{\alpha\in H_1(X(R),\Z/2\Z)\mid f_*\alpha\in H_1(B(R),\Z/2\Z)\textrm{ is a multiple of }\cl_{R}(B) \}.$$
\item
If $R=\R$ and $B(\R)$ is connected, any one-dimensional compact $\ci$ submanifold $Z\subset X(\R)$ has an algebraic approximation in $X(\R)$ in the sense of Definition~\ref{def:approximation}.
\end{enumerate} 
\end{prop}

\begin{proof}
Since (ii) is a consequence of (i) and Theorem \ref{thm:AK}, we need only prove (i).
We may assume that $X(R)\neq\varnothing$. 
By \cite[Theorems~\ref*{BW1-th:image psi} and~\ref*{BW1-th:phi}]{BW1}, $H_1^{\alg}(X(R),\Z/2\Z)$ is the subgroup of $H_1(X(R),\Z/2\Z)$ consisting of the classes orthogonal to the image of the Borel--Haefliger map $\cl_{R}:\Pic(X_C)^G[2^{\infty}]=\Pic(X)[2^\infty]\to H^1(X(R),\Z/2\Z)$.
Now the pull-back map
$\Pic(B)[2^\infty]=\Pic(B_C)^G[2^{\infty}]\to \Pic(X_C)^G[2^{\infty}]=\Pic(X)[2^\infty]$
is an isomorphism by Lemma~\ref{lem:Hi0 Pictors}.
The projection formula \cite[(\ref*{BW1-eq:projection formula})]{BW1}
 then identifies $H_1^{\alg}(X(R),\Z/2\Z)$ with the set of $\alpha\in H_1(X(R),\Z/2\Z)$ such that $f_*\alpha\in H_1(B(R),\Z/2\Z)$ is orthogonal to $\cl_{R}(\Pic(B)[2^{\infty}])$.
Applying \cite[Theorem~\ref*{BW1-th:image psi}]{BW1} to $B$ concludes the proof.
\end{proof}

\subsection{Cubic surface fibrations}
\label{subsec:cubicfibrations}
We now start to prove Theorem \ref{thm:solides fibres en del Pezzo}.

\begin{prop}
\label{prop:cubicsurfacefibrations}
Theorem \ref{thm:solides fibres en del Pezzo} holds when $R=\R$ and $\delta=3$.
\end{prop}

\begin{proof}
Consider the generic fiber $X_{\R(B)}$ of $f$: it is a smooth cubic surface over $\R(B)$. Its $27$ lines are permuted by the Galois group of $\R(B)$. Since $27$ is odd, this action has an orbit of odd cardinality $\nu$. Let $D\to B$ be the corresponding degree $\nu$ morphism of smooth projective curves over $\R$. By construction, $X_{\R(D)}$ contains a line. Projecting from this line, one sees that $X_{\R(D)}$ is birational to a conic bundle over $\P^1_{\R(D)}$. Let $Y$ be a smooth projective model of $X_{\R(D)}$ dominating $X$ and $\pi:Y\to X$ be the corresponding generically finite morphism of degree $\nu$. The variety $Y$ is birational to a conic bundle over $\P^1_{\R}\times D$. By Theorem \ref{thm:fibres en coniques} and \cite[Propositions~\ref*{BW1-prop:real(1,1)} and~\ref*{BW1-prop:birinvIHC}]{BW1}, $Y$ satisfies the real integral Hodge conjecture for $1$-cycles, and therefore $\cl:\CH_1(Y)\to H^4_G(Y(\C),\Z(2))_0$ is surjective as $H^2(Y,\sO_Y)=0$ (see Lemma \ref{lem:Hi0 Pictors}).
The diagram
\begin{align*}
\xymatrix@R=3ex{
\CH_1(X)\ar[r]^{\pi^*}\ar[d]^{\cl} & \CH_1(Y) \ar[r]^{\pi_*} \ar@{->>}[d]^{\cl} & \CH_1(X)\ar[d]^{\cl}\\
H^4_G(X(C),\Z(2))_0 \ar[r]^{\pi^*} &H^4_G(Y(C),\Z(2))_0\ar[r]^{\pi_*}& H^4_G(X(C),\Z(2))_0,
}
\end{align*}
where the vertical maps are the equivariant cycle class maps \cite[\S\ref*{BW1-subsubsec:eqcl} and Theorem~\ref*{BW1-th:conditions de krasnov}]{BW1} and where the push-forward map $\pi_*$ in equivariant cohomology is defined in \cite[(\ref*{BW1-eq:pushforward equivariant complex}) and Theorem~\ref*{BW1-th:stability of topological constraints}]{BW1}, is commutative by \cite[\S \ref*{BW1-subsubsec:eqcl}]{BW1}.
 The projection formula \cite[(\ref*{BW1-eq:projection formula equivariant})]{BW1} shows that the composition of the horizontal arrows of the bottom row is multiplication by $\nu$. It follows that the cokernel $K$ of the cycle class map  $\cl:\CH_1(X)\to H^4_G(X(\C),\Z(2))_0$ is killed by $\nu$.

By Voisin's theorem \cite[Theorem 2]{voisinthreefolds}, $X_{\C}$ satisfies the integral Hodge conjecture for $1$-cycles.
 A norm argument for $X_{\C}\to X$, similar to the one explained above, shows that $K$ is $2$-torsion. Since $\gcd(\nu,2)=1$, one has $K=0$, as required.
\end{proof}

\subsection{Vertical curves}
\label{subsec:vertical}

We construct here, following \cite{ewzcl}, many curves in the fibers of some surface fibrations over a curve over $R$, including del Pezzo fibrations.

\begin{prop}
\label{prop:curvesinfibers}
Let $f:X\to B$ be a proper flat morphism of smooth varieties over a real closed field~$R$.
Assume that $B$ is a curve
and that the geometric generic fiber of $f$ is a connected surface with geometric genus zero.
Let $t\in B(R)$ and $S=f^{-1}(t)$.

Then the image of the natural map $H^4_{G,S(C)}(X(C),\Z(2))\to H^4_G(X(C),\Z(2))$ is contained in the image of the cycle class map $\cl:\CH_1(X)\to H^4_G(X(C),\Z(2))$.
\end{prop}

We prove this proposition in~\textsection\ref{subsubsec:proofofcurvesinfibers}.
Before doing so, we shall need two lemmas.

\subsubsection{Two lemmas}
\label{subsubsec:afewlemmas}

In~\S\ref{subsubsec:afewlemmas}, we fix $f:X\to B$ and $t\in B(R)$ as in Proposition~\ref{prop:curvesinfibers} and we assume, in addition, that $S_\red$ is a divisor with simple normal
crossings on~$X$.
Let $(S_i)_{i\in I}$ be its irreducible components.
 The following lemma is essentially \cite[Lemmas 4.4, 4.5]{ewzcl}. We give a short proof in our context, already suggested in \emph{loc.\ cit}.

\begin{lem}\label{vanishingH2}
The groups $H^2(S_i,\sO_{S_i})$ and $H^2(S_{\red},\sO_{S_{\red}})$ all vanish.
\end{lem}

\begin{proof}
Replacing $B$ by an open subset, we may assume that $t\in B$ is defined by the vanishing of a global section $s\in H^0(B,\sO_B)$.
  The coherent sheaves $\RR^if_*\sO_{X}$ are torsion-free
by du Bois and Jarraud \cite{duboisjarraud}
(see \cite[Step~6 p.~20]{kollartorsionfree}).
Since  $\RR^2 f_*\sO_{X}$ and $\RR^3 f_*\sO_{X}$ are moreover torsion (by hypothesis for the first one, by cohomological dimension for the second), these sheaves vanish.
The short exact sequence $0\to \sO_{X}\xrightarrow{s}
\sO_{X}\to \sO_S\to 0$ implies that $\RR^2 f_*\sO_S=0$, that is $H^2(S,\sO_S)=0$.
  Since both $\sO_{S_i}$ and $\sO_{S_{\red}}$ are quotients of $\sO_S$ and since $S$ has coherent cohomological dimension $2$, it follows that 
$H^2(S_i,\sO_{S_i})=0$ and $H^2(S_{\red},\sO_{S_{\red}})=0$, as desired.
\end{proof}

The next lemma originates from \cite[\S\S 4.2--4.3]{ewzcl}.

\begin{lem}
\label{lem:cohofibrespeciale}
The morphism $\rho:\bigoplus_{i\in I} H^4_{G,S_i(C)}(X(C),\Z(2))\to H^4_{G,S(C)}(X(C),\Z(2))$ induced by the inclusions $S_i(C)\subset S(C)$ is surjective.
\end{lem}

\begin{proof}
Order the finite set $I$. For $n\geq 0$ and $i_1<\dots<i_{n+1}\in I$, we define $S_{i_1\mkern-1mu,\mkern1mu\dots,\mkern1mui_{n+1}}:=S_{i_1}\cap\dots \cap S_{i_{n+1}}$.
For $n \geq 0$, we set
$S^{(n)}:=\coprod_{i_1<\dots<i_{n+1}\in I} S_{i_1,\dots,i_{n+1}}$
and define $a_n:S^{(n)}\to S$ to be the natural morphism. That $S_\red$ has simple normal crossings shows that $S^{(n)}$ is smooth of pure dimension $2-n$, and the ordering on $I$ gives rise to an exact sequence of $G$-equivariant sheaves on $S(C)$
analogous to \cite[(3.7)]{ewzcl}:
\begin{equation}
\label{MayerVietorissheaves}
0\to \Z\to a_{0*}\Z\to a_{1*}\Z\to a_{2*}\Z\to 0.
\end{equation}
Define $\mathcal{F}$ to be the cokernel of $\Z\to a_{0*}\Z$, split~(\ref{MayerVietorissheaves}) into two short exact sequences,
let $\iota:S \hookrightarrow X$ denote the inclusion
and apply the derived functors of $\Hom_G(\iota_*-,\Z(2))$ to get two exact sequences:
\begin{equation}
\label{shortMV}
\bigoplus_i H^4_{G,S_i(C)}(X(C),\Z(2))\xrightarrow{\rho} H^4_{G,S(C)}(X(C),\Z(2))\to \Ext^5_G(\iota_*\mathcal{F},\Z(2)),
\end{equation}
\begin{equation*}
0\to \bigoplus_{i<j} H^5_{G,S_{ij}(C)}(X(C),\Z(2))\to \Ext^5_G(\iota_*\mathcal{F},\Z(2))\to \bigoplus_{i<j<k} H^6_{G,S_{ijk}(C)}(X(C),\Z(2)),
\end{equation*}
where $H^5_{G,S_{ijk}(C)}(X(C),\Z(2))=0$ for $i<j<k$ by purity \cite[(\ref*{BW1-eq:equivariant purity subvariety})]{BW1}.
Purity also shows that $H^6_{G,S_{ijk}(C)}(X(C),\Z(2))\simeq H^0_G(S_{ijk}(C),\Z(-1))$, which is torsion-free, and that $H^5_{G,S_{ij}(C)}(X(C),\Z(2))\simeq H^1_G(S_{ij}(C),\Z)$.
The latter group is torsion-free as well: indeed, the Hochschild--Serre spectral sequence \cite[(\ref*{BW1-eq:hochschild-serre})]{BW1} exhibits it as an extension of a subgroup of $H^1(S_{ij}(C),\Z)$ by $H^1(G, H^0(S_{ij}(C),\Z))$. On the one hand,  $H^1(S_{ij}(C),\Z)$ is torsion-free, as the long exact sequences of cohomology associated with $0\to\Z\xrightarrow{n}\Z\to\Z/n\Z\to 0$ show.
 On the other hand, the $G$\nobreakdash-module $H^0(S_{ij}(C),\Z)$ is isomorphic to $\Z^a \oplus \Z[G]^b$ for some $a,b\geq 0$ (that correspond respectively to the connected components of $S_{ij}$ that are geometrically connected, and to those that are not),
so that $H^1(G, H^0(S_{ij}(C),\Z))=0$. The exact sequences (\ref{shortMV}) now show that $\Ext^5_G(\iota_*\mathcal{F},\Z(2))$, hence also the cokernel of $\rho$, is torsion free.

It remains to show that the cokernel of $\rho$ is torsion, or equivalently that the morphism  $\rho_{\Q}:\bigoplus_{i\in I} H^4_{G,S_i(C)}(X(C),\Q(2))\to H^4_{G,S(C)}(X(C),\Q(2))$ is surjective.
By the Hochschild--Serre spectral sequence, it suffices to prove that the morphism
$\overline{\rho}_{\Q}:\bigoplus_{i\in I} H^4_{S_i(C)}(X(C),\Q(2))\to H^4_{S(C)}(X(C),\Q(2))$ is surjective. 
To check this, we may assume that $C=\C$, by the Lefschetz principle.
By \cite[Corollary~6.28]{peterssteenbrinkbook},
the dual of $\overline{\rho}_{\Q}$ identifies, up to a twist, with the restriction map
$$\overline{\rho}_{\Q}^{\vee}: H^{2}(S(\C),\Q)\to \bigoplus_{i\in I} H^{2}(S_i(\C),\Q)\rlap{.}$$
To prove that $\overline{\rho}_{\Q}^{\vee}$ is injective, we use Deligne's mixed Hodge theory.
By \cite[Proposition~8.2.5 and Th\'eor\`eme~8.2.4]{HodgeIII},
the Hodge numbers of the mixed Hodge structure
$N=\Ker(\overline{\rho}_{\Q}^{\vee})$
all vanish except perhaps $h^{0,0}$, $h^{0,1}$, $h^{1,0}$.
On the other hand, it follows from the definition of the Hodge filtration on $H^2(S(\C),\C)$
that $\gr^0_F\mkern2muH^2(S(\C),\C)=H^2(S_\red,\sO_{S_\red})$
(see, \emph{e.g.}, its description in \cite[Chapter~4, \textsection2.4]{AGIII}),
which vanishes by Lemma~\ref{vanishingH2}.
As~$\gr^0_F$ is exact (see \cite[Th\'eor\`eme~2.3.5]{HodgeII}), we deduce that $\gr^0_F\mkern2muN_\C=0$,
hence $h^{0,0}=h^{0,1}=0$.  As $h^{1,0}=h^{0,1}$, it follows that all Hodge numbers of~$N$ vanish, so that
$N=0$, which completes the proof.
\end{proof}

\subsubsection{Proof of Proposition~\ref{prop:curvesinfibers}}
\label{subsubsec:proofofcurvesinfibers}

By Hironaka's theorem on embedded resolution of singularities, there exists
a proper birational morphism $\pi:X'\to X$, with~$X'$ smooth, such that $(f\circ \pi)^{-1}(t)_\red$ is a simple normal crossings divisor on~$X'$.  As $\pi_*\circ\pi^*$ is the identity of $H^4_G(X(C),\Z(2))$
(a consequence of the equivariant projection formula \cite[(\ref*{BW1-eq:projection formula equivariant})]{BW1}), we may
replace~$X$ with~$X'$ and assume that~$S_\red$ is a simple normal crossings divisor on~$X$.
By Lemma \ref{lem:cohofibrespeciale},
it now suffices to show that
the image of $H^4_{G,S_i(C)}(X(C),\Z(2))\to H^4_G(X(C),\Z(2))$ is contained, for all $i\in I$, in the image of the cycle class map $\cl:\CH_1(X)\to H^4_G(X(C),\Z(2))$. Consider the diagram
$$
\xymatrix{
\CH_1(S_i)\ar^{\cl}[d]\ar[rr] && \CH_1(X) \ar^{\cl}[d] \\
H^2_{G}(S_i(C),\Z(1))\ar[r]^(0.45){\sim}& H^4_{G,S_i(C)}(X(C),\Z(2))\ar[r]& H^4_G(X(C),\Z(2)),
}
$$
whose vertical arrows are the equivariant cycle class maps, whose bottom right horizontal arrow is the purity isomorphism \cite[(\ref*{BW1-eq:equivariant purity subvariety})]{BW1}, and whose top horizontal arrow is the push-forward map.
It is commutative by \cite[\S\ref*{BW1-subsubsec:eqcl}]{BW1}.
By \cite[Proposition~\ref*{BW1-prop:real(1,1)nonarch}]{BW1} and Lemma~\ref{vanishingH2}, the left vertical arrow is surjective. This proves the proposition.

\subsection{Exploiting a Leray spectral sequence}
\label{subsec:joliLeray}
In \S \ref{subsec:joliLeray}, we introduce and study
a variant of the Leray spectral sequence that is adapted to the study of fibrations in the real context. Applied to the map $f(C):X(C)\to B(C)$ associated with a morphism $f:X\to B$ of varieties over $R$, it is very closely related to (and inspired by) the Leray spectral sequence for the natural morphism from the \'etale site of $X$ to the $b$\nobreakdash-site of $B$ defined by Scheiderer in \cite[Section~2]{scheiderer}.

\subsubsection{Two Leray spectral sequences}
\label{subsubsec:two leray}

For the whole of~\textsection\ref{subsubsec:two leray},
let us fix a real closed field~$R$,
a $G$\nobreakdash-equivariant morphism $f:V\to W$
between locally complete semi-algebraic spaces over~$R$ endowed with a continuous semi-algebraic action
of~$G$, and a $G$\nobreakdash-equivariant sheaf~$\sF$ of abelian groups on~$V$.
In order to compute the equivariant cohomology group $H^n_G(V,\sF)$,
one could think
of two Leray spectral sequences.
The more straightforward one has $H^p_G(W,\RR^qf_*\sF)$ as its $E_2^{p,q}$ term.
The other one, which for the purposes of the proof of
Theorem~\ref{thm:solides fibres en del Pezzo}
is significantly more convenient, reads
\begin{align}
\label{eq:jolileray}
E_2^{p,q} = H^p(W/G,\RR^q f^G_*\sF) \Rightarrow H^{p+q}_G(V,\sF)\rlap{.}
\end{align}
Here~$f^G_*$ is the functor, from the category of $G$\nobreakdash-equivariant sheaves of abelian
groups on~$V$ to the category of sheaves of abelian groups on the quotient semi-algebraic space~$W/G$,
which takes a sheaf~$\sF$ to $(\pi_*f_*\sF)^G$,
where $\pi:W\to W/G$ denotes the quotient map
(see \cite[Corollary~1.6]{brumfielquotient}, \cite[Proposition~3.5]{delfsknebuschsemialgretractions}).
This left exact functor admits an exact left adjoint;
hence the spectral sequence~\eqref{eq:jolileray} results from \cite[Proposition~2.3.10, Theorem~5.8.3]{weibelbook}.

\begin{prop}
\label{prop:joli leray general properties}
Let~$f$ and~$\sF$ be as above.  Let $q\geq 0$ be an integer.
\begin{itemize}
\item[(i)]
The sheaf on~$W/G$ associated with the presheaf
\begin{align*}
U \mapsto H^q_G\big(f^{-1}(\pi^{-1}(U)),\sF|_{f^{-1}(\pi^{-1}(U))}\big)
\end{align*}
is canonically isomorphic to $\RR^qf^G_*\sF$.

\item[(ii)]
Viewing $\sF[G]=\sF \otimes_\Z \Z[G]$ as a $G$\nobreakdash-equivariant sheaf
with the diagonal action of~$G$, there is a canonical isomorphism
$\RR^qf^G_*(\sF[G])=\pi_*\RR^qf_*\sF$.

\item[(iii)]
The cokernel of the map
$\RR^qf^G_*(\sF[G]) \to \RR^qf^G_*\sF$ induced by the norm map $\sF[G]\to \sF$
is supported on $W^G \subseteq W/G$.

\item[(iv)]
If~$f$ is proper, the formation of $\RR^qf^G_*\sF$ commutes with arbitrary base change, in the following sense.
Let~$Z$ be a locally complete semi-algebraic space.
Let $\rho:Z \to W/G$ be a morphism.
Let $W' = W \times_{W/G} Z$.
Let $V' = V \times_W W'$.
Let $\rho_V:V'\to V$ and $f':V'\to W'$ denote the projections.
If~$f$ is proper, the base change morphism
$\rho^*\RR^qf^G_*\sF\to \RR^qf'^G_*(\rho_V^*\sF)$
is an isomorphism.

\item[(v)]
If~$f$ is proper, the stalks of $\RR^qf^G_*\sF$ can be computed as follows.
Let $w \in W$.
If $w \in W^G$, there is a canonical isomorphism
$$
(\RR^qf^G_*\sF)_{\pi(w)} = H^q_G\big(f^{-1}(w),\sF|_{f^{-1}(w)}\big)\rlap{.}
$$
If $w \notin W^G$, there is a canonical isomorphism
$$
(\RR^qf^G_*\sF)_{\pi(w)} = H^q\big(f^{-1}(w),\sF|_{f^{-1}(w)}\big)\rlap{.}
$$

\item[(vi)]
Let~$X$ and~$Y$ be algebraic varieties over~$R$ endowed with an action of~$G$
and let $\phi:X\to Y$ be a smooth and proper $G$\nobreakdash-equivariant
morphism.
Assume that~$G$ acts trivially on~$Y$,
that $V=X(R)$ and $W=Y(R)$
and that $f=\phi(R)$.
If~$\sF$ is locally constant with finitely generated stalks,
then so is $\RR^qf^G_*\sF$.

\item[(vii)]
Let $\phi:X\to Y$ be a smooth and projective morphism
between varieties over~$R$.
Assume that $V=X(C)$ and $W=Y(C)$ and that $f=\phi(C)$.
If~$\sF$ is locally constant with finitely generated stalks,
then so are the restrictions of $\RR^qf^G_*\sF$
to $Y(R) \subseteq W/G$ and to $(Y(C)\setminus Y(R))/G \subseteq W/G$.
\end{itemize}
\end{prop}

\begin{proof}
For~(i), see \cite[Lemme~3.7.2]{tohoku}.
Assertion~(ii) follows from \cite[(\ref*{BW1-eq:cohoeqnoneq})]{BW1}
and from~(i).
For~(iii),
applying \cite[Theorem~2.1]{delfsknebuschonthehomology} to obtain a triangulation
and considering the star neighbourhoods of the open simplices,
we see that we can cover the semi-algebraic space $(W\setminus W^G)/G$
by finitely many connected and simply connected open semi-algebraic subsets.
After shrinking~$W$, we may therefore assume that $W=(W/G) \times G$
and that~$\pi$ is the first projection.
In this case, the norm map $\sF[G]\to \sF$ admits a $G$\nobreakdash-equivariant section,
so that the map
$\RR^qf^G_*(\sF[G]) \to \RR^qf^G_*\sF$
admits a section
and is hence surjective.
For~(iv), we note that $\RR f^G_* = \RR (\mathrm{Id}_W)_*^G\circ \RR f_*$
and $\RR f'^G_* = \RR (\mathrm{Id}_{W'})_*^G\circ \RR f'_*$
and that~$\RR f_*$ satisfies the proper base change theorem,
by \cite[Chapter~II, Theorem~7.8]{delfshomology}.
This reduces us to checking~(iv) when~$f$ is the identity map,
in which case it is an exercise.
Assertion~(v) follows from~(iv).
For~(vi),
let us consider the natural spectral sequence
\begin{align*}
E_2^{a,b}=\sH^a(G,\RR^b f_*\sF) \Rightarrow \RR^{a+b} f^G_* \sF\rlap{,}
\end{align*}
where, for a $G$\nobreakdash-equivariant sheaf~$\sG$ on~$Y(R)$,
we denote by $\sH^a(G,\sG)$ the sheaf, on~$Y(R)$,
associated with the presheaf $U \mapsto H^a(G,\sG(U))$.
Applying \cite[Corollary~17.20]{scheiderer} to~$\phi$,
we see that the sheaves $\RR^b f_*\sF$ are locally constant on~$Y(R)$, with finitely generated stalks.
It follows that $\sH^a(G,\RR^b f_*\sF)$ is locally constant on~$Y(R)$, with finitely generated stalks,
for all~$a$ and~$b$; hence, so is the abutment $\RR^q f^G_*\sF$.
Finally, to check~(vii), we may assume that~$Y$ is affine.
In this case, the assertion on~$Y(R)$ follows from~(vi) applied to
the second projection $R_{C/R} X_C \times_{R_{C/R} Y_C} Y \to Y$,
where~$R_{C/R}$ denotes the Weil restriction from~$C$ to~$R$
of a quasi-projective variety over~$C$,
and from~(iv),
while the assertion on $(Y(C)\setminus Y(R))/G$ follows from
 \cite[Corollary~17.20]{scheiderer}
applied to $R_{C/R}\phi:R_{C/R}X_C \to R_{C/R}Y_C$.
\end{proof}

\begin{rmk}
\label{rmk:warning stalks}
To prevent misuse of Proposition~\ref{prop:joli leray general properties}~(v),
we remind the reader that a sheaf on a semi-algebraic space need not vanish even if its stalks
at the points of the underlying set
do (see \cite[Chapter~I, Example~1.7]{delfshomology}).
This phenomenon does not occur for sheaves
which become constant on a set-theoretic cover by finitely many locally closed semi-algebraic subsets
(\emph{e.g.}, for locally constant sheaves, see \cite[Chapter~II, Definition~1]{delfshomology}).
\end{rmk}

\subsubsection{Application}

We resume the proof of Theorem~\ref{thm:solides fibres en del Pezzo}.

\begin{prop}\label{prop:Leraybizarre}
Let~$B$ be a smooth, geometrically connected, affine curve over a real closed field~$R$.
Let $f:X\to B$ be a smooth and projective morphism
whose fibre over any $b \in B(C)$ satisfies $H^3(X_b(C),\Z)=0$.
Assume that all semi-algebraic connected components of~$B(R)$ are semi-algebraically
contractible.
Then the product of restriction maps
\begin{align*}
H^4_G(X(C),\Z(2)) \to \prod_{b \in B(R)} H^4_G(X_b(C),\Z(2)) \times \prod_{b \in B(C)} H^4(X_b(C),\Z(2))
\end{align*}
is injective.
\end{prop}

\begin{proof}
The spectral sequence~\eqref{eq:jolileray} for $f:X(C)\to B(C)$
takes the shape
\begin{align*}
E_2^{p,q}=H^p(B(C)/G,\RR^q f^G_* \Z(2)) \Rightarrow H^{p+q}_G(X(C),\Z(2))\rlap{.}
\end{align*}
By Proposition~\ref{prop:joli leray general properties}~(v) and~(vii)
and Remark~\ref{rmk:warning stalks},
the desired statement is equivalent to the injectivity of
the edge homomorphism
$H^4_G(X(C),\Z(2)) \to E_2^{0,4}$.
For this, it is enough to check that $E_2^{1,3}=0$ and that $E_2^{p,q}=0$ for all $p,q$ such that $p\geq 2$.

By Proposition~\ref{prop:joli leray general properties}~(v) and~(vii) and by our assumption
on the fibres of~$f$, the sheaf $\RR^3 f^G_* \Z(2)$ is supported on $B(R) \subseteq B(C)/G$
and is locally constant on~$B(R)$.
On the other hand, as~$B(R)$ is a disjoint union of contractible semi-algebraic spaces,
any locally constant sheaf on~$B(R)$ must be constant and have trivial cohomology in positive
degrees (see~\cite[Chapter~II, Theorem~2.5]{delfshomology}).  Hence $E_2^{1,3}=0$.

Let $p\geq 2$.
Let $\pi:B(C)\to B(C)/G$ denote the quotient map.
Proposition~\ref{prop:joli leray general properties}~(ii) and~(iii)
provides a norm map
$\pi_*\RR^qf_*\Z(2) \to \RR^qf^G_*\Z(2)$
whose cokernel is supported on $B(R) \subseteq B(C)/G$.
As $p>\dim(B(R))$ and $p+1>\dim(B(C)/G)$,
the induced map $H^p(B(C),\RR^qf_*\Z(2)) \to E_2^{p,q}$ is surjective
(see \cite[Chapter~II, Lemma~9.1]{delfshomology}).
To conclude the proof,
it therefore suffices
to check that $H^p(B(C),\sF)=0$
for any locally constant sheaf with finitely generated stalks~$\sF$ on~$B(C)$.
Indeed $\RR^qf_*\Z(2)$ is such a sheaf
by \cite[Corollary~17.20]{scheiderer} applied to $R_{C/R}f$.
As the group $H^p(B(C),\sF)$ is finitely generated
(see \cite[Theorem~17.9]{scheiderer}), it suffices to check that it is divisible.
For this,
as $p\geq\dim(B(C))$,
it is enough that $H^p(B(C),\sF/n\sF)=0$ for all $n\geq 1$.
Now this last vanishing follows from Poincar\'e duality for sheaves of abelian groups
with finite exponent
(see \cite[(\ref*{BW1-eq:semi-algebraic poincare duality})]{BW1})
and from the fact that $H^0_c(B(C),\sG)=0$
for any locally constant sheaf~$\sG$ on~$B(C)$ as~$B$ is affine.
\end{proof}

\subsection{Constructing multisections}
\label{subsec:horizontal}
It is now possible to give the

\begin{proof}[Proof of Theorem \ref{thm:solides fibres en del Pezzo}]
By Proposition \ref{prop:cubicsurfacefibrations}, we only have to consider cases (i), (ii) and (iii) of Theorem \ref{thm:solides fibres en del Pezzo}.
Choose points $t_1,\dots,t_s\in B$ such that $B^0:=B\setminus\{t_1,\dots,t_s\}$ is affine, such that the semi-algebraic connected components of $B^0(R)$ are semi-algebraically contractible, and such that the base change $f^0:X^0\to B^0$ is smooth and projective.
By Proposition \ref{prop:curvesinfibers} and the commutative diagram with exact row
$$
\xymatrix@R=3ex{
& \CH_1(X)\ar@{->>}[r] \ar^{\cl}[d] &\CH_1(X^0)\ar^{\cl}[d]\\
\displaystyle\bigoplus_{i=1}^s H^4_{G,X_{t_i}(C)}(X(C),\Z(2))\ar[r]& H^4_G(X(C),\Z(2))\ar[r]&H^4_G(X^0(C),\Z(2)),
}
$$
it suffices to show the surjectivity of $\cl:\CH_1(X^0)\to H^4_G(X^0(C),\Z(2))_0$. Fix a class $\alpha\in  H^4_G(X^0(C),\Z(2))_0$. To prove that it lies in the image of $\cl$, we first reduce to the case when $\alpha$ has degree $0$ on the geometric fibers of $f^0$. We distinguish two cases.

 Suppose that the degree of $\alpha$ on the geometric fibers of $f^0$ is even. 
By Manin and Colliot-Th\'el\`ene \cite[Chapter IV, Theorem 6.8]{kollarbook}, 
$f^0_{C}:X^0_C\to B^0_C$ has a section.
Its image in~$X^0$ is a curve of degree~$1$ or~$2$ over~$B^0$. Removing from $\alpha$ a multiple of the class of this curve, we may assume that $\alpha$ has degree $0$ on the geometric fibers of $f^0$. 

Suppose now that the degree of $\alpha$ on the geometric fibers of $f^0$ is odd. It follows from \cite[Proposition~\ref*{BW1-prop:hodgereel0cycles}]{BW1} applied to the real fibers of $f^0$ that $X^0(R)\to B^0(R)$ is surjective. In particular, we are in case (i) or (iii). Then we claim that $f^0:X^0\to B^0$ has a section. If $\delta=1$ such a section is given by the base locus of the anticanonical divisor of the generic fiber, and if $\delta=3$, one is provided by \cite[Corollary p.~390]{langrealplaces}. The cases $\delta\in\{5,7,9\}$ have already been dealt with in Example \ref{exeasydP}: then, the generic fiber $X_{R(B)}$ of $f^0$ is even rational over~$R(B)$. When $\delta=8$, $X_{R(B)}$ is either the blow-up of $\P^2_{R(B)}$ in a point (which has obvious rational points),
or a form of $\P^1\times\P^1$, hence a homogeneous space under a linear algebraic group (being homogeneous is a geometric property of projective varieties, by the representability of their automorphism group scheme). When $\delta=6$, $X_{R(B)}$ contains a torsor under a $2$-dimensional torus \cite[Theorem 3.10]{maninrational}.
 Consequently, the existence of a section in these cases
follows from the strong Hasse principle for homogeneous spaces under linear algebraic groups over function fields of curves over $R$ \cite[Corollary 6.2]{ScheidererHasse} (over $\R$, the case $\delta=8$ could have been deduced from \cite[Satz~22]{WittHasse} and the case $\delta=6$ from \cite[Théorème~1.1~(a)]{CTHasse}). 
  Removing from $\alpha$ a multiple of the class of this section, we may assume  that $\alpha$ has degree $0$ on the geometric fibers of $f^0$. 

The real fibers of $f^0$ have a connected or empty real locus: this is obvious in cases~(ii) and~(iii) and is a consequence of
\cite[Corollary~3 to Theorem~3.7, Corollary~1 to Theorem~3.10, Theorem~3.15]{maninrational}
and \cite[Theorem~13.3]{delfsknebuschbasictheory2} in case (i).
It follows from \cite[Lemma~\ref*{BW1-lem:H2dreel}]{BW1}
that if $b\in B^0(R)$, then $H^4_G(X_b^0(C),\Z(2))_0\simeq\Z$, generated by the class of any $R$-point (resp.\ closed point) of $X_b^0$ if $X_b^0(R)\neq\emptyset$ (resp.\ if $X_b^0(R)=\emptyset$). Similarly, if $b\in B^0(C)$, then $H^4(X_b^0(C),\Z(2))\simeq\Z$, generated by the class of any $C$-point of $X_b^0$. Since $\alpha$ has degree $0$ on the geometric fibers of $f^0$, we deduce that
 it satisfies the hypotheses of Proposition \ref{prop:Leraybizarre}, hence that it vanishes, proving the theorem. 
\end{proof}

\section{Non-archimedean real closed fields}
\label{nonarchimedeansection}

In this section, we construct counterexamples to the real integral Hodge conjecture that are specific to non-archimedean real closed fields, explain geometric consequences of these counterexamples,
and discuss and exploit the relationship between
the ``EPT'' and ``signs'' properties for a surface over a real closed field
and the real integral Hodge conjecture for conic bundles over the same surface.

\subsection{Curves of bounded degree}
\label{subsec:curves of bounded degree}

Given an integer~$k$ and a bounded family, defined over~$\R$, of smooth projective
varieties~$X$ which satisfy $H^q(X,\Omega_X^p)=0$ for all $p,q$ such that $p+q=2k$, $(p,q)\neq (k,k)$,
the real integral Hodge conjecture
for codimension~$k$ cycles on the members of this family defined over an arbitrary real closed
field
is equivalent
to the real integral Hodge conjecture for codimension~$k$ cycles
on the members that are defined over~$\R$ together with
a bound on the degrees of the cycles which span their cohomology:

\begin{prop}
\label{prop:ihcnonarchbounded}
Let $f:X\to B$ be a smooth and projective morphism between varieties over a real closed field~$R_0$.
Let $C_0=R_0(\sqrt{-1}\mkern2mu)$.
Let $\Theta \subseteq B(R_0)$ be a semi-algebraic subset.
Let~$L$ be a relatively ample line bundle on~$X$.
Let~$k$ be an integer.
For any integer~$\delta$, any field extension~$R/R_0$ and any $b \in B(R)$,
let $\CH^k(X_b)_{\deg\leq \delta}$ denote the subgroup of $\CH^k(X_b)$ generated by
the classes of integral closed subvarieties of~$X_b$ of dimension~$k$ whose degree,
with respect to~$L$, is~$\leq\delta$,
and let $\Theta_R \subseteq B(R)$ denote the semi-algebraic subset
obtained from~$\Theta$ by extension of scalars (see \cite[\textsection4]{delfsknebuschonthehomology}).
The following conditions are equivalent:
\begin{enumerate}[(i)]
\item for any real closed field extension $R/R_0$ and any $b\in \Theta_R$, the
equivariant cycle class map
$\cl:\CH^k(X_b) \to H^{2k}_G(X_b(C),\Z(k))_0$
is surjective, with $C=R(\sqrt{-1}\mkern2mu)$;
\item there exists an integer~$\delta$ such that for any $b\in \Theta$,
the restriction of the map
$\cl:\CH^k(X_b) \to H^{2k}_G(X_b(C_0),\Z(k))_0$
to the subgroup $\CH^k(X_b)_{\deg\leq\delta}$
is surjective.
\end{enumerate}
\end{prop}

\begin{proof}
Let~$B_\rs$ denote the real spectrum of~$B$ (see \cite{costeroyspectrereel},
\cite[Chapter~7]{bcr},
\cite[(0.4)]{scheiderer}).
We recall that~$B_\rs$ is a topological space whose points are represented by
pairs~$(R,b)$ consisting of a real closed field extension $R/R_0$ and a point $b \in B(R)$;
two pairs $(R,b)$ and $(R',b')$ define the same point of~$B_\rs$ if and only if~$b$ and~$b'$ lie above
the same scheme-theoretic point of~$B$ and~$R$ and~$R'$ induce the same ordering on its residue field.
We shall abuse notation and freely write $(R,b)$ for a point of~$B_\rs$.
We also recall that
semi-algebraic subsets of~$B(R_0)$ are in one-to-one correspondence with constructible subsets of~$B_\rs$
and that the latter form the base of a compact topology (see \cite[Proposition~7.1.12, Theorem~7.2.3]{bcr}).

\begin{lem}
\label{lem:constructibility of algebraicity degree delta}
Let~$\delta$ be an integer.
There is a constructible subset $Z_\delta \subseteq B_\rs$
such that for any real closed field extension~$R/R_0$
and any $b \in B(R)$,
if we let $C=R(\sqrt{-1}\mkern2mu)$,
the restriction
of the equivariant cycle class map
$\cl:\CH^k(X_b) \to H^{2k}_G(X_b(C),\Z(k))_0$
to $\CH^k(X_b)_{\deg\leq\delta}$
is surjective
if and only if
the point $(R,b) \in B_\rs$ belongs to~$Z_\delta$.
\end{lem}

\begin{proof}
Let~$Z_\delta^+$
(resp., $Z_\delta^-$)
denote the set of points of~$B_\rs$ for which every (resp., some)
representative $(R,b)$ satisfies the property of the lemma.
We shall now check that any $(R,b) \in B_\rs$
is contained in a constructible subset~$Z$ of~$B_\rs$
which is itself either contained in~$Z_\delta^+$ or disjoint from~$Z_\delta^-$.
This will prove the lemma.  Indeed, this assertion implies that $Z_\delta^+=Z_\delta^-$ and that,
if~$Z_\delta$ denotes this common subset,
both~$Z_\delta$ and $B_\rs \setminus Z_\delta$ are unions of constructible subsets;
the constructible topology on~$B_\rs$ being compact, it then follows that~$Z_\delta$ is constructible.

Let~$M$ denote the open subvariety of the
relative Hilbert scheme of~$f$ which parametrises the
reduced closed subschemes of pure codimension~$k$
and degree~$\leq \delta$ (with respect to~$L$) in the fibres of~$f$
(see \cite[Th\'eor\`eme~3.2, Lemme~2.4]{tdte4},
\cite[Th\'eor\`eme~12.2.1]{ega43}).
Let $\pi:M\to B$ denote the projection. Let $\mu:U\to M$ be the universal family.
Let us write~$M$ as the set-theoretic union of finitely many locally closed subvarieties which are
smooth over~$R_0$.  Let $M'$ be their disjoint union, let $\nu:M'\to M$ be the natural map,
let $\pi'=\pi\circ\nu:M'\to B$, let $\mu':U'\to M'$ denote the base change of~$\mu$ by~$\nu$
and let $f':X' \to M'$ denote the base change of~$f$ by~$\pi'$.
We note that~$X'$ is smooth, since~$M'$ and~$f'$ are smooth,
and that~$U'$ is a reduced closed subvariety of~$X'$ of pure codimension~$k$
(see \cite[Corollaire~3.3.5]{ega42}).
We write $[U'] \in \CH^k(X')$ for the class of the corresponding cycle.
For $(R',m') \in M'_\rs$,
the fibre $U'_{m'} = \mu'^{-1}(m')$
is a reduced closed subvariety of $X'_{m'}$ of pure codimension~$k$,
therefore
the pull-back map $\CH^k(X')\to \CH^k(X'_{m'})$
sends~$[U']$ to~$[U'_{m'}]$
(see \cite[Proposition~8.2]{fulton}).
In view of the compatibility of the equivariant cycle class map with pull-backs, we deduce that
$\cl([U'_{m'}]) \in H^{2k}_G(X'_{m'}(C'),\Z(k))$, for any $(R',m') \in M'_\rs$, is the image of
the fixed class $\cl([U'])$ by the composed map
\begin{align}
\label{eq:restriction to fibre eq coh}
H^{2k}_G(X'(C_0),\Z(k)) = H^{2k}_G(X'(C'),\Z(k)) \to H^{2k}_G(X'_{m'}(C'),\Z(k))
\end{align}
(see \cite[Chapter~II, Theorem~6.1]{delfshomology} for the first isomorphism).

For a quasi-projective variety~$V$ over~$R_0$, let $V_\cx$ denote the complex
spectrum of~$V$, \emph{i.e.}, the real spectrum of the Weil restriction
of $V \otimes_{R_0}C_0$ from~$C_0$ to~$R_0$
(see \cite[(0.4) and Definition~5.6.2]{scheiderer}).
There is a natural map $B_\rs \to B_\cx$.
Let $Y_{\cx/\rs} = X_\cx \times_{B_\cx} B_\rs$
and let $g:Y_{\cx/\rs}\to B_\rs$ denote the second projection.
Taking up the notation of~\textsection\ref{subsubsec:two leray}
and freely identifying sheaves on the semi-algebraic site of~$B(R_0)$ with
sheaves on the topological space~$B_\rs$ (see \cite[Chapter~I, Proposition~1.4]{delfshomology}),
we set $\sF=\RR^{2k}g^G_*\Z(k)$.
This is a locally constant sheaf on~$B_\rs$
(see Proposition~\ref{prop:joli leray general properties}~(vii)).
The fibre of~$g$ above a point of~$B_\rs$ represented by a pair~$(R,b)$ is the ``abstract semi-algebraic space'',
in the terminology of \cite[Appendix~A to Chapter~I]{delfsknebuschbook},
associated with the semi-algebraic space $X_b(C_1)$, where $C_1=R_1(\sqrt{-1}\mkern2mu)$
and~$R_1$ is the real closed subfield of~$R$ which is algebraic over
the residue field of the scheme-theoretic point of~$B$
over which~$b$ lies.
In particular, the stalk of~$\sF$ at this point coincides with $H^{2k}_G(X_b(C),\Z(k))$
(see Proposition~\ref{prop:joli leray general properties}~(v),
\cite[Chapter~I, Corollary~1.5 and Chapter~II, Theorem~6.1]{delfshomology}).
It follows from Proposition~\ref{prop:joli leray general properties}~(vi)
that the groups $H^{2k}_G(X_b(R),\Z(k))$ are also the stalks of a locally constant
sheaf and, hence, that the subgroups $H^{2k}_G(X_b(C),\Z(k))_0$
fit together in a locally constant subsheaf $\sF_0 \subseteq \sF$.
Finally,
let $Y'_{\cx/\rs}=X'_\cx\times_{M'_\cx}M'_\rs$,
let $g':Y'_{\cx/\rs}\to M'_\rs$ denote the second projection,
let $\sF'=\RR^{2k}g'^G_*\Z(k)$,
and let $\pi'_\rs:M'_\rs\to B_\rs$ denote the map induced by~$\pi'$.

We now fix $(R,b) \in B_\rs$ and construct~$Z$.
Let us choose a connected constructible open neighbourhood
$\Omega \subset B_\rs$ of~$(R,b)$
on which the sheaves~$\sF$ and~$\sF_0$ are constant.
As $\sF'=\pi'^*_\rs\sF$
(see Proposition~\ref{prop:joli leray general properties}~(iv)),
we can canonically identify the stalk of~$\sF'$ at any $(R',m') \in \pi'^{-1}_\rs(\Omega)$
with the group $H^{2k}_G(X_b(C),\Z(k))$.
Let us define a map
\begin{align}
\label{eq:locally constant map hilbert}
\tau:\pi'^{-1}_\rs(\Omega) \to H^{2k}_G(X_b(C),\Z(k))_0
\end{align}
by sending a point $(R',m')$ to the image
of $\cl([U'_{m'}])$
by the
canonical isomorphism
$H^{2k}_G(X'_{m'}(C'),\Z(k))_0=H^{2k}_G(X_b(C),\Z(k))_0$.
As $\cl([U'_{m'}])$
is also the image of~$\cl([U'])$
by~\eqref{eq:restriction to fibre eq coh},
we see that~$\tau$
coincides with the map that
sends $(R',m')$ to the germ at $(R',m')$ of the global section
of~$\sF'$ defined by $\cl([U'])$.
In particular, as~$\sF'$ is locally constant,
we deduce that~$\tau$ is locally constant.

Let $W_1,\dots,W_m$ (resp., $W_{m+1},\dots,W_n$) denote the (finitely many,
see \cite[Corollaire~3.7]{costeroyspectrereel}) connected components of $\pi'^{-1}_\rs(\Omega)$
whose images by~$\pi'_\rs$ contain~$(R,b)$ (resp., do not contain~$(R,b)$).
The subset
\begin{align*}
Z = \pi'_\rs(W_1) \cap \dots \cap \pi'_\rs (W_m) \cap (\Omega \setminus \pi'_\rs(W_{m+1})) \cap \dots \cap \left(\Omega \setminus \pi'_\rs(W_n)\right)
\end{align*}
of~$\Omega$ is constructible
(\emph{op.\ cit.}, Proposition~2.3)
and contains~$(R,b)$.

For $(R',b') \in \Omega$,
if~$R'_1$ denotes the real closed subfield of~$R'$ which is algebraic over
the residue field of the scheme-theoretic point of~$B$
over which~$b'$ lies, and~$b'_1$ denotes the point~$b'$ viewed as an element of $B(R'_1)$,
the real spectrum of $\pi'^{-1}(b'_1)$
coincides with $\pi_\rs'^{-1}((R',b'))$.
Therefore $\pi'^{-1}(b'_1)(R'_1) \subseteq \pi_\rs'^{-1}((R',b'))$ is a dense subset.
In particular, the locally constant map~$\tau$ takes the same set of values
on $\pi_\rs'^{-1}((R',b'))$ and on this subset.
It follows that
for $(R',b') \in \Omega$, the image of
$\CH^k(X_{b'})_{\deg\leq\delta}$ by
the equivariant cycle class map
$\cl:\CH^k(X_{b'}) \to H^{2k}_G(X_{b'}(C'),\Z(k))_0$
coincides, via the identification
$H^{2k}_G(X_{b'}(C'),\Z(k))_0=H^{2k}_G(X_b(C),\Z(k))_0$,
with the subgroup of
$H^{2k}_G(X_b(C),\Z(k))_0$
generated by $\tau(\pi_\rs'^{-1}((R',b')))$.
As the map~$\tau$ is locally constant
and as every fibre of the projection $\pi'^{-1}_\rs(Z) \to Z$
meets each of $W_1,\dots,W_m$
and none of $W_{m+1},\dots,W_n$,
it follows that~$Z$ is either contained in~$Z_\delta^+$ or disjoint from~$Z_\delta^-$.
\end{proof}

We can now prove the proposition.
Let $\Theta_\rs \subseteq B_\rs$ denote the unique constructible subset such that
$\Theta=\Theta_\rs \cap B(R_0)$
(see \cite[Proposition~7.2.2~(i)]{bcr}).
As the groups $H^{2k}_G(X_b(C),\Z(k))_0$ are finitely generated,
property~(i) holds if and only if $\Theta_\rs \subseteq \bigcup_\delta Z_\delta$.
As the constructible topology on~$B_\rs$ is compact,
this is, in turn, equivalent to the existence of~$\delta$ such that $\Theta_\rs \subseteq Z_\delta$.
On the other hand, property~(ii) is equivalent to the existence of~$\delta$
such that $\Theta \subseteq Z_\delta$.
Now $\Theta \subseteq Z_\delta$ if and only if $\Theta=Z_\delta \cap \Theta_\rs \cap B(R_0)$,
which, by the definition of~$\Theta_\rs$, holds if and only if $\Theta_\rs=Z_\delta \cap \Theta_\rs$.
Thus (i)$\Leftrightarrow$(ii).
\end{proof}

We end this paragraph with the following related result.

\begin{prop}
\label{prop:invIHCextR}
Let $R/R_0$ be an extension of real closed fields,
let $C_0=R_0(\sqrt{-1}\mkern2mu)$
and $C=R(\sqrt{-1}\mkern2mu)$.
Let~$X$ be a smooth, projective variety over~$R_0$.
For any integer~$k$,
the map $\cl:\CH^k(X)\to H^{2k}_G(X(C_0),\Z(k))_0$ is surjective if and only if so is the map $\cl:\CH^k(X_R)\to H^{2k}_G(X(C),\Z(k))_0$.
\end{prop}

\begin{proof}
Apply Lemma~\ref{lem:constructibility of algebraicity degree delta}
with $B=\Spec(R_0)$.
\end{proof}

\begin{rmks}
(i)
Propositions~\ref{prop:ihcnonarchbounded} and~\ref{prop:invIHCextR}
still hold,
with the same proofs,
if, in their statements,
one replaces the equivariant cycle class map
to $H^{2k}_G(X_b(C),\Z(k))_0$
with the Borel--Haefliger cycle class map
to $H^k(X_b(R),\Z/2\Z)$. We recall that the latter factors through the former
(see \cite[\textsection\ref*{BW1-subsubsec:topological constraints}]{BW1}).

(ii)
In Example~\ref{ex:implicationsDucros} below,
we shall provide an example of a smooth and projective morphism $f:X\to B$ defined over~$\R$
and of an integer~$k$
such that the equivariant cycle class map
$\cl:\CH^k(X_b) \to H^{2k}_G(X_b(\C),\Z(k))_0$ is surjective
for all $b\in B(\R)$ even though there is no~$\delta$
such that
its restriction to $\CH^k(X_b)_{\deg\leq\delta}$ is simultaneously surjective
for all $b \in B(\R)$.
This is a specifically real phenomenon,
which cannot occur over~$\C$:
if the cycle class map
$\cl:\CH^k(X_b) \to H^{2k}(X_b(\C),\Z(k))$ is surjective
for all $b\in B(\C)$, then there is always
an integer~$\delta$ such that
its restriction to $\CH^k(X_b)_{\deg\leq\delta}$ is simultaneously surjective
for all $b \in B(\C)$.
Indeed, when~$B$ is irreducible, the surjectivity
of $\cl:\CH^k(X_b) \to H^{2k}(X_b(\C),\Z(k))$
for a very general $b \in B(\C)$ furnishes cycles which spread out over
a dense open subset of~$B$, since the relative Hilbert scheme of~$f$ only has countably many
irreducible components; an induction on~$\dim(B)$ completes the proof of the existence of~$\delta$.
\end{rmks}

\subsection{Counterexamples to the real integral Hodge conjecture}
\label{par:cexnonarch}
As announced at the end of \cite[\textsection\ref*{BW1-section:examples}]{BW1},
we now give examples of smooth, projective, connected varieties~$X$ of dimension~$d$ over a real closed field~$R$,
such that $H^2(X,\sO_X)=0$ and that the real integral Hodge conjecture for $1$\nobreakdash-cycles on~$X$ fails
(in the sense that the equivariant cycle class map $\CH_1(X)\to H^{2d-2}_G(X(C),\Z(d-1))_0$ is not surjective;
see \cite[Definition~\ref*{BW1-def:ihc real closed field}]{BW1})
while the complex integral Hodge conjecture does hold (in the sense that
the cycle class map $\CH_1(X_C)\to H^{2d-2}(X(C),\Z(d-1))$ is surjective).
Our examples are Calabi--Yau threefolds, with or without real points (\S\ref{sss:CYR}), and a rationally connected threefold with real points (\S\ref{sss:RCR}),
over the field or real Puiseux series $\bigcup_{n\geq 1}\R((t^{1/n}))$.
Such varieties satisfy the complex integral Hodge conjecture
by Voisin's theorem \cite[Theorem~2]{voisinthreefolds} combined with
the Lefschetz principle.

By \cite[Corollary \ref*{BW1-cor:IHC pour solides RC ou CY}]{BW1}, these are examples of smooth projective varieties $X$ over $R$ such that $H_1(X(R),\Z/2\Z)\neq H_1^\alg(X(R),\Z/2\Z)$, if $X(R)\neq\varnothing$, or on which there is no geometrically irreducible curve of even geometric
genus, if $X(R)=\varnothing$.
These examples of cycle-theoretic obstructions
(in the terminology of~\cite[\textsection\ref*{BW1-nomenclature}]{BW1})
of a new kind
complement those given over~$\R$ in \cite[\S\ref*{BW1-subsec:cycle theoretic obs}]{BW1}.

By Proposition \ref{prop:ihcnonarchbounded}, these examples show that to prove the real integral Hodge conjecture for $1$-cycles on rationally connected or Calabi--Yau threefolds over $\R$ (see \cite [Question \ref*{BW1-mainquestion}]{BW1}), one cannot restrict to curves of bounded degree, in bounded families of varieties. We refer to \S\ref{sss:discussion IHCR} for further discussion of this question.

Another counterexample to the real integral Hodge conjecture, for a uniruled threefold with no real point, will appear later (see Example \ref{ex:cexuniruled} in \S\ref{subsubsec:cokcycleclassmap}).

\subsubsection{Calabi--Yau threefolds}
\label{sss:CYR}

Here are two Calabi--Yau counterexamples.  They are both simply connected.

\begin{example}[Calabi--Yau threefold with real points]
\label{ex:cexCYavecpoint}
We consider the real closed field $R:=\cup_n\R((t^{1/n}))$.
Let $G:= (u^2+v^2)F(x_0,\dots,x_3)\in H^0(\P^1_\R\times\P^3_\R,\sO(2,4))$, where $u,v$ (resp. $x_0,\dots,x_3$) are the homogeneous coordinates of $\P^1$ (resp. $\P^3$) and $F$ is the equation of a smooth quartic surface in $S\subset\P^3_{\R}$ as in \cite[Example~\ref*{BW1-ex:quartique avec points}]{BW1}. Let $H\in H^0(\P^1_\R\times\P^3_\R,\sO(2,4))$ be the equation of any smooth hypersurface. Consider $G+tH\in  H^0(\P^1_R\times\P^3_R,\sO(2,4))$ and let $X:=\{G+tH=0\}\subset \P^1_R\times\P^3_R$. The Calabi--Yau threefold  $X$ is smooth since $\{H=0\}$ is smooth.

Let $X_0:=\{G=0\}\subset\P^1_\R\times\P^3_\R$ and $X_0^{\sm}$ be its smooth locus. Since $X_0(\R)\subset X_0^{\sm}$, one may consider 
the composition $\CH_1(X_0)\to \CH_1(X_0^{\sm})\xrightarrow{\cl_{\R}} H^2(X_0(\R),\Z/2\Z)$
and define $H^2_{\alg}(X_0(\R),\Z/2\Z)$ to be its image. By the choice of $F$, and noting that $X_0(\R)=\P^1(\R)\times S(\R)$, we see that there exists a class $\alpha_0\in H^2(X_0(\R),\Z/2\Z)$ such that $\deg(\alpha_0\smile\cl_{\R}(\sO_{X_0}(0,1)))\neq 0\in \Z/2\Z$, but no class $\beta\in H^2_{\alg}(X_0(\R),\Z/2\Z)$ is such that $\deg(\beta\smile\cl_{\R}(\sO_{X_0}(0,1)))\neq 0\in \Z/2\Z$.

Introduce $f:\sX\to \A^1_\R$, where  $\sX:=\{G+tH=0\}$ and $t$ is the coordinate of $\A^1_\R$. The variety $X_0$ is the fiber of $f$ over $0$ and $X$ is a real closed fiber of $f$ specializing to~$X_0$. Denoting by $f_r:\sX_r\to\A^1_{\R,r}$ the induced map between real spectra, the sheaf $\RR^2f_{r*}\Z/2\Z$ satisfies proper base change by
\cite[Chapter~II, Theorem~7.8]{delfshomology} and is constant in a neighbourhood of $0$ by \cite[Corollary 17.20 a)]{scheiderer}
(the proof of this last result only uses the smoothness of $f$ in a neighbourhood of the real locus). This induces an isomorphism $H^2(X(R),\Z/2\Z)\simeq H^2(X_0(\R),\Z/2\Z)$.
Let $\alpha\in H^2(X(R),\Z/2\Z)$ be the class corresponding to $\alpha_0$ under this isomorphism. 

 If $\alpha$ were algebraic, there would exist a curve $Z\subset X$ such that $\deg(\cl_{R}(Z)\smile\cl_{R}(\sO_{X}(0,1)))\neq 0\in \Z/2\Z$, hence with odd degree with respect to $\sO_{X}(0,1)$. The specialization $Z_0\subset X_0$ of $Z$ would have odd degree with respect to $\sO_{X_0}(0,1)$, hence would be such that $\deg(\cl_{\R}(Z_0)\smile\cl_{\R}(\sO_{X_0}(0,1)))\neq 0\in \Z/2\Z$, a contradiction. We deduce that $H^2(X(R),\Z/2\Z)\neq H^2_{\alg}(X(R),\Z/2\Z)$. By \cite[Corollary \ref*{BW1-cor:IHC pour solides RC ou CY}]{BW1}, $X$ does not satisfy the real integral Hodge conjecture for $1$-cycles.

Note that in this example, we could have chosen $S$ to be any smooth quartic surface such that $H^1(S(\R),\Z/2\Z)\neq H^1_{\alg}(S(\R),\Z/2\Z)$, after having verified that the specialization isomorphism $H^2(X(R),\Z/2\Z)\simeq H^2(X_0(\R),\Z/2\Z)$ is compatible with the specialization map for Chow groups.
\end{example}

\begin{example}[Calabi--Yau threefold with no real point]
\label{ex:cexCYsanspoint} 
Consider the real closed field $R:=\cup_n\R((t^{1/n}))$ and let $G:= (u^2+v^2)F(x_0,\dots,x_3)\in H^0(\P^1_\R\times\P^3_\R,\sO(2,4))$, where $u,v$ (resp. $x_0,\dots,x_3$) are the homogeneous coordinates of $\P^1$ (resp. $\P^3$) and $F$ is the equation of a smooth quartic surface in $\P^3_{\R}$ with no real points containing no geometrically irreducible curve of even geometric genus, as in \cite[Example~\ref*{BW1-ex:kollar quartique sans point}]{BW1}. Let $H\in H^0(\P^1_\R\times\P^3_\R,\sO(2,4))$ be the equation of any smooth hypersurface. Consider $G+tH\in  H^0(\P^1_R\times\P^3_R,\sO(2,4))$ and let $X:=\{G+tH=0\}\subset \P^1_R\times\P^3_R$.

The variety $X$ is a smooth Calabi--Yau threefold  over $R$ (because $\{H=0\}$ is smooth) that has no $R$-points since $X_0:=\{G=0\}\subset\P^1_\R\times\P^3_\R$ has no $\R$\nobreakdash-points.
Suppose that $X$ contains a geometrically integral curve of even geometric genus. Then, by
\cite[Corollary \ref*{BW1-cor:what is elw1}]{BW1}, $\elw_1(X)=1$ (see \cite[\S\ref*{BW1-subsec:elw}]{BW1} for generalities on intermediate indices). Specializing a coherent sheaf on $X$ with odd Euler characteristic shows that $\elw_1(X_0)=1$. Applying \cite[Corollary \ref*{BW1-cor:what is elw1}]{BW1} again shows that $X_0$ contains a geometrically integral curve of even geometric genus. This curve must be contained in $\{F=0\}$, contradicting the choice of $F$.
By \cite[Corollary~\ref*{BW1-cor:IHC pour solides RC ou CY}]{BW1}, $X$ does not satisfy the real integral Hodge conjecture for $1$-cycles.
\end{example}

\subsubsection{Rationally connected threefolds}
\label{sss:RCR}

We now turn to an example of a rationally connected threefold failing the real integral Hodge conjecture for $1$-cycles.

\begin{example}[rationally connected threefold with real points]
\label{ex:cexRC}
 As before, we let $R:=\cup_n\R((t^{1/n}))$.
 We consider a variety constructed by Ducros \cite[\S 8]{ducros} in his study of the Hasse principle for varieties over functions fields of curves over real closed fields, and we reinterpret it in our setting (beware that his notation differs slightly from ours).

Ducros sets $K:=R(\P^1_R)$ and defines by equations in \cite[Proposition 8.11]{ducros} a conic bundle $X_K\to \P^1_K$. We may take a smooth projective model $X\to \P^1_R\times\P^1_R$ of $X_K$: it is a conic bundle over $\P^1_R\times\P^1_R$. The content of \cite[Proposition~8.11]{ducros} is that $X_K(K)=\emptyset$ although there is no ``reciprocity obstruction'' to the existence of a rational point on $X_K$. The first condition exactly means that the composition $p:X\to\P^1_R$ with the second projection does not admit an algebraic section. By \cite[Th\'eor\`eme 4.3]{ducros}, the second condition is equivalent to the existence of a continuous semi-algebraic section $\sigma:\P^1(R)\to X(R)$ of $p(R):X(R)\to \P^1(R)$, taking values in the smooth locus of $p$.

The image of $\sigma$ induces a homology class $\alpha\in H_1(X(R),\Z/2\Z)$. If $X$ satisfied the real integral Hodge conjecture for $1$-cycles, $\alpha$ would be the Borel--Haefliger class of a $1$-cycle $z$ on $X$, by \cite[Corollary \ref*{BW1-cor:IHC pour solides RC ou CY}]{BW1}. Restricting $z$ to a general fiber of $p:X\to \P^1_R$, one sees that $z$ has odd degree over $\P^1_R$, hence that $X_K$ has a $0$-cycle of odd degree. Since $X_K$ obviously has a closed point of degree $2$, being a conic bundle over $\P^1_K$, it has index $1$ over $K$. Colliot-Th\'el\`ene and Coray have shown in \cite[Corollaire~4 p.~309]{ctcoray} that conic bundles over $\P^1_K$ that have at most $5$ singular geometric fibers have a rational point if and only if they have index $1$. Since there are only $4$ singular geometric fibers in Ducros' example, this result applies to show that $X_K(K)\neq\emptyset$, hence that $p:X\to\P^1_R$ has a section: this is a contradiction.
\end{example}

\subsubsection{Implications for the real integral Hodge conjecture}
\label{sss:discussion IHCR}

Examples  \ref{ex:cexCYavecpoint}, \ref{ex:cexCYsanspoint} and \ref{ex:cexRC} feature Calabi--Yau and rationally threefolds over $R$ for which the real integral Hodge conjecture for $1$-cycles fails. This contrasts with the hope, expressed in \cite[Question \ref*{BW1-mainquestion}]{BW1},
 that
the following question may have a positive answer:

\begin{question}
\label{mainquestionpart2}
 Does the real integral Hodge conjecture hold for $1$-cycles on rationally connected and Calabi--Yau threefolds over $\R$ ?
\end{question}

Since, for such varieties $X$, the integral Hodge conjecture holds for $1$-cycles on $X_C$ by Voisin's theorem \cite[Theorem 2]{voisinthreefolds} and the Lefschetz principle, our examples show that it is not possible to deduce in a formal way a positive answer to Question~\ref{mainquestionpart2}
from Voisin's theorem in the spirit of the proof of \cite[Proposition~\ref*{BW1-prop:real(1,1)nonarch}]{BW1}. This is reflected by the fact that many of our partial answers to Question \ref{mainquestionpart2} use tools that are specific to the field $\R$: either Hodge theory, or the Stone--Weierstrass approximation theorem (see Theorem \ref{thm:fibres en coniques}, Theorem \ref{thm:solides fano} and Theorem \ref{thm:solides fibres en del Pezzo} (iv)).

Despite Example \ref{ex:cexRC}, the real integral Hodge conjecture for $1$-cycles does hold over arbitrary real closed fields for some particular families of rationally connected varieties. We will see that such is the case for cubic hypersurfaces of dimension $\geq 3$ in Theorem \ref{thm:cubiques}.

\begin{example}
\label{ex:implicationsDucros}
Example \ref{ex:cexRC} is particularly interesting because it concerns a conic bundle over $\P^1_R\times\P^1_R$, a kind of variety for which the real integral Hodge conjecture for $1$-cycles over $\R$ is known (Corollary \ref{coro:solides fibres en coniques}).

A glance at the equations given by Ducros \cite[Proposition 8.11]{ducros}
shows that the variety $X$ of Example \ref{ex:cexRC} spreads out to a smooth projective family $f:\sX\to B$ over an open subset $B$ of $\A^1_{\R}$ with coordinate $t$.
Let~$L$ be a relatively ample line bundle on $\sX$.
Then, on the one hand, for every $b\in B(\R)$, $H_1(\sX_b(\R),\Z/2\Z)$ is generated by Borel--Haefliger classes of algebraic curves on $\sX_b$,
by
Corollary~\ref{coro:solides fibres en coniques}
and \cite[Corollary~\ref*{BW1-cor:IHC pour solides RC ou CY}]{BW1}.
On the other hand, it follows from Proposition \ref{prop:ihcnonarchbounded}
that for any integer $\delta\geq 1$, there exists $b\in B(\R)$ such that  $H_1(\sX_b(\R),\Z/2\Z)$ fails to be generated by the Borel--Haefliger classes of curves on $\sX_b$ that have degree $\leq \delta$ with respect to $L$.

Thus there is no hope to prove Theorem \ref{thm:fibres en coniques}
 by showing that some curve-counting invariant (among real avatars of Gromov--Witten invariants) is non-zero. Indeed such an invariant would control curves of a fixed degree, and would be constant in real deformations of smooth projective varieties over $\R$.
\end{example}

\begin{example}
In \cite[Example 26.2]{kollarmangolte}, Koll\'ar and Mangolte exhibit a related example,
which can be viewed as
a smooth quartic threefold $X$ over $R=\cup_n\R((t^{1/n}))$ such that $H_1(X(R),\Z/2\Z)\neq 0$ although all rational curves on $X$ have trivial Borel--Haefliger class. 

Letting~$D$ be the curve used in the construction of \emph{loc.\ cit.},
one can prove the stronger fact that $H_1(X(R),\Z/2\Z)\neq H_1^{\alg}(X(R),\Z/2\Z)$
as soon as $D(\R)$ has at least two connected components.
(Indeed, it is possible, although we do not do it here, to construct a canonical specialisation map
$H_1(X(R),\Z/2\Z)\to H_1(D(\R),\Z/2\Z)$ which preserves algebraic classes and which, in this example, turns out
to be an isomorphism.)
The real integral Hodge conjecture for $1$\nobreakdash-cycles on~$X$ then fails.

Extrapolating from this example, Koll\'ar and Mangolte conjecture that there should exist smooth quartic threefolds~$X$
over~$\R$
such that $H_1(X(\R),\Z/2\Z)\neq 0$ although all rational curves on~$X$ have trivial Borel--Haefliger class.
They explain  that otherwise, the rational curves whose classes span $H_1(X(\R),\Z/2\Z)$ could not have uniformly bounded degrees as $X$ varies over the smooth quartic threefolds over $\R$ (an assertion parallel to our Proposition \ref{prop:ihcnonarchbounded}). Example~\ref{ex:implicationsDucros} shows that this heuristic argument is not satisfactory.
\end{example}

There is no analogue of Example \ref{ex:cexRC} for conic bundles without real points. Indeed, let $X\to S$ be a conic bundle over a smooth geometrically rational surface over $R$ such that $X(R)=\emptyset$. If $S(R)\neq\varnothing$, the fiber of $X\to S$ over a general real point is a real conic in $X$, showing that $\elw_1(X)=1$. If $S(R)=\varnothing$, then $S$ contains a geometrically rational curve $\Gamma$ by a theorem of Comessatti (\cite[discussion below Teorema VI p.~60]{comessatti}, see also \cite{colliotsanspoint}). For the same reason, the inverse image of $\Gamma$ in $X$ contains a geometrically rational curve, showing again that $\elw_1(X)=1$.
 In both cases,  \cite[Corollary~\ref*{BW1-cor:IHC pour solides RC ou CY}]{BW1} shows
 that $X$ satisfies the real integral Hodge conjecture for $1$-cycles.

In fact, we do not have any counterexample to the real integral Hodge conjecture for $1$-cycles on rationally connected varieties without real points over arbitrary real closed fields (although 
we will see that it fails in general for uniruled threefolds with no real points in \S\ref{subsubsec:cokcycleclassmap} below, see Example  \ref{ex:cexuniruled}). It is natural to wonder about the following partial generalization of \cite[Question \ref*{BW1-mainquestion}]{BW1}:

\begin{question}
\label{IHCRCR}
 Does the real integral Hodge conjecture hold for $1$-cycles on rationally connected varieties $X$ over $R$ such that $X(R)=\emptyset$ ?
\end{question}

By Proposition \ref{prop:ihcnonarchbounded}, a positive answer would imply that
in any bounded family of rationally connected varieties over $\R$,
the members that have no real point contain a geometrically irreducible curve of even geometric genus whose degree with respect to a fixed polarization is uniformly bounded.
Already in the case of quartic threefolds without real points (Example \ref{ex:elw hypersurface quartique}), we do not know if such a uniform bound exists.

\subsection{Hypersurfaces in \texorpdfstring{$\P^4$}{𝐏⁴}}
\label{subsec:hypersurfaces in P4}

In this paragraph, we discuss the existence of geometrically irreducible curves of even geometric genus in smooth hypersurfaces $X\subset\P^4_R$. By \cite[Proposition \ref*{BW1-prop:parity of genus in irrelevant situations}]{BW1}, this is only interesting if $X(R)=\varnothing$, which we now assume; the degree $\delta$ of $X$ is then even. By \cite[Theorem \ref*{BW1-th:nohodgetheoreticob}]{BW1}, the existence of a geometrically irreducible curve of even geometric genus on~$X$ would follow from the validity of the real integral Hodge conjecture for $1$-cycles on $X$.

The intersection of $X$ by two general hyperplanes is a curve of genus $\frac{(\delta-1)(\delta-2)}{2}$, showing that there exist such curves when $\delta\equiv 2\pmod 4$.
 When $\delta=4$, we have already explained in Example \ref{ex:elw hypersurface quartique} that $X$ contains a geometrically irreducible curve of even geometric genus if $R=\R$, but that we do not know whether this still holds over arbitrary real closed fields. This is related to Question \ref{IHCRCR}.

When $\delta\geq 6$, we are out of the rationally connected or Calabi--Yau range, and there are no general reasons to expect that the real integral Hodge conjecture for $1$-cycles on $X$ holds. Recall from \cite[Theorem \ref*{BW1-th:phi}]{BW1} the existence of a morphism:
\begin{equation}
\label{eq:rappelphi}\phi:\CH_1(X)\to\Z/2\Z
\end{equation}
sending the class of an integral curve to $1$ if and only if it is geometrically integral of even geometric genus. Over $\C$, Griffiths and Harris have conjectured that a very general hypersurface $X\subset\P^4_{\C}$ of degree $\delta\geq 6$ satisfies $\CH_1(X)\simeq\Z$, generated by the intersection of $X$ with two general hyperplanes \cite[1) p.~32]{griffithsharrisnl}. At this point, it is natural to consider the real analogue of this conjecture:
\begin{question}
\label{griffithsharrisreel}
Does a very general hypersurface $X\subset\P^4_{\R}$ of degree $\delta\geq 6$ satisfy $\CH_1(X)\simeq\Z$, generated by the intersection of $X$ with two general hyperplanes ?
\end{question}
If this question has a positive answer, a very general degree $\delta$ hypersurface $X\subset\P^4_{\R}$ without real points  contains no geometrically integral curve of even geometric genus if $\delta\equiv 0\pmod 4$ and $\delta\geq 8$. Indeed, in this case, the morphism $\phi$ of (\ref{eq:rappelphi}) sends the generator of $\CH_1(X)$ to $0$ because $\frac{(\delta-1)(\delta-2)}{2}$ is odd. This motivates the question:

\begin{question}
If $\delta\equiv 0\pmod 4$ and $\delta\geq 8$, do there exist smooth hypersurfaces 
 of degree $\delta$ in $\P^4_{\R}$ containing no geometrically integral curve of even geometric genus~?
\end{question}

Although we do not know if any such hypersurface exists over $\R$, we are able to give examples over non-archimedean real closed fields.

\begin{example}\label{ex:cexhyp}
If $\delta\equiv 0\pmod 4$ and $\delta\geq 8$, we construct a smooth hypersurface $X$ of degree $\delta$  in $\P^4_R$ over the real closed field $R:=\cup_n\R((t^{1/n}))$ containing no geometrically irreducible curve of even geometric genus.

Write $\delta=2\varepsilon$ with $\varepsilon\geq 4$ even. Let $F_1:=x_0^{\varepsilon}+\dots+x_4^{\varepsilon}\in H^0(\P^4_\R,\sO(\varepsilon))$ be the Fermat equation, and choose $F_2\in H^0(\P^4_\R,\sO(\varepsilon))$ very general.
The surface $S:=\{F_1=F_2=0\}$ is smooth and has no $\R$-point. By the Noether--Lefschetz theorem \cite[Th\'eor\`eme 15.33]{voisinbook} applied to the threefold $\{F_1=0\}\subset\P^4_{\C}$, the geometric Picard group $\Pic(S_{\C})$ is generated by $\sO(1)$, hence so is $\Pic(S)$. Since a hyperplane section of $S$ has odd genus $\varepsilon^3-2\varepsilon^2+1$, \cite[Theorem \ref*{BW1-th:phi}]{BW1} shows that $S$ contains no geometrically irreducible curve of even geometric genus.

Fix a general $F\in H^0(\P^4_\C,\sO(\varepsilon))$ vanishing on $S_{\C}$ that is not defined over $\R$, define $G:=F\bar{F}\in H^0(\P^4_\R,\sO(\delta))$ to be the product with its complex conjugate, and let $X_0:=\{G=0\}\subset\P^4_{\R}$. A point $x\in X_0(\R)$ satisfies $F=0$ or $\bar{F}=0$, hence $F=\bar{F}=0$ because it is real. Since the pencil generated by $F$ and $\bar{F}$ coincides with the one generated by $F_1$ and $F_2$, we would have $x\in S(\R)$. This is a contradiction and shows that $X_0(\R)=\varnothing$. Note that $X_0$ contains no geometrically irreducible curve of even geometric genus:
as the two irreducible components of $(X_0)_\C$ are exchanged by Galois,
such a curve would lie on their intersection, hence on $S$.

Choose the equation $H\in H^0(\P^4_\R,\sO(\delta))$ of a smooth hypersurface, and let $X:=\{G+tH=0\}\subset \P^4_R$. 
The hypersurface $X$ is smooth because so is $\{H=0\}$ and has no $R$-points because $X_0(\R)=\varnothing$. Arguing as at the end of Example \ref{ex:cexCYsanspoint},
we see that it contains no geometrically irreducible curve of even geometric genus because neither does $X_0$.
\end{example}

\subsection{Conic bundles without real points}
\label{subsec:conic bundles non arch}
The main goal of this paragraph is the study, in  \S\ref{subsubsec:cokcycleclassmap}, of the image of $\cl:\CH_1(X)\to H^4_G(X(C),\Z(2))$ for conic bundles $X$ over surfaces over a real closed field $R$, when $X(R)=\varnothing$. It turns out to be strongly related to two classical properties, of independent interest, of algebraic varieties over $R$: the signs and EPT properties.
We first study them in \S\ref{subsubsec:signsEPT}.

\subsubsection{The signs and EPT properties}
\label{subsubsec:signsEPT}

\begin{defn}
Let $X$ be a smooth projective variety over a real closed field $R$. We say that
\begin{enumerate}[(i)]
\item $X$ satisfies the \textit{signs property} if for every $K\subset X(R)$ that is a union of connected components, there is a rational function $g$ on~$X$ that is invertible on $X(R)$, such that $g>0$ on $K$ and $g<0$ on $X(R)\setminus K$;
\item $X$ satisfies the \textit{EPT property} if for every line bundle $\mathcal{L}$ on $X$ with vanishing Borel--Haefliger class $\cl_R(\mathcal{L})\in H^1(X(R),\Z/2\Z)$, there exists a divisor $D$ on $X$ such that $\mathcal{L}\simeq\sO_X(D)$ and the support of $D$ contains no $R$-points.
\end{enumerate}
\end{defn}

These two properties are closely related, and complement each other. Together, they completely describe, for any open subset $U\subset X$, the possible distributions of signs $U(R)\to \Z/2\Z$ induced by rational functions $g$ on~$X$ invertible on $U(R)$.

The signs and EPT properties have been proven to hold when $X$ is a curve, by Witt \cite{wittreel} over the field $\R$ of real numbers and by Knebusch \cite{knebusch1} over general real closed fields. For an arbitrary $X$ over $\R$, the signs property is an immediate consequence of the Stone--Weierstrass
 approximation theorem, and the EPT property has been proven by Bröcker \cite{brocker}. Several proofs of Bröcker's EPT theorem have been given since then, and we refer to \cite[\S 4]{scheidererpurity} for one of them and an account of the literature.
At the end of \S\ref{subsubsec:cokcycleclassmap}, we will establish the validity of the signs and EPT properties for a large class of surfaces over an arbitrary real closed field that includes all geometrically rational surfaces:

\begin{thm}
\label{thm:signsEPTsurfaces}
Let $S$ be a smooth projective surface over $R$ such that $H^2(S,\sO_S)=\Pic(S_C)[2]=0$. Then $S$ satisfies the signs and EPT properties.
\end{thm}

Despite these positive results, the signs property does not hold in general. In \cite[Example 15.2.2]{bcr}, there is an example of a $K3$ surface $X$ over a non-archimedean real closed field that does not satisfy the signs property (more precisely, consider the minimal resolution of singularities of the surface of \emph{loc.\ cit.}).

It was an open problem (\cite[p.~309]{IS}, \cite[(4.2)]{scheidererpurity}) to decide whether the EPT property always holds. We construct the first example showing that it fails in general. Our example has two sources of inspiration:  the counterexample \cite[Example 15.2.2]{bcr} to the signs property mentioned above and, most importantly, Ducros' example \cite[\S 8]{ducros} discussed in Example \ref{ex:cexRC}.

\begin{prop}
\label{prop:cexEPT}
There exists a $K3$ surface $S$ over $R:=\cup_n\R((t^{1/n}))$ that does not satisfy the EPT property.
\end{prop}

\begin{proof}
Consider $\P^1_R\times \P^1_R$ with homogeneous coordinates $([x:y],[v:w])$.
Let $T$ be the double cover of $\P^1_R\times\P^1_R$, with branch locus of bidegree $(4,4)$, defined by the equation $z^2=x(y-x)(xw^2-yv^2)(xw^2+yv^2)-tw^4y^4$. Since $T$ has only rational double points as singularities, its minimal resolution of singularities $S$ is a $K3$ surface over~$R$. We denote by $i:T\to T$ the involution associated with the double cover.

The unique connected component $K$ of $S(R)$ contained in $0<x/y<1$ is semi-algebraically isomorphic to a sphere, hence satisfies $H^1(K,\Z/2\Z)=0$. The strict transform $\bar{D}$ of $D:=\{v=0\}$ in $S$ has real points only in $K$, so that its Borel--Haefliger class is trivial. Suppose for contradiction that $S$ satisfies the EPT property. Then $\bar{D}$ is linearly equivalent to a divisor whose support has no $R$-points. Pushing forward to $T$ shows that $D$ is linearly equivalent in $T$ to a Weil divisor whose support contains only finitely many $R$-points. 

Write $D=E_1-E_2+\divi(f)$, where $E_1$ and $E_2$ are effective Weil divisors whose supports contain finitely many $R$-points. The divisor $E_2+i(E_2)$ comes from $\P^1\times \P^1$, hence is the zero locus of a section in $H^0(T,\sO(d_1,d_2))$ for some $d_1, d_2\in\Z$. It follows that $E_1+i(E_2)$ is the zero locus of a section in $H^0(T,\sO(d_1,d_2+1))$. These two divisors have finitely many $R$-points in their support. Consequently, after possibly replacing $d_2$ by $d_2+1$, we have constructed $s\in H^0(T,\sO(d_1,d_2))$ with $d_2$ odd such that $\{s=0\}$ contains finitely many $R$-points.

Suppose there exists $a\in T(R)$ such that $s(a)\neq 0$, $i^*s(a)\neq 0$ but $s(a)+i^*s(a)=0$. Choose a semi-algebraic path $\gamma:[0,1]\to T(R)$ joining $a$ and a point $b\in T(R)$ in the ramification locus. Since $s$ vanishes at only finitely many $R$-points, it is possible to assume (after changing $b$) that neither $s$ nor $i^*s$ vanishes along the path. Since $b$ is a ramification point, $s(b)=i^*s(b)$. By hypothesis, $s(a)=-i^*s(a)$.
Since $[0,1]$ is semi-algebraically connected,
the continuous semi-algebraic map $\lambda\mapsto i^*s(\gamma(\lambda))/s(\gamma(\lambda))$ must reach the value $0$, a contradiction. We have proven that $s(a)+i^*s(a)=0$ implies that $s(a)=0$ or $i^*s(a)=0$, hence that $s+i^*s$ vanishes at only finitely many $R$-points of $T$.

Since $s+i^*s$ is $i^*$-invariant, it comes from a
section $s'\in H^0(\P^1_R\times\P^1_R,\sO(d_1,d_2))$ that vanishes at only finitely many
$R$\nobreakdash-points outside of the closed semi-algebraic subset
$\Theta \subset \P^1(R)\times \P^1(R)$
defined by the inequality
\begin{align}
\label{eq:theta inequality}
x(y-x)(xw^2-yv^2)(xw^2+yv^2)\leq tw^4y^4\rlap{.}
\end{align}
Replacing $B:=\{s'=0\}$ by an appropriate reduced irreducible component, we
may assume that $B$ is integral.

Let $A=\R[[t^{1/n}]]$, with~$n$ large enough that
$B=\sB\otimes_A R$ for some closed integral
subscheme $\sB \subset \P^1_A \times \P^1_A$.
Let~$\sB'$ be the normalisation of~$\sB$ and $B'=\sB'\otimes_AR$.
The map $\sB'\otimes_A \R \to \P^1_\R$ induced by the first projection
is generically finite of degree~$d_2$, which is odd;
therefore $\sB' \otimes_A \R$ possesses an irreducible component~$Y$ of odd multiplicity~$m$,
which dominates~$\P^1_\R$ with odd degree.
Letting $Y_\rs \subset \sB'_\rs \supset B'_\rs$ denote the real spectra of~$Y$, $\sB'$
and~$B'$,
the inequality~\eqref{eq:theta inequality} now defines a closed subset~$\Theta_A$ of~$\sB'_\rs$
which contains $B'_\rs$.
The completion of the local ring of~$\sB'$ at the generic point of~$Y$ is isomorphic,
over~$A$, to $\R(Y)[[u]]$, where $u=(\alpha t^{1/n})^{1/m}$ for some $\alpha \in \R(Y)^*$.
As~$m$ is odd,
any ordering of~$\R(Y)$ can be extended to $\R(Y)((u))$, compatibly with the given ordering on~$A$,
by declaring that $u$ is infinitely small of the same sign as~$\alpha$.
Thus, any point of~$Y_\rs$ above the generic point of~$Y$ belongs to the closure of~$B'_\rs$ in~$\sB'_\rs$
(see \cite[Proposition~7.1.21]{bcr}) and, hence, to~$\Theta_A$.
Letting~$Y'$ be the normalisation of~$Y$, it follows that the image of~$Y'_\rs$ in~$Y_\rs$
is contained in~$\Theta_A$.

We have now constructed a smooth, proper, irreducible curve~$Y'$ over~$\R$, endowed with two rational
functions $x,v \in \R(Y')^*$, such that the map $x:Y'\to \P^1_\R$ has odd degree and such that
the function $g=x(1-x)(x-v^2)(x+v^2)$ takes negative values on~$Y'(\R)$
wherever it is invertible.
There exists $P \in Y'(\R)$ at which~$x$ vanishes to odd order.
The function~$g$ then has a zero or pole of odd order at~$P$; its
sign therefore changes around~$P$, a contradiction.
\end{proof}

Proposition \ref{prop:cexEPT} and \cite[Example 15.2.2]{bcr} illustrate the importance of the requirement that $H^2(S,\sO_S)$ vanish in the statement of Theorem \ref{thm:signsEPTsurfaces}. We now provide an example showing that the hypothesis that $\Pic(S_C)[2]=0$ is also essential. We adapt the argument of \cite[Example 15.2.2]{bcr}.

\begin{prop}
\label{prop:cexsigns}
There exists a bielliptic surface $S$ over $R:=\cup_n\R((t^{1/n}))$ that does not satisfy the signs property. 
\end{prop}

\begin{proof}
Let $E$ be the elliptic curve with Weierstrass equation $y^2=x(x+1)(x-t)$ over~$R$. Consider the abelian surface $A=E^2$, with coordinates $(x_1,y_1)$ (resp.~$(x_2,y_2)$) on the first (resp.\ second) factor. The set of real points $A(R)$ has four semi-algebraic connected components. Let $K\subset A(R)$ be the one satisfying the inequalities $-1\leq x_1,x_2\leq 0$.
  Let $\Z/2\Z$ act on $A$ diagonally, by $-1$ on the first factor, and on the second factor, by translation by the $2$-torsion point $(x_2,y_2)=(t,0)$ in the same connected component of $E(R)$ as the identity. Let $S$ be the quotient of $A$ by this fixed-point-free action: it is a bielliptic surface over $R$. Since the quotient map $\pi:A\to S$ is finite \'etale, the image $\pi(K)$ of $K$ in $S(R)$ is a semi-algebraic connected component of $S(R)$. 

Suppose for contradiction that $S$ satisfies the signs property. Then there exists a non-zero rational function $g$ on $S$ that is invertible on $S(R)$ and such that $g\geq 0$ on $\pi(K)$ and $g\leq 0$ on $S(R)\setminus \pi(K)$. View $g$ as a rational function  on $A$, represent it as a rational function in $x_1,y_1,x_2,y_2$, and multiply it by an appropriate square
so that $g\in R[x_1,y_1,x_2,y_2]$. Define $h\in R[x_1,x_2]$ by the formula
 $h=g(x_1,y_1,x_2,y_2)+g(x_1,-y_1,x_2,y_2)+g(x_1,y_1,x_2,-y_2)+g(x_1,-y_1,x_2,-y_2)$.
Then $h$ is a non-zero polynomial that is non-negative when $-1\leq x_1,x_2\leq 0$ and non-positive when $x_1,x_2\geq t$, when $x_1\geq t$ and $-1\leq x_2\leq 0$, as well as when $x_2\geq t$ and $-1\leq x_1\leq 0$. Specializing $h$ when $t=0$ yields a non-zero polynomial $h_0\in\R[x_1,x_2]$ that is  non-negative when $-1\leq x_1,x_2\leq 0$ and non-positive in all other cases for which $-1\leq x_1,x_2$. The way the sign of $h_0$ changes when crossing the line $\{x_1=0\}$ shows that the multiplicity of $x_1$ as a factor of $h_0$ is both even and odd, a contradiction.
\end{proof}

Scheiderer has shown in \cite[Appendix]{scheidererpurity} that the EPT property is equivalent to a weakened formulation of it, and his arguments imply an analogous result for the signs property. We explain this in the following proposition. As in \cite{scheidererpurity}, we say that an irreducible variety $Z$ over $R$ is \textit{real} (resp.\ \textit{non-real}) if the Zariski closure of $Z(R)$ is equal to (resp.\ is a proper subset of) $Z$.

\begin{prop}
\label{prop:signsEPTfaibles}
Let $X$ be a smooth projective variety over $R$.
\begin{enumerate}[(i)]
\item  Suppose that for every union $K$ of connected components of $X(R)$, there is an invertible
function~$g$ on a dense
open subset $U \subseteq X$
such that $g>0$ on $U(R)\cap K$ and $g<0$ on $U(R) \setminus (U(R)\cap K)$.
Then $X$ satisfies the signs property.
\item Suppose that for every line bundle $\mathcal{L}$ on $X$ with vanishing Borel--Haefliger class $\cl_R(\mathcal{L})\in H^1(X(R),\Z/2\Z)$, there exists a divisor $D$ on $X$ all of whose real components have even multiplicity, and such that $\mathcal{L}=\sO_X(D)$. Then $X$ satisfies the EPT property.
\end{enumerate}
\end{prop}

\begin{proof}
Statement (ii) is due to Scheiderer \cite[Appendix, Proposition]{scheidererpurity}.

To show (i), we may assume that~$X$ is irreducible. We fix a union $K$ of connected components of $X(R)$ and an invertible function~$g$ on a dense open subset $U \subseteq X$
such that $g>0$ on $U(R) \cap K$ and $g<0$ on $U(R) \setminus (U(R) \cap K)$. Let $\Spec(A)\subset X$ be an affine open subset containing $X(R)$. 
 Since $g$ does not change its sign on connected components of $X(R)$, a real irreducible divisor can only appear with even multiplicity in $\divi(g)$. Scheiderer shows in \cite[Appendix, Proof of Lemma 2]{scheidererpurity} that for every irreducible divisor $y$ of $\Spec(A)$, there exists $h\in A$ that is a sum of squares in $R(X)$ such that the support of $2y-\divi(h)$ only contains non-real divisors. He also shows in \cite[Appendix, Proof of Lemma~1]{scheidererpurity} that for every non-real irreducible divisor $y$ of $\Spec(A)$, there exists $h\in A$ that is a sum of squares in $R(X)$ such that the support of $y-\divi(h)$ has no $R$-points. Combining these two facts, we see that there exists $h\in A$ that is a sum of squares in $R(X)$ such that the support of $\divi(g)-\divi(h)$ has no $R$-points. The rational function $g/h \in R(X)^*$ then shows that $X$ satisfies the signs property.
\end{proof}

\subsubsection{Torsion in the cokernel of the cycle class map}
\label{subsubsec:cokcycleclassmap}
We now complement the examples of \S \ref{par:cexnonarch} by studying some uniruled threefolds $X$ over $R$ with $X(R)=\emptyset$. 
Since such a variety may not satisfy $H^2(X,\sO_X)=0$, the definition of the real integral Hodge conjecture for $1$-cycles on $X$ of \cite[Definition~\ref*{BW1-def:ihc real closed field}]{BW1} does not apply. As a substitute, we will rather consider the statement, for a smooth proper
 variety $X$ over $R$ of pure dimension $d$, that the cokernel of $\cl:\CH_1(X)\to H^{2d-2}_G(X(C),\Z(d-1))_0$ has no torsion. If $R=\R$
or $H^2(X,\sO_X)=0$,
a norm argument based on the validity of the (rational, complex) Hodge conjecture for $1$\nobreakdash-cycles
shows that this statement is equivalent to the real integral Hodge conjecture for $1$-cycles on $X$.

\begin{prop}
\label{prop:cokersstorsion}
Let $f:X\to S$ be a morphism of smooth projective connected varieties over $R$ whose generic fiber is a  conic. Suppose that $S$ is a surface satisfying the signs and EPT properties, and that $X(R)=\emptyset$. Then the cokernel of $\cl:\CH_1(X)\to H^{4}_G(X(C),\Z(2))$ has no torsion.
\end{prop}

\begin{proof}
First, we may assume that $f$ is a Sarkisov standard model by applying Theorem \ref{thm:existenceSarkisov} and \cite[Proposition~\ref*{BW1-prop:birinvIHC}]{BW1}, noting that if $\mu:S'\to S$ is a birational morphism of smooth projective surfaces, writing $\mu$ as a composition of blow-ups at closed points shows that $S'$ satisfies the signs and EPT properties.

Let $\alpha\in H^{4}_G(X(C),\Z(2))$ be such that $N\alpha$ is algebraic for some integer $N\geq 1$. 
We wish to show that $\alpha$ is algebraic. The proof closely follows that of Theorem \ref{thm:fibres en coniques} given in \S\ref{subsec:proofconicbundles}, and we only explain the differences.

Step \ref{step:real locus1} is irrelevant since $X(R)=\varnothing$.

To perform Step \ref{step:pushforward}, we note that $f_*\alpha\in H^2_G(S(C),\Z(1))$ is algebraic by \cite[Proposition~\ref*{BW1-prop:real(1,1)nonarch}]{BW1} and the compatibility of $\cl$ with proper push-forwards \cite[\S \ref*{BW1-subsubsec:eqcl}]{BW1} and we replace 
 \cite[Satz 22]{WittHasse} with the more general \cite[\S 5]{elmanlamprestel}
and the use of Bröcker's theorem \cite{brocker} with the hypothesis that $S$ satisfies the EPT property.

Step \ref{step:real locus2} is unchanged.

In the proof of Lemma \ref{lem:AKhomology} used in Step \ref{step:notrace}, the use of \cite[Lemma 9 (1)]{AKhomology} can be circumvented by taking $\phi$ to be the identity, and the Stone--Weierstrass theorem can be replaced with the hypothesis that $S$ satisfies the signs property.

Steps \ref{step:killonopen} and \ref{step:removestuff} require no modification.
\end{proof}

The hypotheses about the signs and EPT properties are optimal, as shown by the following proposition:

\begin{prop}
\label{prop:IHCconicbundlesnonarch}
Let $S$ be a smooth projective surface over $R$, let $\Gamma$ be the anisotropic conic over $R$ and let $X:=S\times\Gamma$. The following assertions are equivalent:
\begin{enumerate}[(i)]
\item The cokernel of $\cl:\CH_1(X)\to H^4_G(X(C),\Z(2))$ has no torsion.
\item The surface $S$ satisfies the signs and EPT properties.
\end{enumerate}
\end{prop}

\begin{proof}
That (ii) implies (i) is a consequence of Proposition \ref{prop:cokersstorsion}.
To prove the converse, suppose that (i) holds. We consider Leray spectral sequences as in \S\ref{subsubsec:Leray}. Those associated with the $G$-equivariant sheaf $\Z(1)$ and to the morphisms $f:X\to S$ and $\Gamma\to\Spec(R)$ give rise to a commutative diagram with exact rows:
$$
\xymatrix@R=3ex{
H^0(G,\Z) \ar[r]\ar[d] & H^3(G,\Z(1)) \ar[d]\ar[r] &H^3_G(\Gamma(C),\Z(1))\ar[d]& \\
H^0_G(S(C),\Z) \ar[r] &H^3_G(S(C),\Z(1))\ar[r] &H^3_G(X(C),\Z(1))\rlap{.}
}
$$
Since $H^3_G(\Gamma(C),\Z(1))=0$ by cohomological dimension \cite[\S\ref*{BW1-subsubsec:cohomological dimension}]{BW1}, we deduce that the upper left horizontal map sends $1$ to the non-zero class of $H^3(G,\Z(1))$, hence that the lower left horizontal map sends~$1$ to $\omega^3$ (see \cite[\S\ref*{BW1-subsubsec:omega}]{BW1}).

In turn, the Leray spectral sequence for the equivariant sheaf $\Z(2)$ and the morphism $f:X\to S$ provides an exact sequence:
\begin{align}
\label{eq:Lerayconiqueconstante}
H^1_G(S(C),\Z(1))\stackrel{\omega^3}{\longrightarrow} H^4_G(S(C),\Z(2))&\xrightarrow{f^*} H^4_G(X(C),\Z(2)) \nonumber \\
  &\xrightarrow{f_*}  H^2_G(S(C),\Z(1))\stackrel{\omega^3}{\longrightarrow} H^5_G(S(C),\Z(2)),
\end{align}
where two of the horizontal maps are the cup product by $\omega^3\in H^3_G(S(C),\Z(1))$, by the above computation and the multiplicative properties of Leray spectral sequences.

Let $\sigma\in H^0(S(R),\Z/2\Z)$ be a collection of signs. By \cite[\S\ref*{BW1-subsubsec:cohomological dimension},  (\ref*{BW1-eq:canonical decomposition})]{BW1}, $H_G^5(S(C),\Z(3))\xrightarrow{\sim}H_G^5(S(R),\Z(3))\simeq H^0(S(R),\Z/2\Z)\oplus H^2(S(R),\Z/2\Z)$ and one may lift $(\sigma,0)$ to a class $\alpha\in H_G^5(S(C),\Z(3))$.
By the exact sequence \cite[(\ref*{BW1-eq:real-complex long 10})]{BW1} $$H^4(S(C),\Z)\to H^4_G(S(C),\Z(2))\xrightarrow{\omega} H^5_G(S(C),\Z(3))\to H^5(S(C),\Z)=0,$$
one may write $\alpha=\beta \smile\omega$ for some 
 class $\beta\in H^4_G(S(C),\Z(2))$ (see \cite[\S\ref*{BW1-subsubsec:omega}]{BW1}).
Since $\omega$ is $2$-torsion, $2\beta$ is the norm of a class in $H^4(S(C),\Z)$, hence is algebraic. So are $2f^*\beta$ and, by (i), $f^*\beta\in H^4_G(X(C),\Z(2))$. Let $z$ be a $1$-cycle on $X$ such that $f^*\beta=\cl(z)$ and $U$ be the complement in $S$ of the image by $f$ of the support of $z$.
The exact sequence analogous to (\ref{eq:Lerayconiqueconstante}) for
 the base change morphism $f_U:X_U\to U$
$$H^1_G(U(C),\Z(1))\stackrel{\omega^3}{\longrightarrow} H^4_G(U(C),\Z(2))\to H^4_G(X_U(C),\Z(2))$$
shows that there exists $\gamma\in H^1_G(U(C),\Z(1))$ such that $\gamma\smile\omega^3=\left.\beta\right|_U$.
Consider the reduction $\bar{\gamma}\in H^1_G(U(C),\Z/2\Z)=H^1_\et(U,\Z/2\Z)$ of $\gamma$ modulo $2$.
 Its image in $H^1_\et(U,\Gm)=\Pic(U)$ vanishes after shrinking $U$, so that we may assume that it lifts,
by the Kummer exact sequence, to $g\in H^0_\et(U,\Gm)$. By construction, the signs of~$g$ on $U(R)$ are given by $\sigma$, and $S$ has the signs property by Proposition~\ref{prop:signsEPTfaibles}~(i).

Let $\mathcal{L}\in\Pic(S)$ be such that $\cl_{R}(\mathcal{L})=0$. 
We deduce from  \cite[Theorem~\ref*{BW1-th:conditions de krasnov}]{BW1} that $\cl(\mathcal{L})|_{S(R)}=0\in H^2_G(S(R),\Z(1))$.
It follows that $\cl(\mathcal{L})\smile\omega^3=0$,
since $H^5_G(S(C),\Z(2))\xrightarrow{\sim}H^5_G(S(R),\Z(2))$ by \cite[\S\ref*{BW1-subsubsec:cohomological dimension}]{BW1}.
By (\ref{eq:Lerayconiqueconstante}),
we now deduce that $\cl(\mathcal{L})=f_*\alpha\in H^2_G(S(C),\Z(1))$ for some class $\alpha\in H^4_G(X(C),\Z(2))$.  
On the one hand, representing $\mathcal{L}$ by a linear combination of integral curves $C_i\subset S$ and considering an appropriate combination of integral curves $D_i\subset X$ such that $f(D_i)=C_i$ shows that there exists $z\in \CH_1(X)$ such that $\cl(f_*z)$ is a multiple of $\cl(\mathcal{L})$.
On the other hand, a norm argument shows that the double of every element in $H^4_G(S(C),\Z(2))$ is algebraic. Combining these two facts and (\ref{eq:Lerayconiqueconstante}) shows that some multiple of $\alpha$ is algebraic. 
 By (i), $\alpha=\cl(w)$ for some $1$-cycle $w$ on $X$. Let $B$ be an integral curve in the support of $w$ that is not contracted by $f$. Since $B(R)\subset X(R)=\emptyset$, if the morphism $B\to f(B)$ has odd degree, $f(B)(R)$ cannot meet the smooth locus of $f(B)$. We deduce that all the real components of $f_*w$ have even multiplicity. Consequently, $S$ has the EPT property by Proposition~\ref{prop:signsEPTfaibles}~(ii).
\end{proof}

\begin{example}[uniruled threefolds with no real point]
\label{ex:cexuniruled}
Combining Proposition~\ref{prop:IHCconicbundlesnonarch} with the counterexamples to the signs and EPT properties given in \cite[Example~15.2.2]{bcr}, in Proposition~\ref{prop:cexEPT} and in Proposition~\ref{prop:cexsigns} provides examples of smooth projective uniruled threefolds $X$ over the real closed field $R=\cup_n\R((t^{1/n}))$ such that the cokernel of the cycle class map $\cl:\CH_1(X)\to H^4_G(X(C),\Z(2))_0$ has non-trivial torsion and such that $X(R)=\emptyset$. According to \cite[Proposition \ref*{BW1-prop:ihc defect and bo}]{BW1}, these varieties have non-trivial third unramified cohomology group $H^3_{\nr}(X,\Q/\Z(2))$.

  If $S$ is the bielliptic surface of Proposition \ref{prop:cexsigns} and $\Gamma$ is the anisotropic conic over~$R$, the variety $X=S\times \Gamma$ satisfies $H^2(X,\sO_X)=0$. It then follows from the discussion at the beginning of \S\ref{subsubsec:cokcycleclassmap} that this uniruled threefold with no real point fails the real integral Hodge conjecture for $1$-cycles.
\end{example}

We may now give the

\begin{proof}[Proof of Theorem \ref{thm:signsEPTsurfaces}]
If $S(R)=\emptyset$, there is nothing to prove, so we assume that there exists a real point $s\in S(R)$. Define $X:=S\times\Gamma$, where $\Gamma$ is the anisotropic conic over $R$.
By Proposition \ref{prop:IHCconicbundlesnonarch} and since $H^2(X,\sO_X)=0$, it suffices to prove the real integral Hodge conjecture for $X$.  By \cite[Theorem \ref*{BW1-thm:relation ihc phi}]{BW1}, it suffices to show that $X$ contains a geometrically integral curve of even genus. But there is an obvious one: $\{s\}\times\Gamma\subset X$.
\end{proof}

\subsection{Cubic hypersurfaces}
\label{subsec:cubichypersurfaces}

Despite the above counterexamples, the real integral Hodge conjecture does hold for some particular families of varieties over arbitrary real closed fields. We illustrate this in the case of cubic hypersurfaces in Theorem~\ref{thm:cubiques} below. In particular, although cubic threefolds over $R$ are rationally connected threefolds with $R$-points, they do not give rise to examples analogous to Example~\ref{ex:cexRC}.

Note that, even over the field of real numbers, Theorem \ref{thm:cubiques} is stronger than what we have considered before, because it concerns only lines, and not general curves. It is this boundedness on the degree of the curves that allows the proof to go through over an arbitrary real closed field.

\begin{thm}
\label{thm:cubiques}
Let $X$ be a smooth cubic hypersurface of dimension $d\geq 3$ over $R$.
\begin{enumerate}[(i)]
\item The group $\CH_1(X)$ is generated by real lines and sums of two complex conjugate lines. If $R=\R$, real lines suffice.
\item The Borel--Haefliger classes of real lines on $X$ generate $H_1(X(R),\Z/2\Z)$.
\item The cohomology classes of real lines on $X$ generate $H^{2d-2}_G(X(C),\Z(d-1))_0$. 
\end{enumerate}
\end{thm}

\begin{proof} A smooth cubic surface $S$ which is a linear section of $X$ contains exactly $27$ lines over $C$. Since $27 $ is odd, at least one of these has to be $G$-invariant, hence is a real line $L\subset X$. The first part of (i) then follows from a theorem of Shen \cite[Theorem 1.7]{Shenrationality}.
To prove the second part of (i), we introduce the variety of lines $F_X$ of $X$: it is a smooth projective geometrically connected variety over $R$ \cite[Corollary 1.12, Theorem 1.16 (i)]{AK}, and we let $Z\subset X\times F_X$ be the universal line. The first part of (i) is equivalent to the surjectivity of $[Z]^*:\CH_0(F_X)\to\CH_1(X)$. If $R=\R$, Lemma \ref{lem:CH0pointsreels} below then implies that $\CH_1(X)$ is generated by the $[Z]^*x$ for $x\in F_X(\R)$, that is, by real lines on $X$.

Projecting from the real line $L$ shows that $X$ is birational to a conic bundle over~$\P^{d-1}_{R}$. If $R=\R$, we deduce from Theorem \ref{thm:fibres en coniques} and \cite[Propositions~\ref*{BW1-prop:birinvIHC} and \ref*{BW1-propIHCprojectivespace}]{BW1} that $X$ satisfies the real integral Hodge conjecture for $1$-cycles.
 Combining this fact with
(i), we see that if $X$ is a smooth cubic hypersurface of dimension $d$ over~$\R$, the group $H^{2d-2}_G(X(\C),\Z(d-1))_0$ is generated by classes of curves of degree $1$ on $X$. 
By Proposition~\ref{prop:ihcnonarchbounded} applied to the universal family of smooth cubic hypersurfaces of dimension $d$ over $\R$, the real integral Hodge conjecture holds for $1$\nobreakdash-cycles on smooth cubic hypersurfaces of dimension~$d$ over real closed field extensions of~$\R$, hence over arbitrary real closed fields by Proposition \ref{prop:invIHCextR} since any two real closed fields can be embedded simultaneously in a third one (see
\cite[Theorem~28]{conway} or \cite[Proposition~3.5.11, Examples~p.~197]{ChangKeisler}).

By \cite[Theorem~\ref*{BW1-th:phi} and Proposition~\ref*{BW1-prop:conoyau norme}]{BW1},
we deduce that the Borel--Haefliger cycle class map $\cl_R:\CH_1(X)\to H_1(X(R),\Z/2\Z)$ is surjective. The first part of (i) now implies that $H_1(X(R),\Z/2\Z)$ is generated by Borel--Haefliger classes of real lines and sums of complex conjugate lines. Since the Borel--Haefliger class of a sum of complex conjugate lines vanishes, we have shown (ii).

The cycle class $\cl_C(L') \in H^{2d-2}(X(C),\Z(d-1))$
of a line $L' \subset X_C$ is independent of~$L'$ by the Lefschetz hyperplane theorem.
Therefore so is the equivariant cycle class $\cl(L'+\bar L')=N_{C/R}(\cl_C(L')) \in H^{2d-2}_G(X(C),\Z(d-1))_0$.
It follows that $\cl(L'+\bar L')=2\cl(L)$.
Thus (i) and the real integral Hodge conjecture for $1$\nobreakdash-cycles on~$X$ imply~(iii).
\end{proof}

\begin{rmks}
(i)
Theorem \ref{thm:cubiques} (iii) and the second assertion of Theorem \ref{thm:cubiques} (i) fail for some smooth cubic surfaces $X$ over~$\R$. Indeed, $\Pic(X)=H^2_G(X(\C),\Z(1))$ by \cite[Proposition~\ref*{BW1-prop:real(1,1)}]{BW1} but this group may be of rank $4$ while $X$ contains only $3$ real lines \cite[Ch.~VI,~(5.4.4)]{silhol}. In contrast, a case by case analysis based on \cite[Ch.~VI,~(5.4)]{silhol} shows that Theorem~\ref{thm:cubiques}~(ii) still holds for smooth cubic surfaces over~$\R$. This argument works over an arbitrary real closed field $R$ as the classification \cite[Ch.~VI,~(5.4)]{silhol} of real cubic surfaces holds over such a field, with the same proof.
Finally, the first assertion of Theorem~\ref{thm:cubiques}~(i) also holds
for smooth cubic surfaces over~$R$, as follows from the validity of
Theorem~\ref{thm:cubiques}~(ii) for them, from
\cite[Proposition~\ref*{BW1-prop:conoyau norme}]{BW1}
and from the fact that $\Pic(X_C)$ is generated by lines.

(ii)
A direct approach to establishing Theorem~\ref{thm:cubiques}~(ii) over $R=\R$
through a topological classification of the real locus
of cubic hypersurfaces runs into the difficulty that
such a classification is only understood in dimension $\leq 4$ (see \cite{KharlaFin}).
Nevertheless,
in a forthcoming work \cite{slavasergei},
Finashin and Kharlamov provide a direct, purely topological proof of Theorem~\ref{thm:cubiques}~(ii) over $R=\R$.
\end{rmks}

\begin{lem}
\label{lem:CH0pointsreels}
Let $Y$ be a smooth projective connected variety over $\R$ such that $Y(\R)\neq\emptyset$. Then $\CH_0(Y)$ is generated by the classes of real points of $Y$.
\end{lem}

\begin{proof}
Looking at smooth curves in $Y$ going through an arbitrary closed point of~$Y$ and a fixed real point of $Y$, we are immediately reduced to the case where $Y$ is a curve. Then $\CH_0(Y)=\Pic(Y)$. Let $g$ be the genus of 
$Y$, $J:=\Pic^0(Y)$ be its jacobian, and $J(\R)^0$ be the connected component of the identity of $J(\R)$. We also choose a connected component $K\subset Y(\R)$.
Combining the exact sequence $0\to J(\R)\to \Pic(Y)\stackrel{\deg}{\longrightarrow} \Z\to 0$
and \cite[Chapter~IV, Corollary 4.2]{vanhamelthese} (in the statement of which $H^1$ should read $H^2$), we see that
$$J(\R)^0=\Ker\left(\cl:\Pic(Y)\to H^2_G(Y(\C),\Z(1))\right)\rlap{.}$$

By \cite[Lemma \ref*{BW1-lem:H2dreel}]{BW1}, $H^2_G(Y(\C),\Z(1))$ is generated by classes of real points on $Y$. Consequently, it suffices to show that every line bundle on $Y$ whose class belongs to $J(\R)^0$ is linearly equivalent to a linear combination of real points of $Y$. We introduce the subgroup $H\subset J(\R)^0$ of classes satisfying this property.

Look at the morphism $\phi:Y^{2g}\to J$ defined by $$\phi(P_1,\dots,P_g,Q_1,\dots,Q_g)=\sO_Y(P_1+\dots +P_g-Q_1-\dots-Q_g).$$
 It is dominant as is checked geometrically, hence generically smooth. We deduce that the induced map $K^{2g}\to J(\R)^0$ is open at some point, hence that $H$ contains a non-empty open subset of $ J(\R)^0$. By homogeneity, $H$ is a non-empty open subset of $J(\R)^0$.
As $J(\R)^0\setminus H$ is stable under translation by any element of~$H$,
it must also be open.  Hence~$H$ is open and closed in $J(\R)^0$, so that $H=J(\R)^0$, as desired.
\end{proof}

\begin{rmk}
In Theorem \ref{thm:cubiques}, we do not know whether the second part of (i) holds over arbitrary real closed fields. The proof does not work because Lemma \ref{lem:CH0pointsreels} fails in this more general setting. Here is an example.

Let $Y$ be a smooth projective curve of genus $2$ over $R:=\cup_n\R((t^{1/n}))$ such that $Y(R)\neq\emptyset$. We suppose that $Y$ is defined over $\R((t))$ and has a model over $\R[[t]]$ whose special fiber is a stable curve $Y_0$ of genus $2$ over $\R$ that, over~$\C$, consists of two genus one curves intersecting transversally in one point $P$ and exchanged by the action of~$G$.
Since $Y_0$ is of compact type, the jacobian $J$ of $Y$ has good reduction over $\R[[t]]$, with as special fiber the jacobian of $Y_0$. The $R$-points of $Y$ all specialize to the only real point $P$ of $Y_0$, so that the classes in $J(R)$ representing line bundles linearly equivalent to a sum of $R$-points of $Y$ all specialize to the origin in $J_0(\R)$. But all points of $J_0(\R)$ lift to $J(R)$ by Hensel's lemma, so that there exist line bundle classes in $J(R)$ not linearly equivalent to a sum of $R$-points of $Y$.
\end{rmk}

\bibliographystyle{myamsalpha}
\bibliography{hodgereel}
\end{document}